%% file: closedobservables.tex
\documentclass[final,leqno,onefignum,onetabnum]{article}

\usepackage[utf8]{inputenc}
\usepackage{amsmath}
\usepackage{amsfonts}
\usepackage{amssymb}
\usepackage{amsthm}
\usepackage[left=2cm,right=2cm,top=2cm,bottom=2cm]{geometry}

\usepackage{pgfplots}
\pgfplotsset{compat=newest}
\usetikzlibrary{plotmarks}
\usepackage{grffile}

\usepackage{tikz}
\usetikzlibrary{calc}
\usepackage[underline=false]{pgf-umlsd}
\usepgflibrary{decorations.shapes}
\usetikzlibrary{backgrounds,fit,decorations.pathreplacing,er,positioning, shapes, shapes.geometric, arrows, trees,decorations.markings,fit,er,shadows}
\tikzstyle{process} = [rectangle, minimum width=2cm, minimum height=0.7cm,text centered, draw=black, fill=white]
\tikzstyle{processlarge} = [rectangle, minimum width=2.6cm, minimum height=1.3cm,text centered, draw=black, fill=white]

\pgfplotsset{plot coordinates/math parser=false}
\newlength\fheight
\newlength\fwidth

\usepackage{framed}
\usepackage{float}
\usepackage{breakcites}
\usepackage{bmpsize}
\usepackage{graphicx}
\usepackage{epsfig}
\usepackage{subcaption}
\usepackage{placeins}
\graphicspath{{./fig/}}
\graphicspath{{fig/},{matlab/}}
\DeclareGraphicsExtensions{.png,.eps}

\usepackage{listings}
\newtheorem{theorem}{Theorem}

\newtheorem{lemma}{Lemma}

\newtheorem{remark}{Remark}

\title{Numerical Model Construction with Closed Observables} 

\author{Felix Dietrich\footnotemark[2]\ \footnotemark[3]\ \footnotemark[4]
\and Gerta K\"{o}ster\footnotemark[3]
\and Hans-Joachim Bungartz\footnotemark[4]}

\begin{document}

\maketitle

\renewcommand{\thefootnote}{\fnsymbol{footnote}}

\footnotetext[3]{Munich University of Applied Sciences, Lothstr. 64, 80335 Munich, Germany}
\footnotetext[4]{Technische Universit\"{a}t M\"{u}nchen, Boltzmannstr. 3, 85748 Garching, Germany}
\footnotetext[2]{{felix.dietrich@tum.de}. Questions, comments, or corrections
to this document may be directed to that email address. This work was partially funded by the German Federal Ministry of Education and Research through the project MultikOSi on assistance systems for urban events -- multi criteria integration for openness and safety (Grant No. 13N12824).}

\renewcommand{\thefootnote}{\arabic{footnote}}


\begin{abstract}
Performing analysis, optimization and control using simulations of many-particle systems is computationally demanding when no macroscopic model for the dynamics of the variables of interest is available. In case observations on the macroscopic scale can only be produced via legacy simulator code or live experiments, finding a model for these macroscopic variables is challenging.

In this paper, we employ time-lagged embedding theory to construct macroscopic numerical models from output data of a black box, such as a simulator or live experiments. Since the state space variables of the constructed, coarse model are dynamically closed and observable by an observation function, we call these variables \textit{closed observables}.
The approach is an online-offline procedure, as model construction from observation data is performed offline and the new model can then be used in an online phase, independent of the original.

We illustrate the theoretical findings with numerical models constructed from time series of a two-dimensional ordinary differential equation system, and from the density evolution of a transport-diffusion system.
Applicability is demonstrated in a real-world example, where passengers leave a train and the macroscopic model for the density flow onto the platform is constructed with our approach. If only the macroscopic variables are of interest, simulation runtimes with the numerical model are three orders of magnitude lower compared to simulations with the original fine scale model.
We conclude with a brief discussion of possibilities of numerical model construction in systematic upscaling, network optimization and uncertainty quantification.
\end{abstract}



\pagestyle{myheadings}
\thispagestyle{plain}
\markboth{F.~DIETRICH, G.~K\"{O}STER AND H.-J.~BUNGARTZ}{MODEL CONSTRUCTION WITH CLOSED OBSERVABLES}


\section{Introduction}

\subsection{Multiple scale systems}

Many-particle systems often exhibit dynamics on several temporal and spatial scales. 
These systems have many degrees of freedom but the interesting variables for applications (optimization, control, prediction and analysis) are often much fewer and change slowly with time.
In this case it is desirable to find a model for these slow, coarse variables.
A \textit{model} according to our definition consists of a state space $A$, a dynamic $f:A\to A$, and an observer function $y:A\to\mathbb{R}^m$ with $m\in\mathbb{N}$. The states $x_k\in A$ are generated by iterative application of the dynamic, starting with a given initial state $x_0$, so that
\begin{equation}\label{eq:iterative_model}
x_k=f^k(x_0).
\end{equation}
The initial states usually are contained in a subset $A_0\subset A$. Note that this definition includes continuous models with flow maps $F_t$, where we can choose a certain $\Delta t>0$ to get the iterative version (Eq. \ref{eq:iterative_model}) via $f=F_{\Delta t}$.

Consider for example a train station, where passengers get on and off trains and platforms. Such a system can be modeled at a fine temporal and spatial scale, where individuals can be distinguished and react to others in split-seconds. This accurate description might not be ideal for a safety officer controlling the number of persons on the platform. In that case, a coarse description changing every few seconds and on the whole platform hides unnecessary complexity and still provides all important information. At an even coarser scale, a building manager might want to know at which time of the day the station is most crowded.


\subsection{Extracting macroscopic dynamics}
Mathematically, slow-fast systems with well-separated scales can be handled with Fenichel and perturbation theory \cite{fenichel-1972}, as well as homogenization and averaging techniques \cite{givon-2004}.
If the microscopic description is only available as black-box legacy code, or for systems with complex descriptions at the fine scale, these analytic techniques cannot be applied or use too many additional assumptions for the transition to the macroscopic scale. In this case, numerical analysis is a good choice. In direct simulations, many-particle systems typically need large amounts of computing resources due to the fine scales and thus again complicate analyses on the coarse scale.
If the complex system cannot be simulated with all degrees of freedom, dimension and model reduction techniques are needed (both linear \cite{moore-1981,sirisup-2005} and nonlinear methods \cite{barrault-2004,givon-2004,lieberman-2010,e-2011,peherstorfer-2014} are available, also see references therein). The complete fine scale can be reconstructed at all times but is not used in its entirety for simulations. The equation-free framework \cite{kevrekidis-2009}, heterogeneous multi-scale methods \cite{engquist-2003} and other multi-scale analysis techniques \cite{e-2011} take a different approach and are designed with the premise that only the coarse variables are interesting. They assume the existence of a coarse system and use fine scale simulations to estimate the coarse dynamics.

The main problem underlying all methods above is that the coarse model is unknown. Regarding this problem, Takens and Ruelle started to develop the theory of time-lagged embedding between 1970 and 1980 \cite{ruelle-1971,takens-1981}, which enables the extraction of dynamical systems for coarse variables from scalar time series data. See \cite{deyle-2011} for the multivariate case, and \cite{monroig-2009} for extraction of input-output systems from observations. The main result of Takens and Ruelle states that observations  of a system provide enough information to reconstruct the dynamics on the attractors of the system. This is possible if several future (or past) observations are used as coordinates in a high-dimensional embedding.
Research following this result has focused on using the manifold in the high-dimensional space for analysis and time series prediction \cite{sauer-1994}, often using eigenvalue decomposition to compute a low-dimensional parametrization. This linear decomposition is also known as Singular Value Decomposition, Principal Component Analysis,  Karhunen–Lo\'{e}ve  Transform, and Proper Orthogonal Decomposition.
Recently, the nonlinear parametrization of the time-lagged manifold with diffusion maps \cite{coifman-2006} has been found to be optimal, in the sense that diffusion maps recover the most stable Oseledets subspace of the dynamical system \cite{berry-2013}. A re-parametrization with diffusion maps also separates the different time scales of the dynamic on the attractor, similar to a Fourier decomposition \cite{berry-2013}. 

The analyses of the time-lagged manifold mostly focused on properties and insights into the dynamics of the original system. The possibility to construct models directly from the manifold coordinates has been largely ignored.  
However, Takens and Ruelle state that the new variables on the time-lagged manifold, that is, the current state of a system, contain enough information to advance in time \cite{ruelle-1971,takens-1981}. This is the basic principle of a time-autonomous model. There is no need to completely describe all states for all time; The formulation of the model, including the procedure to advance by one step (discrete or infinitesimally), is enough to \textit{generate} the evolution over time.
Hence, the problem is to find variables that contain enough information to advance the system in time without needing more information. If there are such variables, this property is known under many names, for example \textit{closure} \cite{kevrekidis-2009}, \textit{Markov property} or \textit{memory-less} \cite{ethier-1986}, \textit{dependent} or \textit{explanatory variables}. Here, we use the term \textit{closure}.

\subsection{Numerical model construction}

In this paper, we show how the combination of time-lagged embedding, diffusion maps and an observer function can be used to construct coarse, time-autonomous models numerically (see Figure \ref{fig:problem}). The variables on the time-lagged manifold contain enough information to advance in time, hence the system is closed. The variables are also observable through an observer function that we create numerically, hence we call the variables \textit{closed observables}.

The constructed numerical models are one abstraction level above the equation-free computations. The modeling process starts with the fine scale, then possibly a model reduction step to be able to run the simulator for sufficiently long time periods. After this, the equation-free framework can be employed to construct the coarse observations over an even longer time period. Finally, these observations can be used with our approach to actually construct a numerical model.
It is not necessary to go through all steps to be able to construct the numerical model, any sequence of observations of a system over time can be used.
The key point is that the actual modeling of the original system must only occur in the first, fine scale step. Given enough accurate observations, all models on coarser scales can then be constructed automatically.

\begin{figure}[ht]
\centering
\setlength\fheight{0.45\textwidth}
\setlength\fwidth{0.45\textwidth}
\input{twosys.tex}
\caption{\label{fig:problem}A given model (left) with dynamic $F$ and current state $x_n$ has to be transformed into a new model with dynamic $G$ and state $\phi_n$. The dimension of the new state space should be much smaller than the dimension of the original state space, but it still has to produce approximately the same output. All solid arrows depict functions, the dashed arrow indicates that the new model is created using the output of the original model.}
\end{figure}
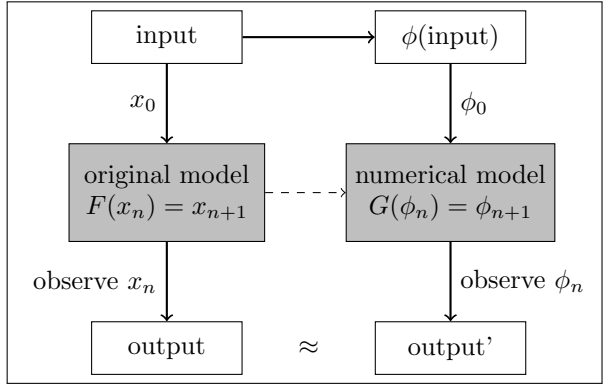

We assume the simulation of the complex system is possible but very costly. We also assume that the interesting output \textit{changes with time}. This dynamic aspect is a distinction from surrogate models \cite{bandler-2004}, which are commonly used to construct a numerical model for the response to a certain input. Surrogate models are not trying to approximate the dynamics but the mapping from input to output of a system. This mapping might need a dynamic process internally, but the resulting surrogate model ignores the process.

Our paper addresses the issue of numerically constructing a model for a \textit{dynamic} process. We state conditions when the model needs less storage than the original observations (Theorem \ref{thm:storagereduction}). The model approximates the output of the original via an observer function. We do not modify the original model equations, but directly work with the output data.
For models with differential equations, we employ the flow map with positive time delta to generate a discrete system.

The construction is different from previous approaches with diffusion maps \cite{coifman-2008}, where a map back into the high-dimensional space is always necessary.
In real applications, the output space is never sampled completely, and interpolation or approximation of the available values are necessary. Also, the original model must be queried many times to be able to construct a numerical model with sufficient accuracy. This is discussed in the numerical part of the paper in \S\ref{sec:algorithms}, together with an error analysis.

\subsection{Real-world example}\label{sec:intro_example}
We use the example shown in Figure \ref{fig:train station} to motivate the applicability of our approach. Section \ref{sec:real world example} covers this example in detail. A train enters a train station where passengers are already waiting on the platform. In this scenario, it is important to know the number of passengers over time both inside and outside the train, because very dense situations or a sudden change of density can be dangerous.
We use the recently proposed Gradient Navigation Model \cite{dietrich-2014} to describe and simulate the interactions between passengers and the geometry on the fine scale. Related to our general problem (Figure \ref{fig:problem}), the state space $A$ contains elements $x$ which hold all positions, velocities and parameters of all passengers. The function $F$ is the simulator implementing the Gradient Navigation Model. We observe the two numbers, persons on the train and on the platform, with an observer function $y$ on the current state $x$, so that $y(x)=(N_T,N_P)\in\mathbb{N}_0^2$.
\begin{figure}[ht]
\centering
\includegraphics[width=0.45\textwidth]{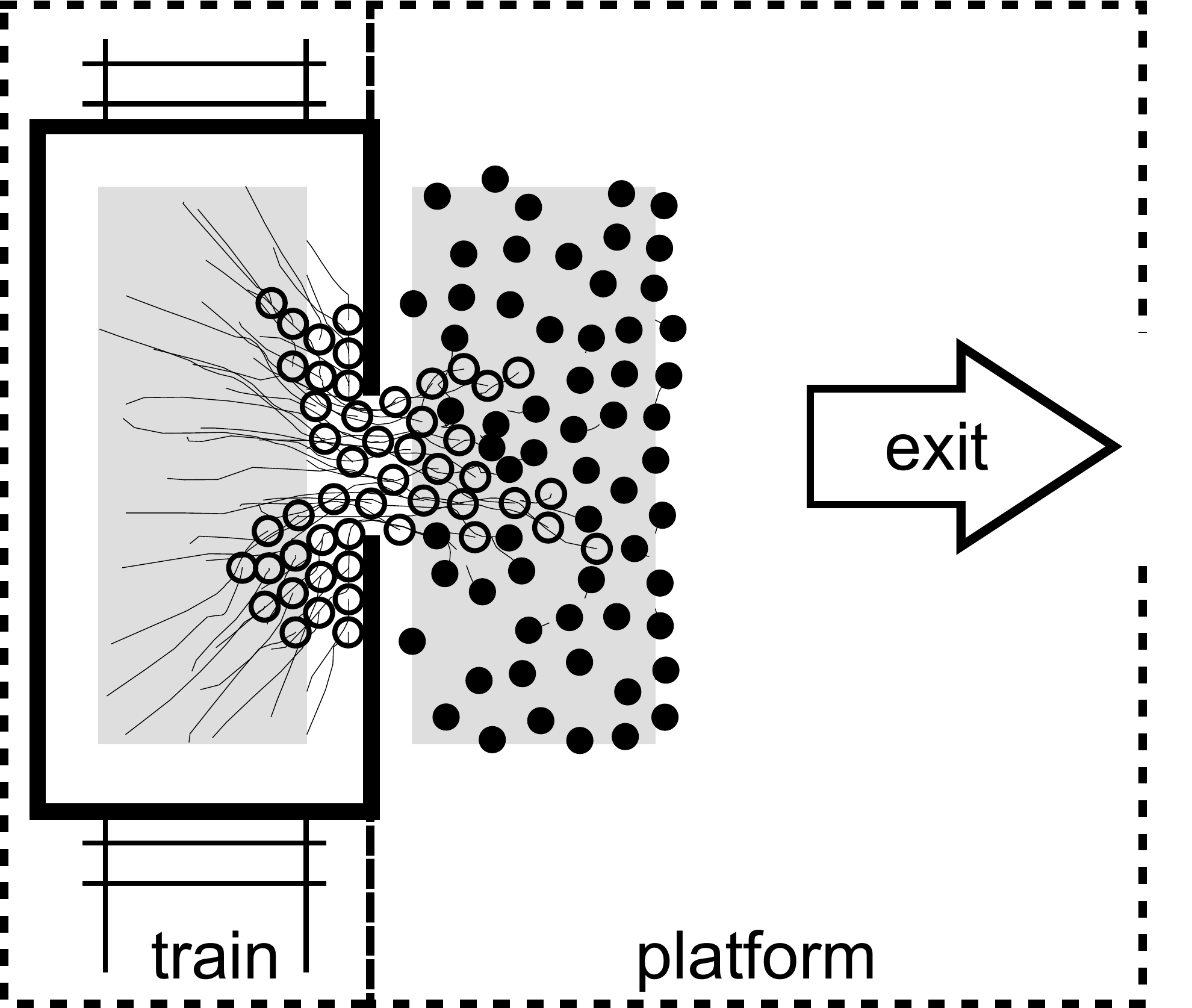}
\caption{\label{fig:train station}Setup for the real-world example. A train (left) enters the train station and passengers try to leave through the crowd of waiting passengers. Both the number of passengers inside and outside of the train are modeled using closed observables.}
\end{figure} 
The passengers can be on the train or on the platform, but nowhere else. With closed observables, the resulting numerical model automatically extracts the underlying dynamics of the number of persons on both the train and the platform. In the notation from Figure \ref{fig:problem}: we construct a two-dimensional, dynamical system where
\begin{equation}
\phi_{n+1}=G(\phi_n),\ \phi_n\in\mathbb{R}^2,\ \phi(x_0)=\phi_0
\end{equation}
and with an observer function $\tilde{y}$ so that $\tilde{y}(\phi_n)\approx y(x_n)$.
After the construction of the numerical model, we use it to analyze the chance of passengers to get out of the train within 25 and 50 seconds, given varying initial numbers of waiting passengers on the platform. The analysis, including all simulations with the constructed model, completes in seconds, whereas the same analysis with the original model would take several hours.
The microscopic, individual scale and the macroscopic, platform-wide scale in this scenario differ by $\mathcal{O}(10^2)$ both temporally and spatially (see \S\ref{sec:real world example}). This means we can use closed observables to capture a system that is about 100 times slower and larger than the microscopic system that creates the output.

\subsection{Outline}
We introduce the mathematical framework of Takens' theorem and diffusion maps in \S\ref{sec:mathematical framework}. Next, numerical model construction is presented and two main theorems regarding approximation and storage efficiency are stated and proved. In \S\ref{sec:algorithms}, the theoretical construction of the new model is complemented with numerical algorithms and an error analysis. Two examples in \S\ref{sec:examples} illustrate the ideas of the construction, the first is a two-dimensional system of ordinary differential equations and the second is a transport-diffusion system where the output is a whole function, the density over the whole domain. Section \ref{sec:real world example} describes the train station example.
We conclude in \S\ref{sec:conclusion} and add remarks on how the construction of the new model can be improved further, and how it can be used in other applications and numerical analysis.

\FloatBarrier
\section{Mathematical framework}\label{sec:mathematical framework}
The following explanations are in the context of dynamical systems theory and time series analysis. They build the theoretical basis of the construction of numerical models.
Consider a time-autonomous, discrete dynamical system with state $x$ and continuous map $f$ of the form \begin{equation}\label{eq:original system}
x_{k+1}=f(x_k),
\end{equation}
which yields the sequence of states $\{x_k\}$. Let $M$ be the attractor of the system, that is the set of points where all states tend to as $k\to\infty$, independent of the initial state.
Let the system (\ref{eq:original system}) have an observer function $y$ of the current state $x_k$, which for iterative application of the flow map $f$ yields a sequence of observations $\{y_k\}$ where $y_k:=y(x_k)=y(f^kx_0)$.

\subsection{Time-delay embedding and Takens' theorem}
The observation space consisting of all $y_k$ does not directly represent dynamic features of the underlying system. This is because the space might not be closed dynamically: a single observation might not contain enough information to predict the next one. Takens' theorem and the resulting method of delays provide the means to construct such a dynamically closed space, which is diffeomorphic to the attractor manifold $M$. Let 
\begin{equation}
\textbf{h}_k=(y_k,y_{k+1},\dots,y_{k+T-1})
\end{equation}
be a vector of $T$ observations, then for sufficiently large $T$ Takens' theorem states the existence of a function $G$ with \begin{equation}\label{eq:dynamic embedding}\textbf{h}_{k+1}=G(\textbf{h}_k).\end{equation}
Note that instead of past observations, we incorporate future observations in the current state, which will be important when constructing a numerical model later. 
Also see \cite{small-2004} for a short review on how to compute $T$. Takens proved the theorem with equal time sampling. See \cite{huke-2007} for an approach where the time sampling must not be equal.
In the original form of the theorem, $y_k$ must be scalar and must also be generated without observational noise and without stochastic effects in $f$. The multivariate case was recently described by Deyle and Sugihara \cite{deyle-2011}, and the stochastic version was described by Stark et al. \cite{stark-1997}.
The existence of the dynamic $G$ in Eq. \ref{eq:dynamic embedding} proofs that the dynamics on attractors of the system in Eq. \ref{eq:original system} can be reconstructed only from observations $\{y_k\}$ of the system.

\subsection{Dimension reduction and observer function}

The states $\textbf{h}_k$ are usually high-dimensional, since many $y_k$ must be considered to close the system ($T$ is large). Dimension reduction can be performed linearly by eigenvalue decomposition or non-linearly via kernel methods. Berry et al. showed that diffusion maps \cite{coifman-2006} are the natural choice for nonlinear dimension reduction of the time-lagged manifold, since they recover the restriction of the Lyapunov metric to the most stable Oseledet subspace  \cite{berry-2013}. Also, time-scale properties of the underlying system (Eq. \ref{eq:original system}) are separated in the reconstructed space, similar to a Fourier decomposition. The slowest dynamics can thus be separated from fast dynamics and noise, and a selection of temporal scales is possible when reconstructing the system.

Generally, the dimension reduction of the time-lagged manifold is performed via a coordinate transform from the space of the high-dimensional $\textbf{h}_k$ into a  lower-dimensional Euclidean space.
Let $\phi$ be the coordinate transform of $\textbf{h}_k$ into coordinates $\phi_k$, so that $\phi_k:=\phi(\textbf{h}_k)$. Eq. \ref{eq:dynamic embedding} can now be rewritten with a dynamic $\tilde{G}$ similar to $G$ but acting on the transformed values $\phi_k$:
\begin{equation}\label{eq:dynamic phi}
\phi_{k+1}=\tilde{G}(\phi_k).
\end{equation}
Eq. \ref{eq:dynamic phi} is again a time-autonomous, discrete dynamical system. An observer function $\tilde{y}$ can be constructed with the implicit definition 
\begin{equation}
\tilde{y}(\phi_k)=\tilde{y}(\phi(\textbf{h}_k))=\tilde{y}(\phi(y_k,y_{k+1},\dots,y_{k+T-1}))=y_k.
\end{equation}
The observer function $\tilde{y}$ is well defined if the coordinate transformation $\phi$ has an inverse $\phi^{-1}$, which ensures the existence of $\tilde{y}=(y\circ\phi^{-1})_1$. Since the time-lagged manifold is diffeomorphic to the attractor manifold \cite{berry-2013}, this is always possible theoretically. In the numerical computation of the coordinate transformation with diffusion maps, this inverse is preserved for the precomputed points and can hence be also reconstructed by interpolation.

\FloatBarrier
\section{Numerical model construction}\label{sec:numerical model construction}

Based on the framework of Takens' theorem and nonlinear dimension reduction it is possible to construct numerical, macroscopic models from output data of simulations or experiments.
A \textit{numerical} model consists of all the objects of a model (state space, dynamic, observer), and all objects are created numerically via interpolation or approximation.
The basic idea for the numerical models discussed here is to store the dynamics and the observer function as precomputed points and then interpolate them. In Theorem \ref{thm:storagereduction}, we will state conditions when the construction has much lower storage requirements regarding interpolation points compared to simply storing the observed values of the original and interpolating those. The following objects have to be precomputed:
\begin{itemize}
\item Coordinate transformation $\phi$ from given input $x_0$ to closed observables $\phi_0$
\item Dynamics of closed observables, in the form $G(\phi_n)=\phi_{n+1}$ or $\Delta G(\phi_n)=\phi_{n+1}-\phi_{n}$
\item Observer function $\tilde{y}$, from $\phi_n$ to the observation of the original observer function
\end{itemize}
With these three objects, the numerical model is independent of evaluations of the original model $f$ (see \S\ref{sec:mathematical framework}). 
We will now show how to construct and store the interpolants for the numerical models. We also perform an error analysis of the numerical modeling process.

\subsection{Storing the interpolants for the numerical model}
The interpolants used by the numerical model require the storage of many data points. However, in many cases, the numerical model needs much fewer points for the same accuracy compared to the storage of all output. This is formalized in Theorem \ref{thm:storagereduction} and is one of the main motivations to construct the new model instead of simply storing output and looking up the values when needed. Also, the theorem states a condition that hints towards new scientific insight: if the condition is true and the numerical model is more efficient than naive storage of observations, then the input variables have a dynamical connection, which might not be obvious and might lead to new understanding of the given process.

We formalize the storage requirements of an interpolant as follows.
Consider a function $f\in C^k([0,1]^p,\mathbb{R}^m)$ with $k\in\mathbb{N}_0$, $p,m\in\mathbb{N}$. Let $\epsilon>0$ and consider an interpolant $f_\epsilon$ of $f$, so that
\begin{equation}
\|f_\epsilon-f\|_\infty<\epsilon.
\end{equation}
If $f$ is a black box (also sometimes called oracle) and can hence only be queried at distinct points $x_i\in[0,1]^p$ to yield the values $f(x_i)$, the interpolant is constructed by sampling the space $[0,1]^p$ with a sampling method $S$. This method $S$ needs $S(\epsilon,p)$ sampling points in the construction of the interpolant $f_\epsilon$. For the  full-grid sampling method, $S(\epsilon,p)$ increases exponentially in $p$, since if $N$ different samples of one coordinate axis are considered, $N^p$ points must be stored for the $p$-dimensional space. If the function $f$ has regularity $k>0$, the error between $f_\epsilon$ and $f$ decreases with the order $\mathcal{O}(N^{-k/p})$. More advanced sampling techniques such as sparse grids \cite{bungartz-2004} exist, where the number of points for a comparable accuracy increases with $\mathcal{O}(N(\log N)^{p-1})$, mitigating the curse of dimensionality.
For all sampling techniques, decreasing $p$ without changing the smoothness properties of the function is a good way to reduce the number of points that need to be stored for $f_\epsilon$ in order to achieve an accuracy smaller than $\epsilon$.

In the setting of this paper, we try to approximate the output of a dynamical system (Eq. \ref{eq:original system}). This means we try to find an interpolant $g_\epsilon$ for the function
\begin{equation}
g:\mathbb{N}\times A_0\to\mathbb{R}^m,\ g(k,x_0)=y(f^kx_0)\approx g_\epsilon(k,x_0).
\end{equation}
If $A_0$ is $p$-dimensional, the function $g$ has $p+1$ parameters. Hence, an interpolant $g_\epsilon$ will need $S(\epsilon,p+1)$ points for an approximation accuracy of at least $\epsilon$. 
With closed observables, we only need $S(\epsilon,p)$ points for the same accuracy, which will be shown via Theorem \ref{thm:storagereduction}.

\begin{theorem}\textbf{Storage reduction}\label{thm:storagereduction}
Define $n_0:=\dim(A_0)$, and $d:=\dim(\phi(A))$. Given a sampling method $S$ and an accuracy $\epsilon>0$, the interpolants $\phi_I$, $G_I$ and $\tilde{y}_I$ need at most $\mathcal{O}(S(\epsilon,n_0))$ points if the following condition is true:
\begin{equation}
d < n_0+1.
\end{equation}
In contrast to that, storing the output of the original model needs $\mathcal{O}(S(\epsilon,n_0+1))$ points, where the additional dimension is approximating the time variable.
\end{theorem}
\begin{remark}
The dimension reduction from $p+1$ to $p$ variables might not seem significant. However, the reduction is for the time variable, which often is discretized much finer than the initial conditions. The gain is very large for problems which depend smoothly on the initial conditions, but need long simulation runs in time. In this case, the coordinate transformation $\phi$ can be interpolated with a method with low storage requirements and the time variable can be disregarded due to our numerical model construction. There are many applications where this is possible, see example \ref{sec:real world example} where the two-dimensional initial conditions are discretized with $5$ and $11$ points per coordinate axis, but the time is discretized with $50$ points (one for each second). When using a full grid in example \ref{sec:examples2}, the number of necessary points is reduced by $10^5$ for a discretization error of $\mathcal{O}(10^{-5})$.
\end{remark}
Theorem \ref{thm:storagereduction} concerns storage efficiency compared to the storage of all model output. Remarkably, the proof is independent of any special properties of closed observables, and hence the theorem applies to all models constructed by storing and interpolating data.

\begin{proof}\textbf{Theorem \ref{thm:storagereduction}, Storage reduction}
Additionally, define $m:=\dim(y(A))$.
\begin{enumerate}
\item For the $n_0$-dimensional input, we need exactly $d$ surfaces with dimension $n_0$ each to store the mapping from input to the new variables.
\item We also need $d$ surfaces with dimension $d$ each to store the dynamic $G$ on the space $D$.
\item Finally, we need $m$ surfaces with dimension $d$ each to store the observer on the new space.
\item Since storing a finite number of surfaces does not add another dimension, the maximum dimensionality we need to store the new model is $\max(d,n_0)$.
\item To store the observed values over time, we need $m$ surfaces ($n_0$-dimensional) for each iteration step $n$, so in total we need $n_0+1$ dimensions.
\item The exact number of hyper-surfaces including dimension for the new model is given by $d\cdot \underline{d}+m\cdot\underline{d}+d\cdot \underline{n_0}$, where the underlined values stand for\textit{dimensions} and the multiplication between a number and a dimension is noted by $(\cdot)$. Note that the dimension is not distributive, that is $a\cdot(\underline{b}+\underline{c})\neq a\cdot \underline{b}+a\cdot \underline{c}$. The dimensions are additive, so $\underline{a}+\underline{b}=\underline{a+b}$. This allows for a reformulation into
\begin{equation}
\text{storage(new model)}=(d+m)\cdot \underline{d}+d\cdot \underline{n_0}
\end{equation}
\item The exact number of hyper-surfaces including dimension for the storage of all observations is given by $m\cdot\underline{n_0+1}$. Reformulation yields
\begin{equation}
\text{storage(full output)}= m\cdot\underline{n_0+1}
\end{equation}
\item The theorem follows from comparing and reducing the storage formulations in items 7 and 8.
\end{enumerate}
\qquad\end{proof}

The storage efficiency (Theorem \ref{thm:storagereduction}) is relevant since a lot of output needs to be stored for the numerical model to be accurate. The theorem states the conditions when it is highly advantageous to construct closed observables instead of simply storing the output. The difference can be as large as the difference between the memory of a smartphone and a supercomputer (five orders of magnitude, also see example \ref{sec:examples2} below).

\FloatBarrier
\subsection{Algorithmic construction of the interpolants}\label{sec:algorithms}

In this section, we describe the algorithmic construction of the coordinate transformation, dynamics and observer function of the numerical model. A brief error analysis follows, where we verify numerically the importance of the interpolation error stated by Theorem \ref{thm:closure of observer}.

We construct the new model with the following four steps. The detailed algorithm is shown in Figure \ref{lst:algorithm}.
\begin{enumerate}
\item Construct a set $M$ from observations of states over time.
\item Embed the observations of $M$ on a set $H$, using the current and also future observations to define coordinates of points on $H$.
\item Re-parametrize $H$ with diffusion maps to form the coordinates of the closed observables.
\item Define the coordinate transformation, dynamic and the observer of the new model on the closed observables. In this step, approximation and interpolation might be necessary.
\end{enumerate}
After the construction steps, the numerical model can be used for simulation and analysis, independent of the original system. A simulation with the numerical model can be performed in the following steps, given an initial state $x_0$:
\begin{enumerate}
\item Compute initial states $\phi_0=\phi(x_0)$ of the closed observables.
\item Iterate the given observables with the constructed dynamic, so that $\phi_{n+1}=\phi_n+\Delta G(\phi_n).$
\item Observe the states $\phi_n$ with the constructed observer function $\tilde{y}$ to generate observations of the original system:
\begin{equation}
\tilde{y}(\phi_n)\approx y(x_n)=y_n.
\end{equation}
\end{enumerate}
The simulation with the numerical model is independent of the original system (Eq. \ref{eq:original system}), because it does not need evaluations of the dynamic $f$ or the observer $y$. The independence of the two systems enables a precomputation of the numerical model on a high performance system or with a large array of experiments, which is then condensed into the numerical model that can be used on a much less powerful device such as a smartphone. 

\begin{figure}[ht]
\begin{framed}
\noindent The following steps need to be completed consecutively in order to construct the numerical model:
\begin{enumerate}
\item Construct an approximation of the attractor manifold $M$ (see \S\ref{sec:mathematical framework}):
\begin{itemize}
\item Decompose the initial space $A_0$ into sample points $x_{1,0},x_{2,0},\dots$
\item Run the simulator or experiment up to time $t_{end}$, starting at all sampled initial points $x_{i,0}$
\item $M$ is approximated by the observations of all states generated in the simulation or experiment.
\end{itemize} 
\item Construct the time-lagged manifold $H$ with points $\textbf{h}_k$ so that for all observations on $M$ created via observing a state $x_k$ with the observer function $y$, 
$$\textbf{h}_k=(y(x_k),y(fx_k),y(f^2x_k),\dots,y(f^{T-1}x_k))$$
\item Compute pair distances $d$ between all points $\textbf{h}_i,\textbf{h}_j\in H$
$$d(i,j)=\|\textbf{h}_i-\textbf{h}_j\|_H$$
\item Construct the diffusion map via the algorithm in \cite{singer-2009}:
\begin{itemize}
\item Transition matrix $W_{i,j}=\exp(-d(i,j)^2/\epsilon^2)$
\item Normalization matrix $N_{ii}=\sum_j W(i,j)$, so that $N$ is a diagonal matrix containing the row sums of $W$
\item Solving the eigensystems $N^{-1/2}WN^{-1/2}\textbf{v}_k=\lambda_k\textbf{v}_k$ for all $k$, then sort $\lambda_k$ by absolute value
\item The diffusion map coordinates are $u_k=\lambda_kN^{-1/2}\textbf{v}_k$, so that the diffusion map is
$$\Psi:\textbf{h}_i\mapsto (u_2,u_3,\dots).$$
The first eigenvector $u_1$ only contains $1$ and is omitted. We also truncate the map for small $\lambda_k$ and remove all $u_k$ which only depend on previous eigenvectors. The resulting coordinates $\phi_n=\Psi(\textbf{h}_n)$ are the closed observables.
\end{itemize}
\item The objects for the numerical model can be generated via interpolation, we use $I[\cdot]$ as a generic interpolation of the given arguments.
\begin{itemize}
\item The coordinate transform $\phi$ for the inputs $x_{i,0}$ is an interpolant of the form
$$\phi(x_{i,0})=I[\phi_{i,0}],$$
where $\phi_{i,0}=\Psi(\textbf{h}_{i,0})=\Psi(y(x_{i,0}),\dots,y(f^{T-1}x_{i,0}))$.
\item The dynamic $\Delta G$ is an interpolation of the difference between subsequent values of $\phi_n$ in time:
$$\Delta G(\phi_n)=I[\phi_{n+1}-\phi_n],$$
where the existence of $\Delta G$ is ensured by Takens' theorem (see \S\ref{sec:mathematical framework}). Note that an interpolation of the difference might take other values such as $\phi_{n-1}$ into account to be more accurate.
\item The observer function $\tilde{y}$ of the closed observables is an interpolant of the observations in the original system:
$$\tilde{y}(\phi_n)=\tilde{y}(\phi(\textbf{h}_n))=\tilde{y}(\phi(y_n,y_{n+1},\dots,y_{n+T-1}))=I[y_n]$$
\end{itemize}
\end{enumerate}
\end{framed}
\caption{\label{lst:algorithm}Algorithm to construct the objects of the numerical model  (input mapping, dynamic and observer) from given output data.}
\end{figure}

We denote interpolated versions of the objects with a capital $I$ as a subscript, and the error between the interpolant and the actual function by $E_{\phi,G,\tilde{y}}$, respectively, so that
\begin{itemize}
\item the coordinate transformation $\phi$ is interpolated by $\phi_I$ with maximum error $E_\phi:=\max_{a\in A}(\phi_I(a)-\phi(a))\in D$,
\item the dynamics $G$ is interpolated by $G_I$ with maximum error $E_G:=\max_{\phi\in D}(G_I(\phi)-G(\phi))\in D$, and
\item the observer $\tilde{y}$ is interpolated by $\tilde{y}_I$ with maximum error $E_{\tilde{y}}:=\max_{\phi\in D}(\tilde{y}_I(\phi)-\tilde{y}(\phi))\in \mathbb{R}^m$.
\end{itemize}
All interpolants are constructed from discrete observations and diffusion map coordinates. The numerical model with these interpolants approximates the observed values of the original model. The approximation error only depends on the interpolation errors $E_\phi$, $E_G$ and $E_{\tilde{y}}$, as stated by Theorem \ref{thm:closure of observer} and proofed below.

\subsection{Error analysis}
We now analyze the different types of errors that occur during construction and simulation of the numerical model. Theorem \ref{thm:closure of observer} ensures that the approximation of the original model by the numerical model only depends on the interpolation error. When generating more output with the original model, the numerical model will usually be more accurate.
Note that the conditions stated in Theorem \ref{thm:closure of observer} are rarely all fulfilled in real applications. However, the construction process seems to be quite robust, which is mostly due to the robustness of diffusion maps regarding noise (see example \ref{sec:real world example} and the discussion in \cite{coifman-2006,berry-2013}).

\begin{theorem}\label{thm:closure of observer}\textbf{Correct output of the new model}
Consider the state space $A$ of the model in Eq. \ref{eq:original system} consisting of all possible states $x$, and the initial states $A_0\subseteq A$. Consider the dynamic $f:A\to A$ and the observer $y:A\to\mathbb{R}^m$. Let the objects $\phi\in C^2(A,D)$ (coordinate transformation into closed observable space $D$), $G\in C^2(D,D)$ (dynamic) and $\tilde{y}\in C^2(D,\mathbb{R}^m)$ (observer) be constructed numerically by the interpolation functions $\phi_I\in C^2(A,D)$, $G_I\in C^2(D,D)$ and $\tilde{y}_I\in C^2(D,\mathbb{R}^m)$. Let the interpolants have at most errors $E_I$, $E_G$ and $E_{\tilde{y}}$ on the respective domains. The errors are small, so that their norms are much smaller than one. The dynamic $G$ must have bounded, small derivatives so that for a constant $M\in(0,1)$, $\sup_\phi\|DG(\phi)\|\leq M$. Then, for any given initial state $a_0\in A_0$ and any fixed number of iterations $n\in\mathbb{N}$,
\begin{equation}\label{eq:thm closure of observer}
\|y\left(f^na_0\right)-\tilde{y}_I\left(G_I^n\phi_I(a_0)\right)\|\leq C_1\left(\frac{1-M^{n+1}}{1-M}\right)\|E_G\|+C_2 M^{n}\|E_\phi\|+C_3\|E_{\tilde{y}}\|.
\end{equation}
The constants $C_1,C_2$ and $C_3$ are positive and independent of $n$ and $a_0$.
\end{theorem}

\begin{remark}
The norms of the errors of the interpolants $E_\phi$, $E_G$ and $E_{\tilde{y}}$ are placeholders for the actual error estimates, which differ greatly for different interpolation methods. As an example, consider a full-grid discretization with cell size $h$ for all interpolants and orders $p_{\phi,G,{\tilde{y}}}$, respectively. If $M\ll 1$ and $p_\phi< p_G,p_{\tilde{y}}$, then Eq. \ref{eq:thm closure of observer} reduces to $\|y\left(f^na_0\right)-\tilde{y}_I\left(G_I^n\phi_I(a_0)\right)\|\leq C_4h^{p_\phi}$ for a positive constant $C_4$. Hence the approximation error only depends on the discretization of the coordinate transform $\phi$.
\end{remark}



\begin{proof}\textbf{Theorem \ref{thm:closure of observer}: correct output of the new model}
The proof consists of the following chain of arguments, see Appendix \ref{sec:appendix lemmas} for details.
\begin{enumerate}
\item Given the exact versions of the numerical objects $\phi$, $G$ and $\tilde{y}$, the output is reproduced exactly. This can be proved with Takens' theorem and the assumption of a well-behaved observer $y$ and dynamic $f$.
\item The interpolants of the numerical objects have small errors, so that multiplications of the errors vanish.
\item Taylor-decomposition of the interpolants (see Lemma \ref{lemma:approximation} in Appendix \ref{sec:appendix lemmas}) yields the statement of the Theorem when combined with the first and second step of the proof.
\end{enumerate}
\qquad\end{proof}

The original model (or live experiment) is assumed to produce negligible errors when computing observations. Also, the matrix and eigensystem computations necessary for the diffusion map can be performed very accurately and thus their errors are  neglected, too.
Closed observables are based on a truncated diffusion map (see Figure \ref{lst:algorithm}), where the truncation is such that all information necessary to close the system is kept. Hence, we can also ignore the truncation error in this case.
Three main sources of error remain:
\begin{enumerate}
\item Mapping from a given input to the initial state of the closed observables
\item Iterating the dynamic
\item Observing the current state
\end{enumerate}
Theorem \ref{thm:closure of observer} states that the dynamic contributes to the total error via $E_G$, therefore we performed a numerical experiment (see \S\ref{sec:examples2} and Figure \ref{fig:error_time}) to verify that changing the error $E_G$ also affects the total error. The other sources of error will be treated in a future paper.

\FloatBarrier
\section{Illustrating examples\label{sec:examples}}

We now show how numerical models can be constructed from observations of a system of ordinary differential equations, and from a transport-diffusion process. We chose the two examples to illustrate the ideas behind numerical model construction, so that the resulting numerical model can be explained from the known properties of the example. To show the potential of the approach, a more complex example of a train station is described in \S\ref{sec:real world example}.
The examples are treated similarly in the following text. First, the example is explained shortly, including the problem that arises. Second, closed observables are constructed using the algorithms described in \S\ref{sec:algorithms}. Third, the solution in form of a numerical model is explained and evaluated regarding efficiency and further use.

\subsection{Two-dimensional initial states, one-dimensional observer}\label{sec:examples2}

A dynamic process is often described infinitesimally via a differential equation, where the derivative of a given quantity with respect to time is set in relation to the value of the quantity itself. 
In this first example, we use an ordinary differential equation to describe a circling motion towards the origin (see Figure \ref{fig:phase_system2}). The state space is two-dimensional, and the observer measures the distance of a given state to the origin. This enables the visualization of both the state space of the system and the observed values, which is not as easy if the state is very high-dimensional, such as in many-particle systems. Also, the different dimensions of the state space, the observed space and the resulting dimensions of the numerical model can be visualized easily in this low dimension.
In the example presented here, closed observables provide a significant storage reduction in the sense of Theorem \ref{thm:storagereduction}, which we demonstrate in terms of the memory capacity of a supercomputer and a smartphone.

To define the given system precisely, let $A=[-1,1]^2$ be the state space, equal to the initial state space $A_0=A$, $y(a)=\|a\|$ be the observer and $F$ the push-forward by $\Delta t=0.1$ of the system of ordinary differential equations
\begin{equation}\label{eq:system2}
\frac{da}{dt}=-a+Ra,\ a\in A,
\end{equation}
where $R$ is the rotation matrix $[0, -1;1, 0]$. The phase diagram in this case is depicted in Figure \ref{fig:phase_system2}. 
\begin{figure}[ht]
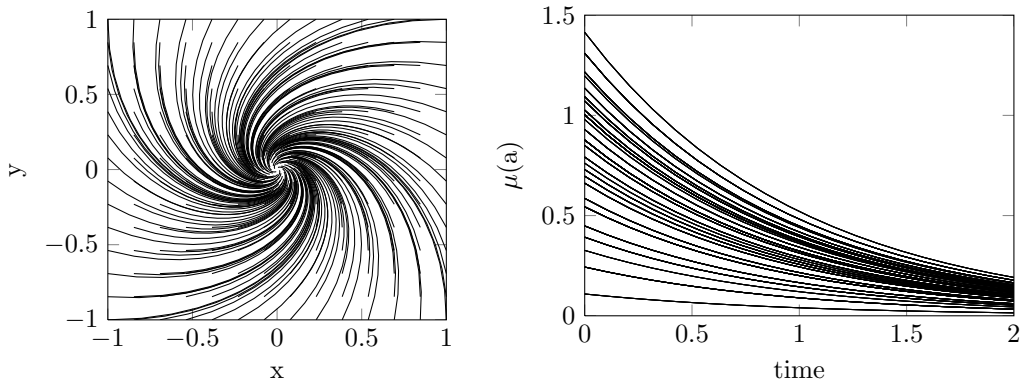

\centering
\setlength\fheight{4cm}
\setlength\fwidth{6cm}
\input{phases2.tex}
\setlength\fheight{4cm}
\setlength\fwidth{6cm}
\input{mus2.tex}
\caption{\label{fig:phase_system2}System with two-dimensional initial states. Starting from any point in $\mathcal{A}_0=[-1,1]^2$, the system spirals inward (left figure). When measured, every trajectory shows the same characteristic (right figure): the distance to the center is dropping to zero exponentially fast.}
\end{figure}
To generate observations of the system, we initialize it at equally spaced points on $A_0$ and store all $y(a_n)=\|a_n\|$ of the discrete system $a_{n+1}=F(a_n)$
up to time $t_{end}=2$.
The closed observables are constructed with the algorithm described in Figure \ref{lst:algorithm} with $T=5$, which corresponds to $\Delta t=0.5$.

Fig. \ref{fig:coordinates_system2} shows the relation between the first and following diffusion map coordinates. All but the first can be discarded, as they are functions of the first, which means that only one dimension is necessary to describe the dynamics of the system.
The state in the new coordinates also moves to a steady state exponentially fast (Fig. \ref{fig:dynamic_system2}). Figure \ref{fig:reduced_initialstates2} shows how the initial states of the original system translate to the initial states of the reduced system. In contrast to the distance to the center before (Figure \ref{fig:phase_system2}, right), they can not directly be interpreted but must be observed first. The observer function is depicted in Figure \ref{fig:reduced_initialstates2} on the right.

\begin{figure}[ht]
\centering
\setlength\fheight{5cm}
\setlength\fwidth{6cm}
\input{dy1ds2.tex}
\setlength\fheight{5cm}
\setlength\fwidth{5cm}
\input{trs2.tex}
\caption{\label{fig:dynamic_system2}Left figure: dynamic on the closed observable.
Right figure: trajectories in the reduced system for several initial states. The exponential decay is clearly visible, though the actual values of the coordinate have no direct meaning (e.g. are not directly the distance to the center).}
\end{figure}
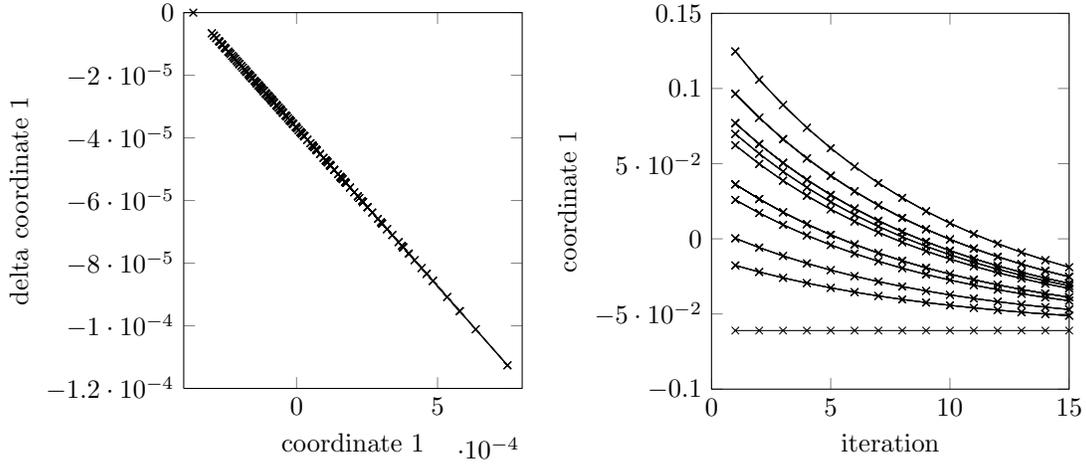
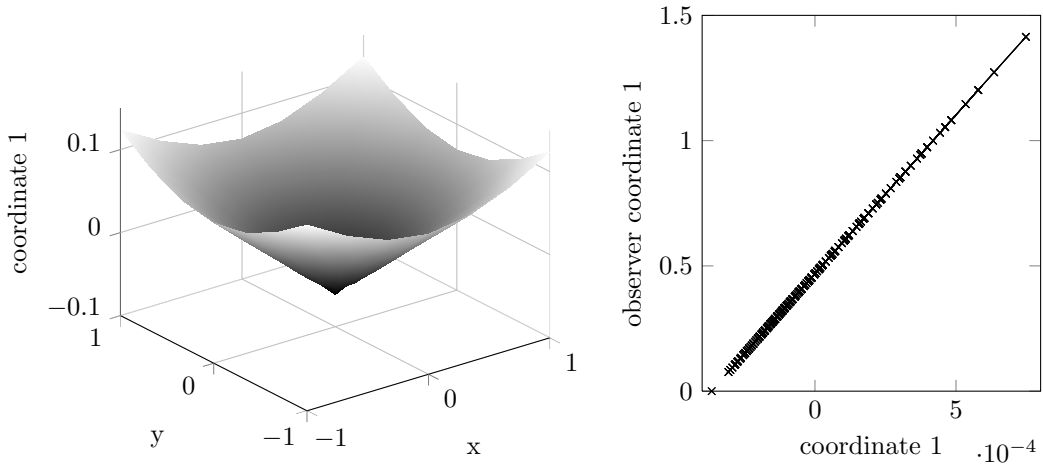
\begin{figure}[ht]
\centering
\setlength\fheight{5cm}
\setlength\fwidth{6cm}
\input{inis2.tex}
\setlength\fheight{5cm}
\setlength\fwidth{6cm}
\input{mucos2.tex}
\caption{\label{fig:reduced_initialstates2}Left: Initial states in reduced system coordinates relative to the original coordinates $(x,y)$. The reduced coordinate values are symmetrical around $(0,0)$, which is due to the choice of the observer in the original system (distance to the center). Right: observer of the closed observable $\phi_1$.}
\end{figure}

The reduction in this example is significant. Naively storing all observations over time for all initial states would take a three-dimensional hyper-surface (see Theorem \ref{thm:storagereduction} and Figure \ref{fig:storage reduction}). The reduced model with closed observables only needs two 1D-lines for dynamic and observer (Figures \ref{fig:dynamic_system2}, left and \ref{fig:reduced_initialstates2}, right), and one 2D-surface for the translation from initial state to the new coordinate (see Figure \ref{fig:reduced_initialstates2}). Even in this low-dimensional case, the naive storage of all observations over time for all initial states would take up all memory of a supercomputer (Tianhe-2, state 2015). If the same naive sampling is applied together with closed observables, the data fits in the memory of a smartphone (Figure \ref{fig:storage reduction}).
\begin{figure}[ht]
\centering
\setlength\fheight{6cm}
\setlength\fwidth{12cm}
\input{redsc.tex}
\caption{\label{fig:storage reduction}Storage of a model with closed observables compared to naive storage for example 2 in \S\ref{sec:examples2}, with parameters $D=1$, $M=1$, $N_0=2$. In both cases we assume equal spacing of points. We assume a memory capacity of $8\text{GB}$ for the smartphone and $88\text{GB}\cdot 16000$ for the supercomputer (Tianhe-2, 2015).}
\end{figure}
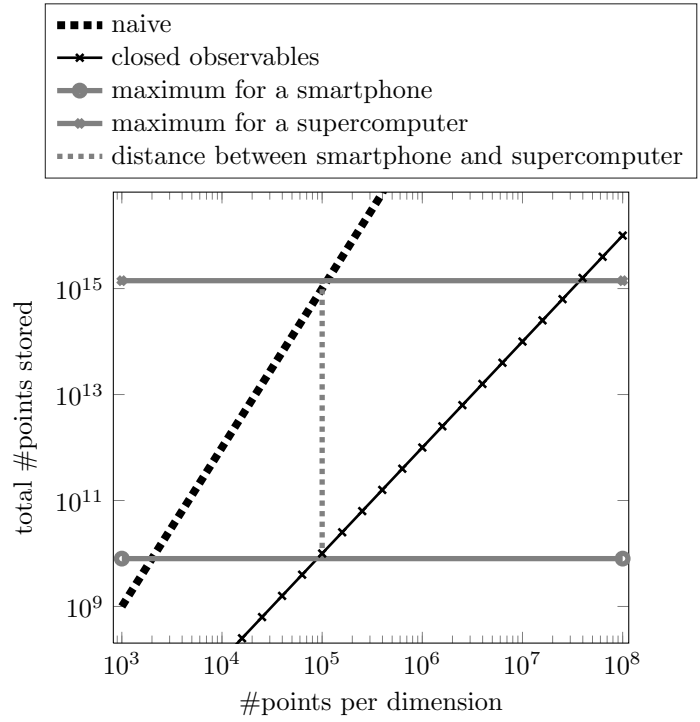

The success in this case can be explained by reformulating the system in Eq. \ref{eq:system2} with the observer $y(a)=\|a\|$. This yields the new system
\begin{equation}
\frac{db}{dt}=-b,
\end{equation}
where $b\in\mathbb{R}$ and $y(b)=b$. The translation from a value $a$ to $b$ is given by $b=\|a\|=y(a)$. Remarkably, the closed observables capture a similar system numerically.

We compared the numerical errors of experiment \ref{sec:examples2} by changing the number of points on the trajectories used to construct the dynamic $\Delta G$ of the numerical model. We used two numerical methods to approximate the difference $\phi_{n+1}-\phi_n$: a one-sided gradient, so that $\Delta G(\phi_n)=\phi_{n+1}-\phi_n$  and a gradient computed with central differences, so that $\Delta G(\phi_n)=(\phi_{n+1}-\phi_{n-1})/2$.
The results are shown in Figure \ref{fig:error_time}, where the relative error between the observations of the original and the numerical model is displayed for different numbers of numerical steps between $t=0$ and $t=2$. The change of the interpolation error as well as the change of the interpolation order directly influences the error of the numerical model, as stated by Theorem \ref{thm:closure of observer}. In our case, the one-sided gradient has order $\mathcal{O}(\Delta t)$, and the two-sided gradient has order $\mathcal{O}(\Delta t^2)$.

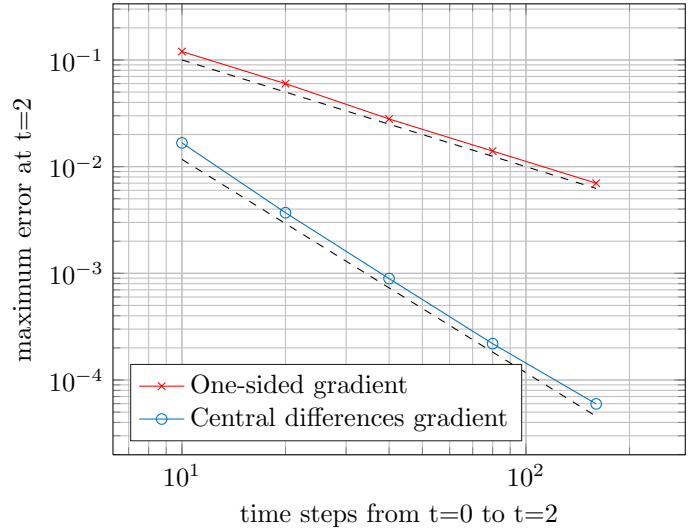
\begin{figure}[ht]
\centering
\setlength\fheight{6cm}
\setlength\fwidth{8cm}
\input{errtime.tex}
\caption{\label{fig:error_time}Maximum of relative approximation error $e=\max_t\|\frac{y( f^t(a_0))-\tilde{y}(G^t(\phi (a_0)))}{y( f^t(a_0))}\|$ from $t=0$ to $t=2$ between the original system in example \ref{sec:examples2} and the model constructed with closed observables. The maximum is also computed over $1000$ randomly chosen initial values in $A_0=\{(0,y)|1\leq y\leq 3\}$. Different numerical approximations were used for $\Delta G$: the error decreases with order $\mathcal{O}(\Delta t)$ with a one-sided gradient and  $\mathcal{O}(\Delta t^2)$ for a gradient computed with central differences (see dashed lines).}
\end{figure}

\FloatBarrier
\subsection{Two-dimensional initial states, infinite-dimensional observer}\label{sec:examplesPDE}

Diffusion, transport and advection of physical quantities such as mass, concentration or energy are very common phenomena in nature. In this example, we use closed observables to numerically capture the dynamics of such processes. We provide observations of a generic quantity $u$ over a one-dimensional domain, and change two parameters of the system to generate different results. Thus, the system as seen by the algorithms is a black box $F$ with two parameters $c$ and $d$, producing observations $u(x,t)$ for a given time $t$ and space $x$ (see Figure \ref{fig:diffusion x} and \ref{fig:diffusion x all} in the appendix):
\begin{equation}
F_{c,d}(x,t)=u(x,t).
\end{equation}

To be able to verify the results of the closed observables, we choose the classical transport-diffusion PDE in one dimension on $x\in[-0.5,0.5]$ with periodic boundary conditions and constants for diffusion $c$ and advection $d$. In our example, the transport coefficient $a$ depends linearly on $d$ via $a(d)=10d$:
\begin{eqnarray}\label{eq:example PDE}
c u_{xx} &=& d u_t + a(d) u_x\\
u(x,0) &=&\exp(-25x^2).
\end{eqnarray}
This system is chosen because for $d\neq 0$ it can be reformulated as depending only on one coefficient:
\begin{equation}\label{eq:example PDE reformulated}
\frac{c}{d} u_{xx}= u_t + 10 u_x.
\end{equation}
The reformulation shows that the underlying system is two-dimensional, because of essentially one parameter $\frac{c}{d}$ and the time dimension. Via closed observables, we are able to capture this property of the system automatically: no reformulation is necessary, we simply provide the observed values of $u(x,t)$ from the simulator. The state at time $t$ is the function $u(\cdot,t)$, and the observer $y$ is the identity, so that $y(u)=u$.

Even though the observed values are functions, numerical model construction creates a finite-dimensional dynamical system and the observer of the closed observables yields a function.
The constructed numerical model is capable of producing different dynamics for given constants $c$ and $d$, where $a(d)=10d$.
The numerical model needs two dimensions to capture the two dimensional initial conditions for diffusion and advection (see Figure \ref{fig:dynamics system PDE}), and three dimensions for the observer (for every pair $(\phi_1,\phi_2)$, $u(x)$ has to be stored for all $x$, see Figure \ref{fig:observer system PDE}). This is one dimension less then naively storing all output for all combinations of the input parameters, since for each pair $(c,d)$ one would need to store $u(x,t)$ for all $x,t$, requiring storage of four dimensions.

Figure \ref{fig:error system PDE} shows a comparison of the error between the result of closed observables at intermediate values of $c$ and $d$ not used in the generation of the model and interpolation of the observations closest to the given values of $c$ and $d$.

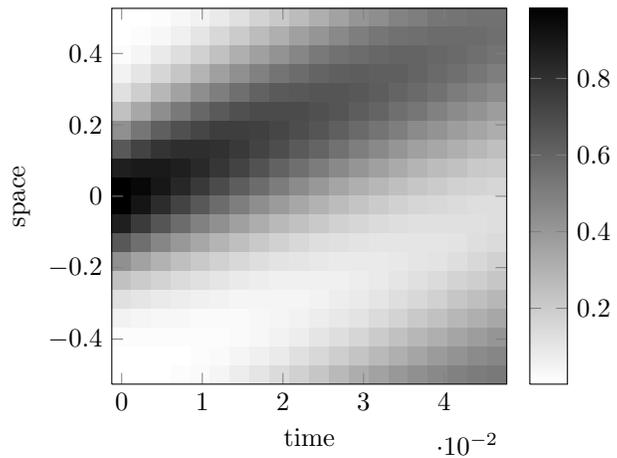
\begin{figure}[ht]
    \centering
\setlength\fheight{5cm}
\setlength\fwidth{7cm}
\input{dwa4.tex}
    \caption{\label{fig:diffusion x}Solution $u(t,x)$ of Eq. \ref{eq:example PDE} for diffusion $c=3$, advection $d=6$, and transport $a=10\cdot d=60$.}
\end{figure}

\begin{figure}[ht]
\centering
\setlength\fheight{9cm}
\setlength\fwidth{15cm}
\input{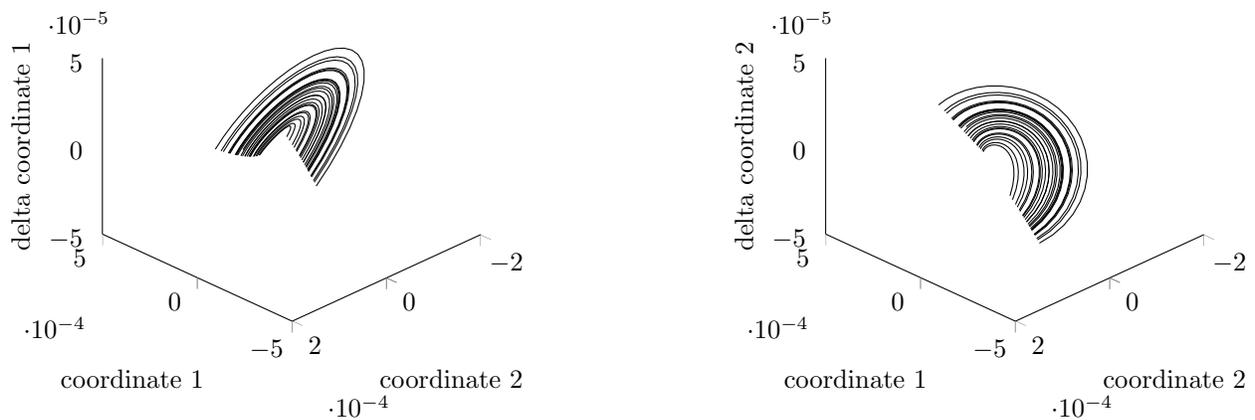}
\caption{\label{fig:dynamics system PDE}Dynamics of the closed observables, $\phi_1$ and $\phi_2$.}
\end{figure}
\begin{figure}[ht]
\centering
\setlength\fheight{7cm}
\setlength\fwidth{9cm}
\input{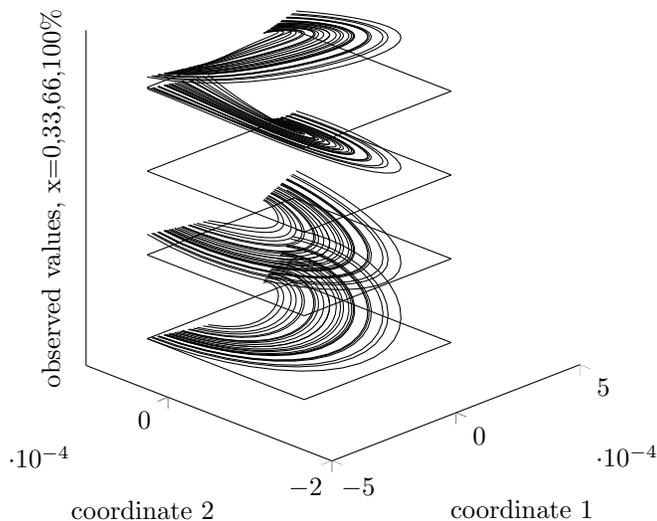}
\caption{\label{fig:observer system PDE}Slices at four points $x_1,x_2,x_3,x_4$ in $[-0.5,0.5]$ of the observer of the closed observables, $\tilde{y}(\phi_1,\phi_2)(x)$. The z-value of the upper slices is translated upward, otherwise they would all be overlapping.}
\end{figure}
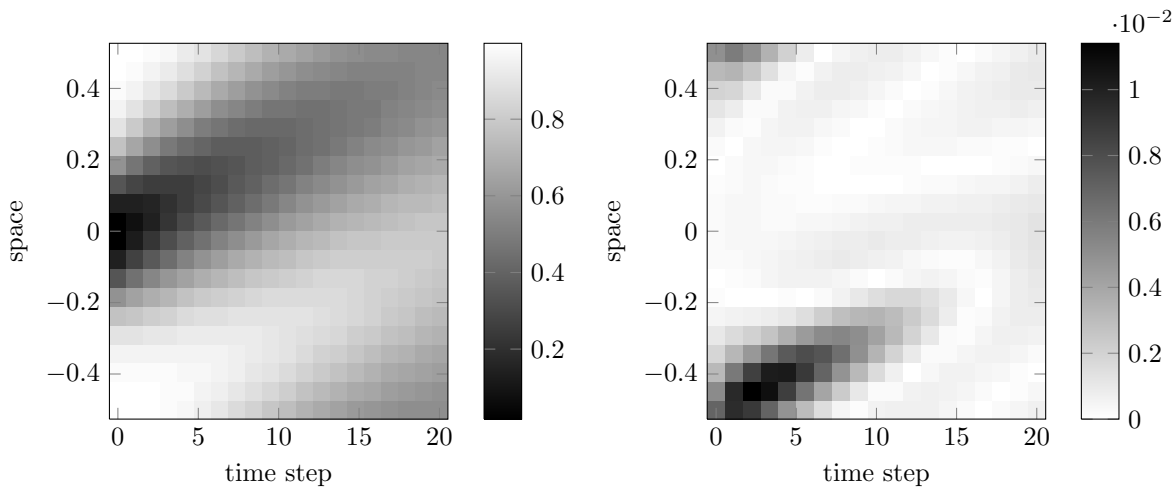
\begin{figure}[ht]
\centering
\setlength\fheight{5cm}
\setlength\fwidth{6cm}
\input{cmppde.tex}
\setlength\fheight{5cm}
\setlength\fwidth{6cm}
\input{errpde.tex}
\caption{\label{fig:error system PDE}Left: predicted dynamics when interpolating at $c=1.6$, $d=2$. Right: relative error $\epsilon=\frac{|O-P|}{|O|}$ between observations $O$ (generated with $c=1.6,d=2$ and not used in the construction of the model) and prediction $P$.}
\end{figure}

\FloatBarrier
\section{A real-world example}\label{sec:real world example}

This section shows how to apply closed observables in a real-world example as introduced at the beginning in \S\ref{sec:intro_example}.
Safety at train stations can be lowered dramatically when a train with many passengers arrives on an already crowded platform. In this scenario, opening the doors of the train is dangerous, as the number of persons on the platform can reach critical levels and people might get pushed on tracks or suffocate due to high interpersonal pressure. For safety personnel, the evolution of the number of persons on the train and on the platform can be crucial to decide whether to open the doors or evacuate the platform. A model of this process can provides hints on the future evolution of passenger numbers.

Consider such a train station with one track and a platform (see Figure \ref{fig:real world station}). Passengers can leave the train through a door and then have to pass through waiting passengers on the platform to leave the station.
There are several possibilities to simulate this system at the macroscopic scale. One could use analytical techniques such as homogenization and averaging to scale up the differential equations describing individual behaviour. This would yield a macroscopic model, possibly also a differential equation, for the exit process out of the train.
Another approach is taken in the natural sciences. The microscopic simulator or a live experiment generates output data (number of persons in and out of the train at each second), starting with several different initial settings. A macroscopic model is then proposed and validated using the output.
Classical surrogate models would try to approximate the output either for each time step individually, or as an input-response system. For example, such a system could yield the mean evacuation time given the initial number of pedestrians inside and out of the train. For dynamic properties of the system, such as the outflow over time, an interpolation of all output would be necessary.

Both analytic and experimental approaches yield a macroscopic model, in most cases a system of equations. The approach of a surrogate model with closed observables yields a dynamical system that can be used in simulations. However, similar to classical surrogate models it is not present as equations, but numerical data and interpolating functions. This data is generated automatically by the process described in \S\ref{sec:algorithms}. The result is either more storage efficient than classical surrogate models (see Thm. \ref{thm:storagereduction}), or more accurate while using the same amount of storage.

Here, we describe how to construct the macroscopic model with closed observables. The following assumptions are used for the train setting:
\begin{description}
\item[Assumption 1] Passengers follow the rules of the Gradient Navigation Model \cite{dietrich-2014}.
\item[Assumption 2] Desired speeds approximately obey a normal distribution with mean $1.34m/s$ and standard deviation $0.26m/s$ (experimental values from \cite{weidmann-1992}).
\item[Assumption 3] Initially, the positions of passengers in the train and on the platform are  uniformly distributed. We use a five second starting phase for the distribution to settle in a state where all passengers on the platform assume their desired distances to others, and passengers in the train queue in front of the door.
\item[Assumption 4] Waiting passengers do not strongly react to the leaving passengers, but move away a little if distances are too small. By this assumption, we exclude psychological effects such as the formation of a passage way in front of the door.
\item[Assumption 5] Small differences in the initial positions of the passengers do not cause large differences in behavior on the system level. This allows us to start several simulation runs with the same initial numbers of passengers but small changes in positions on train and platform, and then average over the results.
\end{description}
We are interested in the number of passengers over time, both on the train and on the platform. A simulation with 42 passengers on the train and 48 passengers waiting on the platform results in the change of passengers depicted in Figure \ref{fig:real world station}. After 15 seconds, all 90 passengers are on the platform.

The scales in this scenario differ from $\mathcal{O}(10^{-1}s)$ to $\mathcal{O}(10^1s)$ on the temporal scale (reaction time of pedestrians is set to $0.5$ seconds, the analysis of the chance to get off the train is between $25$ and $50$ seconds), and from $\mathcal{O}(10^{-1}m)$ to $\mathcal{O}(10^1m)$  on the spatial scale (diameter of a pedestrian is $0.4m$, the platform is about $20m$ wide). This means that we use closed observables to capture a system that is $\mathcal{O}(10^2)$ slower and larger than the microscopic system that creates the output.

\begin{figure}[ht]
\centering
\includegraphics[width=0.45\textwidth]{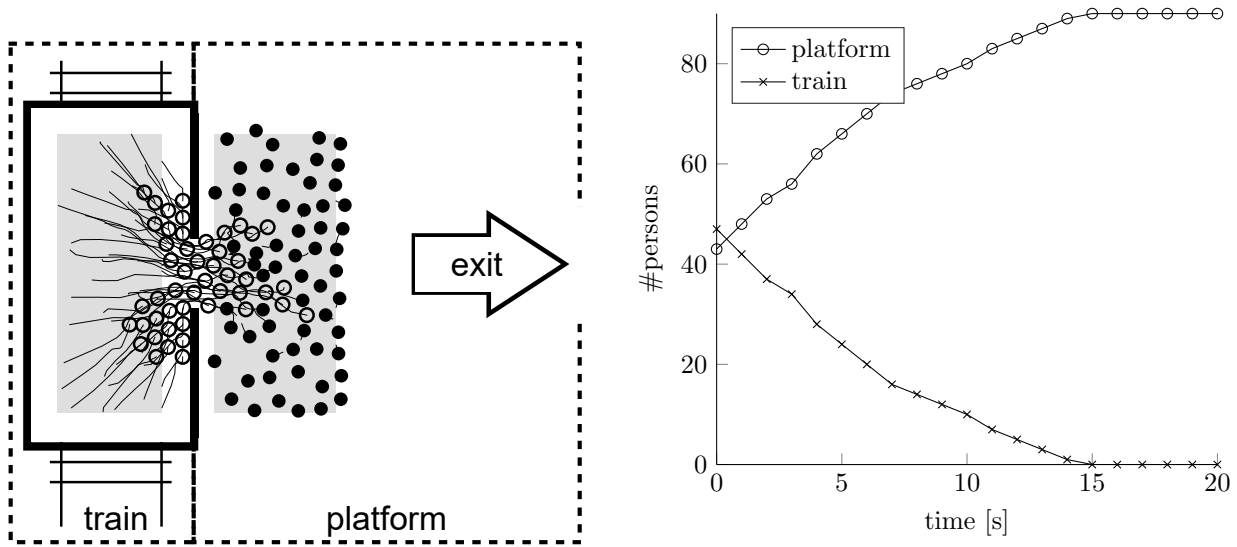}
\setlength\fheight{6cm}
\setlength\fwidth{7cm}
\input{rwevol.tex}
\caption{\label{fig:real world station}Left: The train station with a train and the platform with waiting passengers. Persons stepping off the train (empty circles) have to leave to the right. We are interested in the numbers of persons over time, both on the train and on the platform. Right: Evolution of the passenger number in the train and on the platform. 48 persons start in the train, 42 on the platform. After 15 seconds, all 90 passengers are on the platform.}
\end{figure}

We denote the numbers of passenger on the train $N_T$ and on the platform $N_P$. In our example, passengers cannot leave the station. Also, passengers leaving the train arrive on the platform. These facts can be combined to form the continuity equation, which in our case is present in the discrete version. Consider the outflow $F_T$ of passengers from the train, then:
\begin{equation}\label{eq:real world discrete continuity}
N_P(t+1)-N_P(t) = F_T(t).
\end{equation}
The function $F_T(t)$ depends on interactions between passengers and the geometry and hence is very difficult to derive analytically. A typical solution in this case is to run simulations and estimate $F_T(t)$ for different numbers of passengers, which would need $\mathcal{O}(N^3)$ points for a sampling of the two parameters and time. With closed observables, the same accuracy can be achieved by storing $\mathcal{O}(N^2)$ points (see below).

To model $N_P$ and $N_T$ with closed observables, Eq. \ref{eq:real world discrete continuity} and the consideration of the related facts are not necessary. We only need the output of the simulations started for different numbers of pedestrians on the train and on the platform. As we draw the initial positions of the persons from a uniform distribution, the output is stochastic and we have to run several simulations for the same initial numbers $N_P(0),N_T(0)$ to get a good average evolution $\overline{N}_P(t),\overline{N}_T(t)$. Table \ref{tab:real world parameters} shows the parameters of the simulation runs.
\begin{table}[ht]
\footnotesize
\centering
\begin{tabular}{|l|c|}
\hline 
\textbf{Parameter} & \textbf{Values} \\ 
\hline 
Simulated time on the train station & 50 seconds \\ 
\hline 
$N_T(0)$: Initial numbers of passengers on the train & 10,20,30,40,50 \\ 
\hline 
$N_P(0)$: Initial numbers of passengers on the platform & 0,20,40,60,80,100,120,140,160,180,200 \\ 
\hline 
Runs per initial value pair $(N_P,N_T)$ & 10 \\ 
\hline 
T: number of observations to combine for numerical model & 15 \\
\hline
\end{tabular} 
\caption{\label{tab:real world parameters}Parameters for the simulations and numerical model construction. In total, $5\cdot 11\cdot 10=550$ simulations were started. The parameter $T$ shows how many observations are combined to form a time-lagged variable (closed observables). In the example, observations were one second apart, so that 15 second intervals are combined.}
\end{table}
The output $\overline{N}_P(t),\overline{N}_T(t)$ is used to construct the diffusion map space $D$. For notational convenience, we drop the braces denoting the mean from now on. The functions $\phi$, $G$ and $\tilde{\mu}$ are created using linear interpolation by MATLABs \texttt{scatteredInterpolant}. We use the algorithm outlined in Figure \ref{lst:algorithm}. Figure \ref{fig:real world G} depicts the interpolating surfaces for $\Delta G$. We need two since we employ two-dimensional closed observables, that means in this real world example,
\begin{equation}
(\phi_1,\phi_2)_{n+1}=(\phi_1,\phi_2)_{n}+\Delta G((\phi_1,\phi_2)_{n}).
\end{equation}

Since we use interpolation, the original output can be recreated up to high accuracy (see Figure \ref{fig:real world interpolation error}). For input not covered in the simulations, the trajectories generated by the numerical model closely predict the system behavior. Figure \ref{fig:real world trajectory} shows a comparison between a predicted trajectory and three corresponding simulations not used for model construction ($N_P(0)=60$, $N_T(0)=35$). Note that the prediction is for the mean of the simulation runs.

\begin{figure}[ht]
\centering
\includegraphics[width=1.0\textwidth]{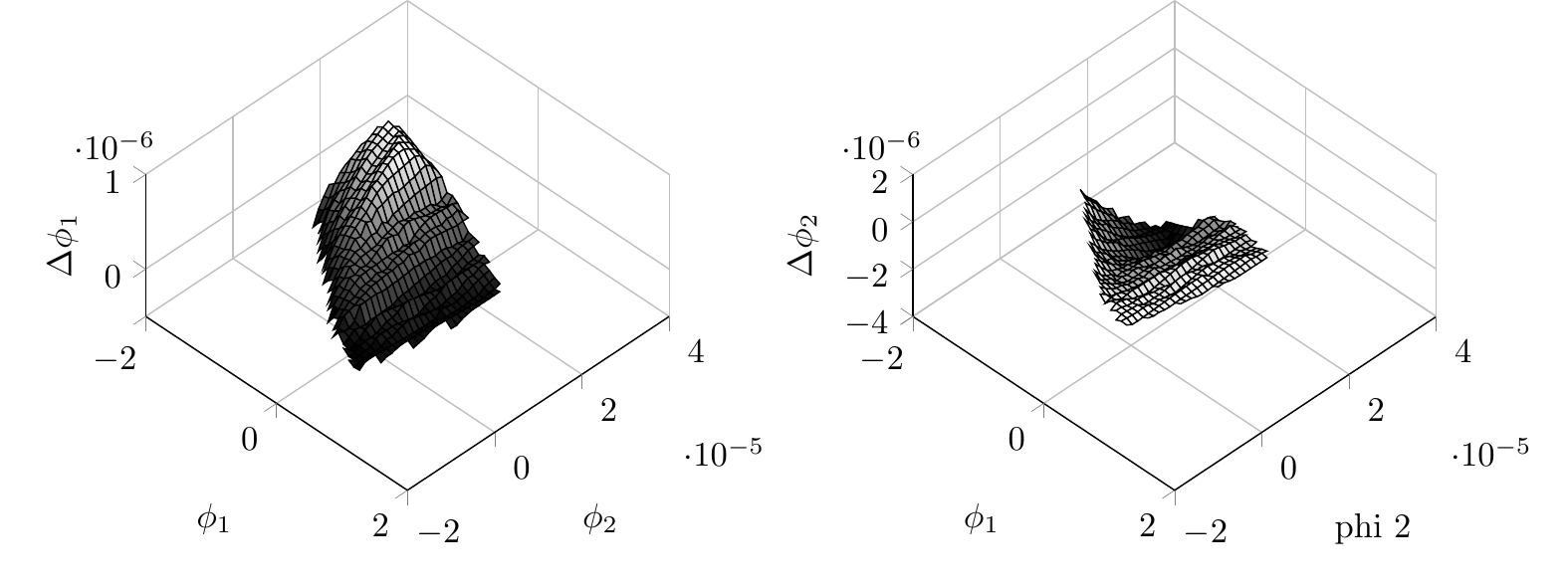}
\caption{\label{fig:real world G}Evolution function $\Delta G(\phi_1,\phi_2)=(\Delta \phi_1,\Delta \phi_2)$ in the form of two functions $\Delta \phi_1=G-\phi_1$, $\Delta \phi_2=G-\phi_2$ on closed observables $\phi_1$ and $\phi_2$.}
\end{figure}
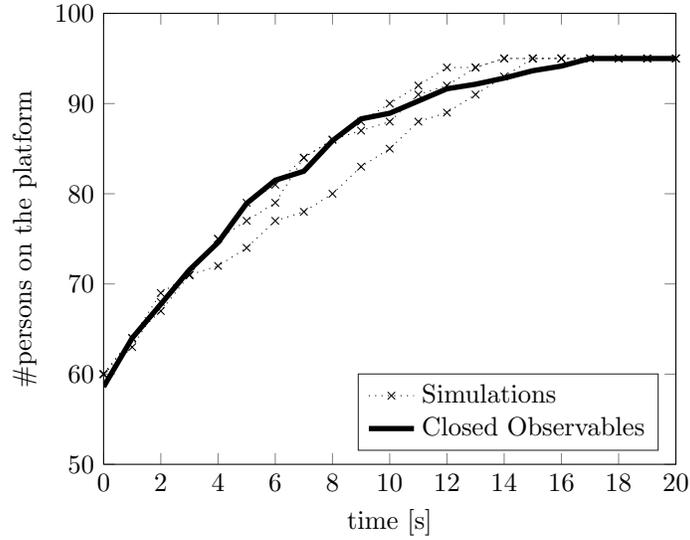
\begin{figure}[ht]
\centering
\setlength\fheight{6cm}
\setlength\fwidth{8cm}
\input{rwt.tex}
\caption{\label{fig:real world trajectory}A prediction (bold line) of the average number of persons on the platform by the constructed model, compared to three simulation runs (dotted). The incorrect starting values for the prediction results from interpolation errors: the given initial values $(N_P(0),N_T(0))=(60,35)$ are transformed by $\phi$ into the space $D$ and after the simulation recreated by observation with $\tilde{y}$.}
\end{figure}

An analysis of the number of passengers getting off the train in the simulated $25$ or $50$ seconds yields the result shown in Figure \ref{fig:real world analysis}. We compute the mean chance $c$ for one pedestrian to get off the train in $n$ seconds by
\begin{equation}
c(n)=1-\frac{\min_{t\leq n} N_T(t)}{\max_{t\leq n} N_T(t)}=1-\frac{\min_{t\leq n} N_T(t)}{N_T(0)}.
\end{equation}
The chance to get off the train in time decreases with the number of passengers waiting on the platform. Running the necessary 400 simulations took about two seconds, whereas the generation of the data with the original model would last for hours. It would also need 10 simulation runs per data point to generate the average value, resulting in 4000 simulations for the same result.

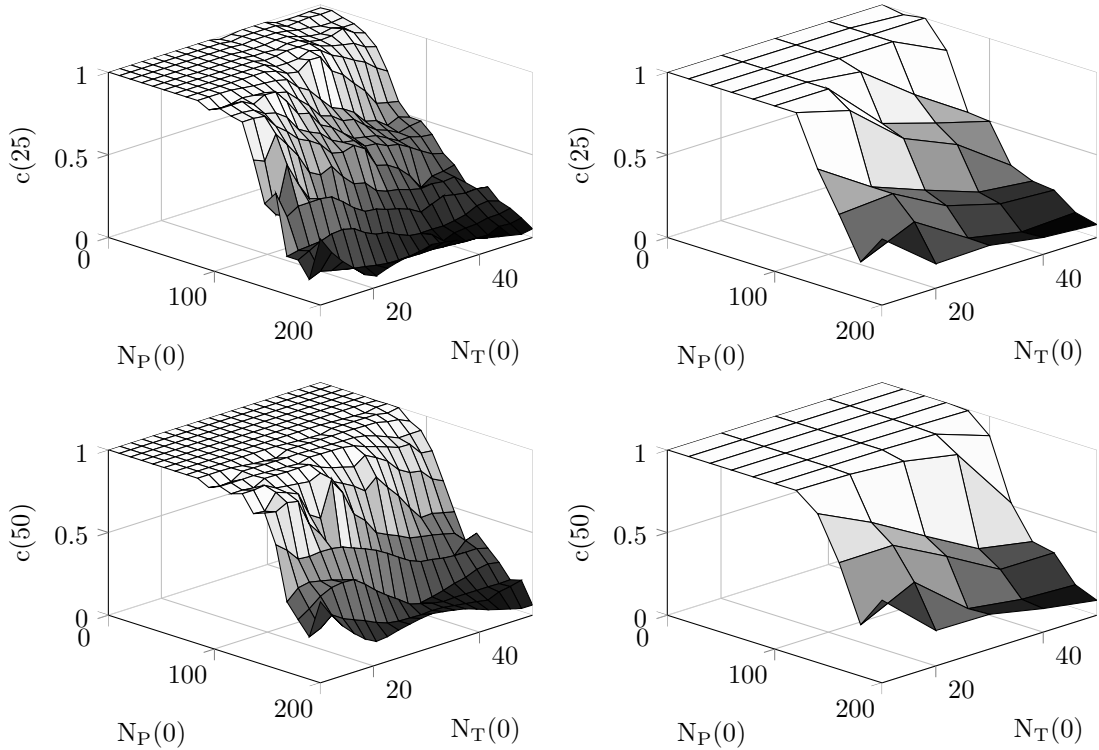
\begin{figure}[ht]
\centering
\setlength\fheight{4cm}
\setlength\fwidth{6cm}
\input{rwsa25.tex}
\setlength\fheight{4cm}
\setlength\fwidth{6cm}
\input{rwa25.tex}
\setlength\fheight{4cm}
\setlength\fwidth{6cm}
\input{rwsa50.tex}
\setlength\fheight{4cm}
\setlength\fwidth{6cm}
\input{rwa50.tex}
\caption{\label{fig:real world analysis}Left column: Prediction and analysis with the numerical model constructed with closed observables. The chance of passengers getting off the train in 25 seconds (top row) and 50 seconds (bottom row) decreases with the number of passengers waiting on the platform. Right column: The same analysis performed on the output used to construct the numerical model. Small features in the surface are not visible.}
\end{figure}

The real-world example shows where closed observables are advantageous. Compared to the original model, creating a model of the process of stepping of the train is much more efficient computationally than the original, for both analysis and prediction (see Figure \ref{fig:real world analysis}). Compared to storage of the full simulation output, it is also much more efficient. Theorem \ref{thm:storagereduction} states that in this example the constructed model is exponentially more efficient regarding data storage, as the input space is two-dimensional ($\dim A_0=2$) and the closed observables also have two dimensions ($\dim D=2$).
A response surface common in surrogate modeling cannot yield the same result, because it would be constructed only for certain end times (for example 25 or 50 seconds). The model constructed with closed observables can produce results up to any given end time.

\FloatBarrier
\section{Conclusion\label{sec:conclusion}}

Many systems often show interesting dynamics on multiple temporal and spatial scales. If such a system is present on a fine scale, generating a model for the same process but a coarser  time scale is difficult but necessary to perform analysis, optimization and control.
We introduced the concept of numerical models with closed observables on a coarse scale and showed how observations of a system can be used to construct it.

In case the number of dynamically meaningful variables, called closed observables, is not higher than the number of input parameters, constructing the numerical model is more efficient than storing and interpolating the output (Theorem \ref{thm:storagereduction}). This can easily make the difference between the memory capacity of a smartphone and a supercomputer (see example \ref{sec:examples2}).

Two academic examples illustrated the features of our approach, from initial value mapping to storage reduction and finite system construction of PDEs. A real-world example showed that numerical model construction can extract the dynamics of macroscopic systems from fine scale simulations, with performance gains of four orders of magnitude.

Future work will focus on two major areas: a search for analytic forms of closed observables, and scaling up stochastic, fine scale systems.
Numerical models can be used in various fields of numerical analysis. In network optimization, fine scale simulations could now be used to build models on the vertices or edges of the network. Systematic upscaling \cite{brandt-2011} needs models on several scales, which could be simplified with the numerical model construction ideas presented here.
Uncertainty quantification can now be split into a construction phase, where the numerical model on the scale of interest is constructed with the fine scale simulator, and an analysis phase, where the high performance of the numerical model can be used to perform quantification of uncertainties without running the original simulator.

When enough experimental data is available, one could also construct numerical models directly from the observations and in this way extract macroscopic dynamics automatically from experiments.




\appendix
\section{Lemmas, Proofs}\label{sec:appendix lemmas}

\begin{lemma}\label{lemma:approximation}
Consider two vector spaces $A$ and $D$. Let $G\in C^2(D,D)$ and $\phi\in C^2(A,D)$. Consider the interpolants of these functions, $G_I\in C^2(D,D)$ and $\phi_I\in C^2(A,D)$, with interpolation errors $E_G(\phi):=G_I(\phi)-G(\phi)$ and $E_\phi(a):=\phi_I(a)-\phi(a)$. We denote $E_G:=\max_{\phi\in D}E_G(\phi)$ and $E_\phi:=\max_{a\in A}E_\phi(a)$.
Also, let $\|DG\|$ be bounded by $M\in(0,1)$, so that $M\geq \sup_\phi\|DG(\phi)\|$.
Then, for a given $a\in A$ and a fixed $n\in\mathbb{N}$,
\begin{equation}
\|G^n(\phi (a))-G_I^n(\phi_I (a))\|\leq C_1\left(\frac{1-M^{n+1}}{1-M}\right)\|E_G\|+C_2(M^{n}\|E_\phi\|)
\end{equation}
where the constants $C_1,C_2$ are positive and do not depend on $a$ and $n$.
\end{lemma}
\begin{proof}
The Taylor-decomposition of a vector-valued $C^2$ function $f$ at a point $x+\epsilon$ with small $\|\epsilon\|$ is given by
\begin{equation}
f(x+\epsilon)=f(x)+ DF(x) \cdot \epsilon+\mathcal{O}(\|\epsilon\|^2).
\end{equation}
This approximation is the basis for the following arguments. For any $\phi\in D$ and $n\geq 1$,
\begin{eqnarray*}
\|G_I^n(\phi)\|&=&\|G_I^{n-1}\circ[G(\phi)+E_G(\phi)]\|\\
&=&\|G_I^{n-2}\circ[G^2(\phi)+DG(G(\phi))\cdot E_G(\phi)+\mathcal{O}(\|E_G(\phi)\|^2)+E_G(\phi)]\|\\
&=&\|G_I^{n-2}\circ[G^2(\phi)+\mathcal{O}((M+1)\|E_G(\phi)\|)+\mathcal{O}(\|E_G(\phi)\|^2)]\|\\
&=&\|G_I^{n-3}\circ[G^3(\phi)+\mathcal{O}((M(M+1)+1)\|E_G(\phi)\|)+\mathcal{O}(\|E_G(\phi)\|^2)]\|\\
&=&\dots\\
&=&\|G^n(\phi)+\mathcal{O}(\sum_{i=0}^nM^i\|E_G(\phi)\|)+\mathcal{O}(\|E_G(\phi)\|^2)\|.
\end{eqnarray*}
With $M\in(0,1)$, we thus have
\begin{equation}
\|G_I^n(\phi)\|\leq\|G^n(\phi)+\mathcal{O}\left(\frac{1-M^{n+1}}{1-M}\|E_G(\phi)\|\right)\|
\end{equation}
Similarly, for any $a\in A$,
\begin{eqnarray*}
\|G^n(\phi_I(a))\|&=&\|[G^{n-1}\circ G](\phi(a)+E_\phi(a))\|\\
&=&\|G^{n-1}\circ[G(\phi(a))+DG(\phi(a))\cdot E_\phi(a)+\mathcal{O}(\|E_\phi(a)\|^2)]\|\\
&=&\|G^{n-2}\circ[G^2(\phi(a))+DG(\phi(a))\cdot \mathcal{O}(M\|E_\phi(a)\|)+\mathcal{O}(\|E_\phi(a)\|^2)]\|\\
&=&\dots\\
&\leq&\|G^n(\phi(a))\|+\mathcal{O}(M^{n}\|E_\phi\|).
\end{eqnarray*}
Note that for $M\in(0,1)$, the right side converges to $\|G^n(\phi(a))\|$ if $n\to\infty$.
Combining the two approximations, we conclude that for any $a\in A$ and $M\in(0,1)$,
\begin{eqnarray*}
\|G_I^n(\phi_I(a))\|&\leq&\|G^n(\phi(a))+\mathcal{O}\left(\frac{1-M^{n+1}}{1-M}\|E_G\|\right)+\mathcal{O}(M^{n}\|E_\phi\|)\|\\
&=&\|G^n(\phi(a))\|+C_1\left(\frac{1-M^{n+1}}{1-M}\right)\|E_G\|+C_2 M^{n}\|E_\phi\|
\end{eqnarray*}
for constants $C_1,C_2>0$ independent of $a$ and $n$.
\qquad\end{proof}

\section{Figures}

\begin{figure}[ht!]
    \centering
    \begin{subfigure}[t]{0.45\textwidth}
        \centering
\setlength\fheight{5cm}
\setlength\fwidth{6cm}
\input{dwa1.tex}
        \caption{Diffusion $c=1.4$, advection $d=2$}
    \end{subfigure}%
    ~ 
    \begin{subfigure}[t]{0.45\textwidth}
        \centering
\setlength\fheight{5cm}
\setlength\fwidth{6cm}
\input{dwa2.tex}
        \caption{Diffusion $c=3$, advection $d=2$}
    \end{subfigure}
    ~ 
    \begin{subfigure}[t]{0.45\textwidth}
        \centering
\setlength\fheight{5cm}
\setlength\fwidth{6cm}
\input{dwa3.tex}
        \caption{Diffusion $c=1.4$, advection $d=6$}
    \end{subfigure}
    ~ 
    \begin{subfigure}[t]{0.45\textwidth}
        \centering
\setlength\fheight{5cm}
\setlength\fwidth{6cm}
\input{dwa4.tex}
        \caption{Diffusion $c=3$, advection $d=6$}
    \end{subfigure}
    \caption{\label{fig:diffusion x all}Results for different values of diffusion and advection constants in the PDE example \S\ref{sec:examplesPDE}. Time is oriented horizontally, space vertically.}
\end{figure}
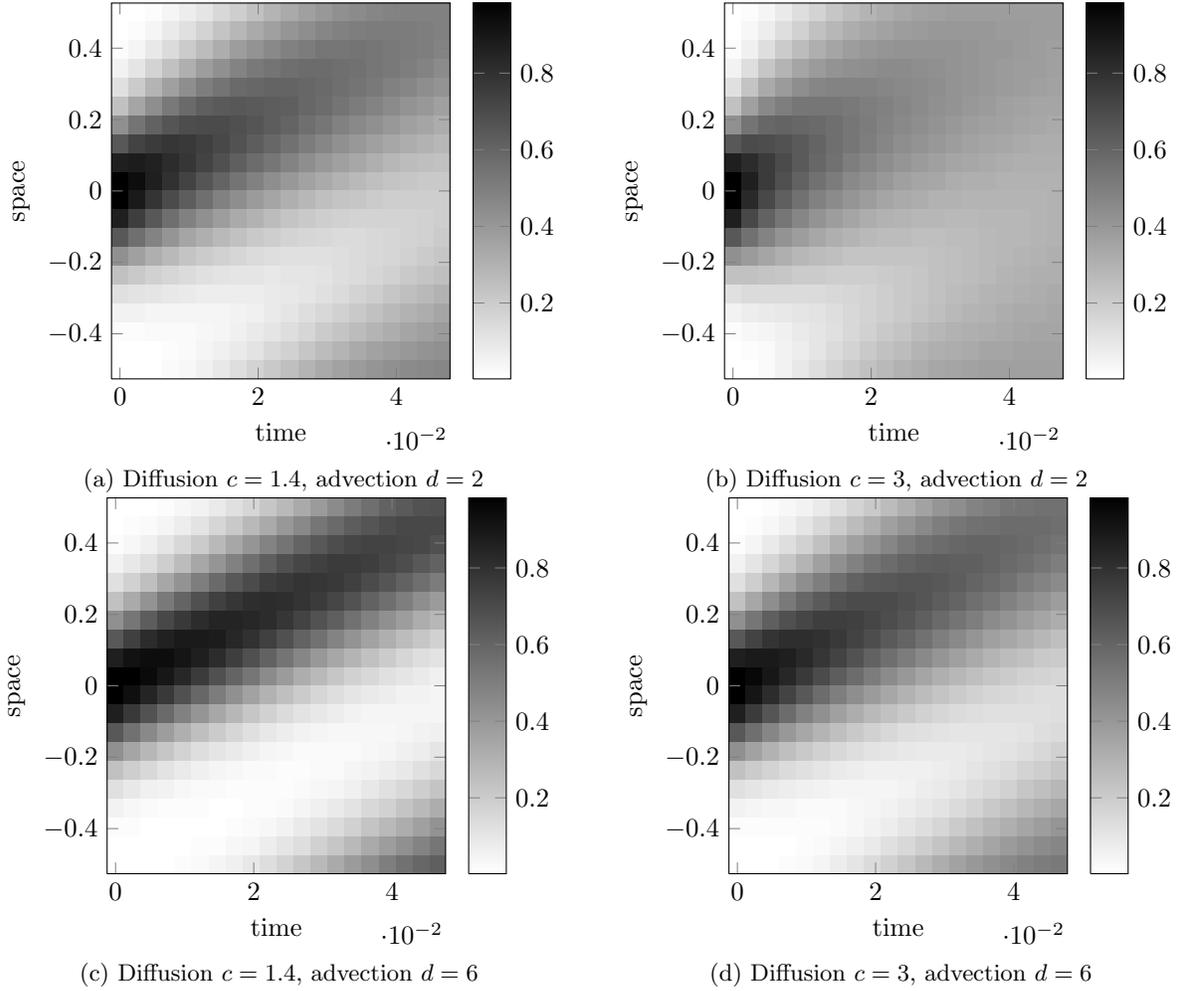

\begin{figure}[h!]
\centering
\setlength\fheight{4cm}
\setlength\fwidth{0.95\textwidth}
\input{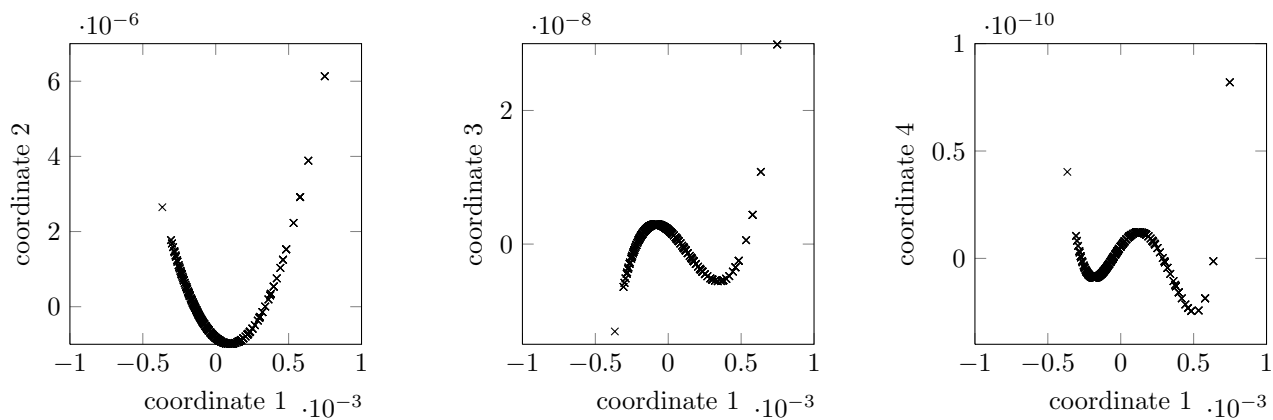}
\caption{\label{fig:coordinates_system2}Reduced system coordinates relative to the first. The second, third and fourth coordinates can be expressed using the first, so they do not contribute any additional information to the system.}
\end{figure}

\begin{figure}[ht]
\centering
\setlength\fheight{0.5\textwidth}
\setlength\fwidth{0.5\textwidth}
\input{rwint.tex}
\caption{\label{fig:real world interpolation error}Error between the predictions of the numerical model and the observations used to construct it. For all but the last time steps (where interpolation is no longer valid), the error is at the order of machine precision.}
\end{figure}
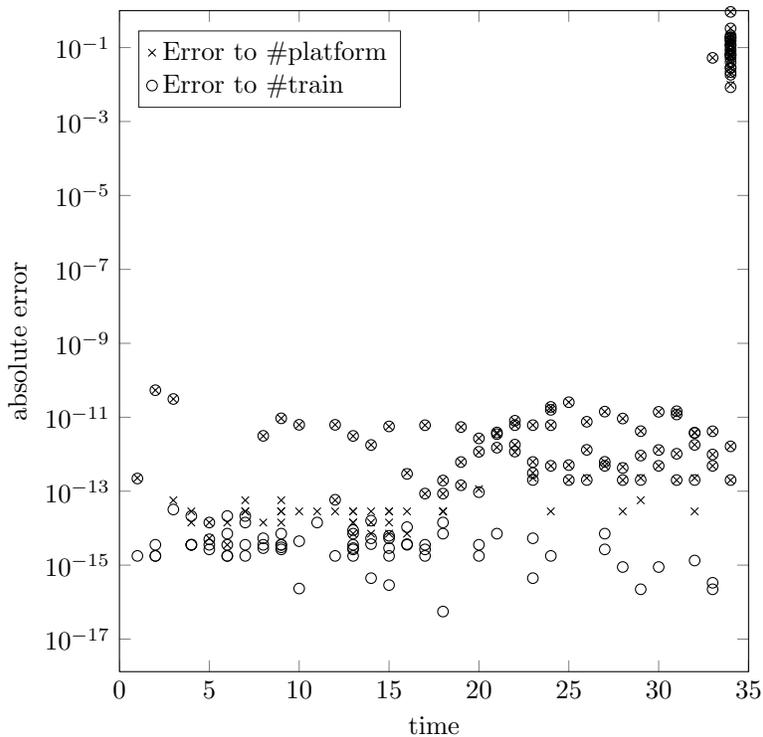

\FloatBarrier
\section{Acknowledgments}
The authors would like to thank Tobias Neckel for fruitful discussions about multiscale modeling. Support from the TopMath Graduate Center of TUM Graduate School at Technische Universit\"{a}t M\"{u}nchen, Germany is gratefully acknowledged.

\bibliographystyle{siam}
\bibliography{Literature}

\end{document}

%% file: twosys.tex
\fbox{
\begin{tikzpicture}
\node (input)[process]{input};
\node (input2)[process, right=5em of input]{$\phi$(input)};

\draw [thick,->] (input)--(input2);

\node (model)[processlarge, below=3em of input,fill=lightgray,align=center]{original model\\$F(x_n)=x_{n+1}$};
\node (newmodel)[processlarge, below=3em of input2,fill=lightgray,align=center]{numerical model\\$G(\phi_n)=\phi_{n+1}$};

\node (output)[process, below=3em of model]{output};
\node (output2)[process, below=3em of newmodel]{output'};

\draw [thick,->](input)--(model) node [midway,left] {$x_0$};
\draw [thick,->](input2)--(newmodel) node [midway,right] {$\phi_0$};
\draw [thick,->](newmodel)--(output2) node [midway,right] {observe $\phi_n$};
\draw [thick,->](model)--(output) node [midway,left] {observe $x_n$};
\draw [dashed,->](model)--(newmodel);

\draw node at ($(output)!0.5!(output2)$) {$\approx$};
\end{tikzpicture}}

%% file: dy1ds2.tex
%
%
\begin{tikzpicture}

\begin{axis}[%
width=0.75\fwidth,
height=\fheight,
at={(0\fwidth,0\fheight)},
scale only axis,
xmin=-0.0004,
xmax=0.0008,
xlabel={coordinate 1},
ymin=-0.00012,
ymax=0,
ylabel={delta coordinate 1},
axis background/.style={fill=white}
]
\addplot [color=black,solid,mark=x,mark options={solid},forget plot]
  table[row sep=crcr]{%
0.000747845590505263	-0.000112608897273791\\
0.000635236693231472	-0.000101108221398998\\
0.000534128471832474	-9.08043205732564e-05\\
0.000443324151259217	-8.15681524104607e-05\\
0.000361755998848757	-7.32861535026762e-05\\
0.00028846984534608	-6.58566786156994e-05\\
0.000222613166730381	-5.91900306238005e-05\\
0.000163423136106581	-5.32058973094306e-05\\
0.00011021723879715	-4.78330778977585e-05\\
6.23841608993914e-05	-4.30078709353128e-05\\
1.93762899640786e-05	-3.86735287587238e-05\\
-1.92972387946452e-05	-3.47793536868049e-05\\
-5.40765924814501e-05	-3.12799716485995e-05\\
-8.53565641300495e-05	-2.81349280575937e-05\\
};
\addplot [color=black,solid,mark=x,mark options={solid},forget plot]
  table[row sep=crcr]{%
0.000578710165056088	-9.5344987611527e-05\\
0.000483365177444561	-8.56387943049932e-05\\
0.000397726383139568	-7.69367283017925e-05\\
0.000320789654837775	-6.91317458155506e-05\\
0.000251657909022225	-6.21291165064377e-05\\
0.000189528792515787	-5.58442828624089e-05\\
0.000133684509653378	-5.02021087105354e-05\\
8.34824009428427e-05	-4.51355903503061e-05\\
3.83468105925366e-05	-4.05848923344421e-05\\
-2.23808174190555e-06	-3.64967242658877e-05\\
-3.87348060077932e-05	-3.28232825375152e-05\\
-7.15580885453084e-05	-2.95220459518086e-05\\
-0.000101080134497117	-2.65547274885353e-05\\
-0.000127634861985652	-2.38873117938619e-05\\
};
\addplot [color=black,solid,mark=x,mark options={solid},forget plot]
  table[row sep=crcr]{%
0.000461418164064298	-8.34072163901125e-05\\
0.000378010947674185	-7.49354919964999e-05\\
0.000303075455677685	-6.73364530927442e-05\\
0.000235739002584941	-6.0518049502593e-05\\
0.000175220953082348	-5.43980621510351e-05\\
0.000120822890931313	-4.89036080328704e-05\\
7.19192828984425e-05	-4.39693270557186e-05\\
2.79499558427239e-05	-3.95373043165039e-05\\
-1.158734847378e-05	-3.55554083929148e-05\\
-4.71427568666948e-05	-3.19774553401757e-05\\
-7.91202122068705e-05	-2.87617542247662e-05\\
-0.000107881966431637	-2.5871362640644e-05\\
-0.000133753329072281	-2.3272877340421e-05\\
-0.000157026206412702	-2.09366526252073e-05\\
};
\addplot [color=black,solid,mark=x,mark options={solid},forget plot]
  table[row sep=crcr]{%
0.000418558451533631	-7.90520847908361e-05\\
0.000339506366742795	-7.1029342983542e-05\\
0.000268477023759253	-6.38318842115535e-05\\
0.0002046451395477	-5.73726608580559e-05\\
0.000147272478689644	-5.15743779410008e-05\\
9.56981007486431e-05	-4.63679165444788e-05\\
4.93301842041643e-05	-4.16919015316666e-05\\
7.63828267249773e-06	-3.74912338331978e-05\\
-2.98529511607001e-05	-3.37170227652779e-05\\
-6.3569973925978e-05	-3.03252189999585e-05\\
-9.38951929259365e-05	-2.72767553843909e-05\\
-0.000121171948310327	-2.45363359720394e-05\\
-0.000145708284282367	-2.20726589576827e-05\\
-0.00016778094324005	-1.98574004789403e-05\\
};
\addplot [color=black,solid,mark=x,mark options={solid},forget plot]
  table[row sep=crcr]{%
0.000461418072760261	-8.34072030174934e-05\\
0.000378010869742768	-7.49354134861679e-05\\
0.0003030754562566	-6.73365131690526e-05\\
0.000235738943087547	-6.05179764424599e-05\\
0.000175220966645087	-5.43981121355097e-05\\
0.000120822854509577	-4.89035500177586e-05\\
7.19193044918188e-05	-4.39693661480043e-05\\
2.79499383438145e-05	-3.95372637155969e-05\\
-1.15873253717824e-05	-3.55554374705956e-05\\
-4.71427628423779e-05	-3.1977430832625e-05\\
-7.91201936750029e-05	-2.87617746709899e-05\\
-0.000107881968345993	-2.58713513037148e-05\\
-0.000133753319649708	-2.32728833215416e-05\\
-0.000157026202971249	-2.09366565104994e-05\\
};
\addplot [color=black,solid,mark=x,mark options={solid},forget plot]
  table[row sep=crcr]{%
0.000578710120461909	-9.53449764870871e-05\\
0.000483365143974822	-8.56387740841693e-05\\
0.000397726369890652	-7.6936755316778e-05\\
0.000320789614573874	-6.91317113596701e-05\\
0.000251657903214204	-6.21291440737971e-05\\
0.000189528759140407	-5.58442475839709e-05\\
0.000133684511556436	-5.02021331237162e-05\\
8.34823784327201e-05	-4.51355605262717e-05\\
3.83468179064484e-05	-4.05849118386165e-05\\
-2.23809393216808e-06	-3.64967020327142e-05\\
-3.87347959648823e-05	-3.28232967700756e-05\\
-7.15580927349579e-05	-2.95220312163224e-05\\
-0.00010108012395128	-2.65547365305851e-05\\
-0.000127634860481865	-2.38873058075043e-05\\
};
\addplot [color=black,solid,mark=x,mark options={solid},forget plot]
  table[row sep=crcr]{%
0.000747845590505265	-0.000112608897273819\\
0.000635236693231446	-0.000101108221398966\\
0.00053412847183248	-9.08043205732629e-05\\
0.000443324151259217	-8.15681524104606e-05\\
0.000361755998848757	-7.32861535026762e-05\\
0.00028846984534608	-6.58566786156994e-05\\
0.000222613166730381	-5.91900306238005e-05\\
0.000163423136106581	-5.32058973094306e-05\\
0.00011021723879715	-4.78330778977585e-05\\
6.23841608993914e-05	-4.30078709353128e-05\\
1.93762899640786e-05	-3.86735287587238e-05\\
-1.92972387946452e-05	-3.47793536868049e-05\\
-5.40765924814501e-05	-3.12799716485995e-05\\
-8.53565641300495e-05	-2.81349280575937e-05\\
};
\addplot [color=black,solid,mark=x,mark options={solid},forget plot]
  table[row sep=crcr]{%
0.000578710120461909	-9.53449764870871e-05\\
0.000483365143974822	-8.56387740841693e-05\\
0.000397726369890652	-7.6936755316778e-05\\
0.000320789614573874	-6.91317113596701e-05\\
0.000251657903214204	-6.21291440737971e-05\\
0.000189528759140407	-5.58442475839709e-05\\
0.000133684511556436	-5.02021331237162e-05\\
8.34823784327201e-05	-4.51355605262717e-05\\
3.83468179064484e-05	-4.05849118386165e-05\\
-2.23809393216808e-06	-3.64967020327142e-05\\
-3.87347959648823e-05	-3.28232967700756e-05\\
-7.15580927349579e-05	-2.95220312163224e-05\\
-0.00010108012395128	-2.65547365305851e-05\\
-0.000127634860481865	-2.38873058075043e-05\\
};
\addplot [color=black,solid,mark=x,mark options={solid},forget plot]
  table[row sep=crcr]{%
0.000373290282107307	-7.44563778247475e-05\\
0.00029883390428256	-6.69066435008107e-05\\
0.000231927260781749	-6.01323716310199e-05\\
0.000171794889150729	-5.40518117742342e-05\\
0.000117743077376495	-4.85927107311965e-05\\
6.91503666452986e-05	-4.36901030412838e-05\\
2.54602636040149e-05	-3.92864514200657e-05\\
-1.38261878160508e-05	-3.53300447725267e-05\\
-4.91562325885775e-05	-3.17748957046917e-05\\
-8.09311282932692e-05	-2.85797344423461e-05\\
-0.000109510862735615	-2.57076982350396e-05\\
-0.000135218560970655	-2.31257728636867e-05\\
-0.000158344333834342	-2.08043437527596e-05\\
-0.000179148677587101	-1.87169547916592e-05\\
};
\addplot [color=black,solid,mark=x,mark options={solid},forget plot]
  table[row sep=crcr]{%
0.000217367604243283	-5.86593576913889e-05\\
0.000158708246551894	-5.27295410109615e-05\\
0.000105978705540932	-4.74053161031178e-05\\
5.85733894378144e-05	-4.26236654252706e-05\\
1.59497240125438e-05	-3.83283656159727e-05\\
-2.23786416034289e-05	-3.44692127703454e-05\\
-5.68478543737744e-05	-3.10012473789748e-05\\
-8.78491017527492e-05	-2.7884388615828e-05\\
-0.000115733490368577	-2.50825955374722e-05\\
-0.000140816085906049	-2.25637851415939e-05\\
-0.000163379871047643	-2.02990102011313e-05\\
-0.000183678881248775	-1.82625321474394e-05\\
-0.000201941413396214	-1.64310680232979e-05\\
-0.000218372481419512	-1.47839277266501e-05\\
};
\addplot [color=black,solid,mark=x,mark options={solid},forget plot]
  table[row sep=crcr]{%
0.000155387317498033	-5.23940349002022e-05\\
0.00010299328259783	-4.71040463082288e-05\\
5.58892362896015e-05	-4.23530422361666e-05\\
1.35361940534349e-05	-3.80852929236416e-05\\
-2.45490988702067e-05	-3.42507390938835e-05\\
-5.87998379640902e-05	-3.08049658277739e-05\\
-8.96048037918641e-05	-2.77078973141848e-05\\
-0.000117312701106049	-2.49239983797238e-05\\
-0.000142236699485773	-2.242113322044e-05\\
-0.000164657832706213	-2.01707969105891e-05\\
-0.000184828629616802	-1.81471842380484e-05\\
-0.00020297581385485	-1.63273723786176e-05\\
-0.000219303186233468	-1.4690625121218e-05\\
-0.000233993811354686	-1.3218461052448e-05\\
};
\addplot [color=black,solid,mark=x,mark options={solid},forget plot]
  table[row sep=crcr]{%
0.000217367556605077	-5.86593485140504e-05\\
0.000158708208091026	-5.27295105645025e-05\\
0.000105978697526524	-4.74053452644906e-05\\
5.85733522620332e-05	-4.26236286641168e-05\\
1.59497235979164e-05	-3.83283927878724e-05\\
-2.2378669189956e-05	-3.4469179231394e-05\\
-5.684784842135e-05	-3.10012699462513e-05\\
-8.78491183676013e-05	-2.78843622376553e-05\\
-0.000115733480605257	-2.50826127787188e-05\\
-0.000140816093383975	-2.25637667438228e-05\\
-0.000163379860127798	-2.02990229048744e-05\\
-0.000183678883032673	-1.82625210246146e-05\\
-0.000201941404057287	-1.64310749405035e-05\\
-0.000218372478997791	-1.47839251019278e-05\\
};
\addplot [color=black,solid,mark=x,mark options={solid},forget plot]
  table[row sep=crcr]{%
0.000373290282107307	-7.44563778247475e-05\\
0.00029883390428256	-6.69066435008107e-05\\
0.000231927260781749	-6.01323716310199e-05\\
0.000171794889150729	-5.40518117742342e-05\\
0.000117743077376495	-4.85927107311965e-05\\
6.91503666452986e-05	-4.36901030412838e-05\\
2.54602636040149e-05	-3.92864514200657e-05\\
-1.38261878160508e-05	-3.53300447725267e-05\\
-4.91562325885775e-05	-3.17748957046917e-05\\
-8.09311282932692e-05	-2.85797344423461e-05\\
-0.000109510862735615	-2.57076982350396e-05\\
-0.000135218560970655	-2.31257728636867e-05\\
-0.000158344333834342	-2.08043437527596e-05\\
-0.000179148677587101	-1.87169547916592e-05\\
};
\addplot [color=black,solid,mark=x,mark options={solid},forget plot]
  table[row sep=crcr]{%
0.000578710165056088	-9.5344987611527e-05\\
0.000483365177444561	-8.56387943049932e-05\\
0.000397726383139568	-7.69367283017925e-05\\
0.000320789654837775	-6.91317458155506e-05\\
0.000251657909022225	-6.21291165064377e-05\\
0.000189528792515787	-5.58442828624089e-05\\
0.000133684509653378	-5.02021087105354e-05\\
8.34824009428427e-05	-4.51355903503061e-05\\
3.83468105925366e-05	-4.05848923344421e-05\\
-2.23808174190555e-06	-3.64967242658877e-05\\
-3.87348060077932e-05	-3.28232825375152e-05\\
-7.15580885453084e-05	-2.95220459518086e-05\\
-0.000101080134497117	-2.65547274885353e-05\\
-0.000127634861985652	-2.38873117938619e-05\\
};
\addplot [color=black,solid,mark=x,mark options={solid},forget plot]
  table[row sep=crcr]{%
0.000461418072760261	-8.34072030174934e-05\\
0.000378010869742768	-7.49354134861692e-05\\
0.000303075456256598	-6.73365131690532e-05\\
0.000235738943087545	-6.05179764424581e-05\\
0.000175220966645087	-5.4398112135511e-05\\
0.000120822854509576	-4.89035500177572e-05\\
7.19193044918188e-05	-4.39693661480043e-05\\
2.79499383438145e-05	-3.95372637155969e-05\\
-1.15873253717824e-05	-3.55554374705956e-05\\
-4.71427628423779e-05	-3.1977430832625e-05\\
-7.91201936750029e-05	-2.87617746709899e-05\\
-0.000107881968345993	-2.58713513037148e-05\\
-0.000133753319649708	-2.32728833215416e-05\\
-0.000157026202971249	-2.09366565104994e-05\\
};
\addplot [color=black,solid,mark=x,mark options={solid},forget plot]
  table[row sep=crcr]{%
0.000217367556605077	-5.86593485140504e-05\\
0.000158708208091026	-5.27295105645025e-05\\
0.000105978697526524	-4.74053452644906e-05\\
5.85733522620332e-05	-4.26236286641168e-05\\
1.59497235979164e-05	-3.83283927878724e-05\\
-2.2378669189956e-05	-3.4469179231394e-05\\
-5.684784842135e-05	-3.10012699462513e-05\\
-8.78491183676013e-05	-2.78843622376553e-05\\
-0.000115733480605257	-2.50826127787188e-05\\
-0.000140816093383975	-2.25637667438228e-05\\
-0.000163379860127798	-2.02990229048744e-05\\
-0.000183678883032673	-1.82625210246146e-05\\
-0.000201941404057287	-1.64310749405035e-05\\
-0.000218372478997791	-1.47839251019278e-05\\
};
\addplot [color=black,solid,mark=x,mark options={solid},forget plot]
  table[row sep=crcr]{%
1.97675465916394e-06	-3.69210919248305e-05\\
-3.49443372656666e-05	-3.32046522426579e-05\\
-6.81489895083244e-05	-2.98648178730129e-05\\
-9.80138073813373e-05	-2.6862835379784e-05\\
-0.000124876642761121	-2.41642931112838e-05\\
-0.000149040935872405	-2.17380976391696e-05\\
-0.000170779033511575	-1.95566034555e-05\\
-0.000190335636967075	-1.75948647839582e-05\\
-0.000207930501751033	-1.5830623791673e-05\\
-0.000223761125542706	-1.42438415314309e-05\\
-0.000238004967074137	-1.28165689263602e-05\\
-0.000250821536000497	-1.15326893479452e-05\\
-0.000262354225348442	-1.03777117034408e-05\\
-0.000272731937051883	-9.33866385441042e-06\\
};
\addplot [color=black,solid,mark=x,mark options={solid},forget plot]
  table[row sep=crcr]{%
-0.000106166310200297	-2.60436963243647e-05\\
-0.000132210006524662	-2.34278393403114e-05\\
-0.000155637845864973	-2.10759860119969e-05\\
-0.00017671383187697	-1.89611787767132e-05\\
-0.000195675010653683	-1.70594241749116e-05\\
-0.000212734434828595	-1.53490382298248e-05\\
-0.000228083473058419	-1.38106874118979e-05\\
-0.000241894160470317	-1.24269365807075e-05\\
-0.000254321097051025	-1.11821832807186e-05\\
-0.000265503280331743	-1.00623951374506e-05\\
-0.000275565675469194	-9.05496711520149e-06\\
-0.000284620642584395	-8.14859467292736e-06\\
-0.000292769237257323	-7.33308712412556e-06\\
-0.000300102324381448	-6.59933209041233e-06\\
};
\addplot [color=black,solid,mark=x,mark options={solid},forget plot]
  table[row sep=crcr]{%
1.97675465916501e-06	-3.69210919248316e-05\\
-3.49443372656666e-05	-3.32046522426579e-05\\
-6.81489895083244e-05	-2.98648178730121e-05\\
-9.80138073813366e-05	-2.68628353797847e-05\\
-0.000124876642761121	-2.41642931112838e-05\\
-0.000149040935872405	-2.17380976391696e-05\\
-0.000170779033511575	-1.95566034555e-05\\
-0.000190335636967075	-1.75948647839582e-05\\
-0.000207930501751033	-1.5830623791673e-05\\
-0.000223761125542706	-1.42438415314309e-05\\
-0.000238004967074137	-1.28165689263602e-05\\
-0.000250821536000497	-1.15326893479452e-05\\
-0.000262354225348442	-1.03777117034408e-05\\
-0.000272731937051883	-9.33866385441042e-06\\
};
\addplot [color=black,solid,mark=x,mark options={solid},forget plot]
  table[row sep=crcr]{%
0.000217367604243283	-5.86593576913889e-05\\
0.000158708246551894	-5.27295410109615e-05\\
0.000105978705540932	-4.74053161031178e-05\\
5.85733894378144e-05	-4.26236654252706e-05\\
1.59497240125438e-05	-3.83283656159727e-05\\
-2.23786416034289e-05	-3.44692127703454e-05\\
-5.68478543737744e-05	-3.10012473789748e-05\\
-8.78491017527492e-05	-2.7884388615828e-05\\
-0.000115733490368577	-2.50825955374722e-05\\
-0.000140816085906049	-2.25637851415939e-05\\
-0.000163379871047643	-2.02990102011313e-05\\
-0.000183678881248775	-1.82625321474394e-05\\
-0.000201941413396214	-1.64310680232979e-05\\
-0.000218372481419512	-1.47839277266501e-05\\
};
\addplot [color=black,solid,mark=x,mark options={solid},forget plot]
  table[row sep=crcr]{%
0.000461418164064298	-8.34072163901125e-05\\
0.000378010947674185	-7.49354919964999e-05\\
0.000303075455677685	-6.73364530927442e-05\\
0.000235739002584941	-6.0518049502593e-05\\
0.000175220953082348	-5.43980621510351e-05\\
0.000120822890931313	-4.89036080328704e-05\\
7.19192828984425e-05	-4.39693270557186e-05\\
2.79499558427239e-05	-3.95373043165039e-05\\
-1.158734847378e-05	-3.55554083929148e-05\\
-4.71427568666948e-05	-3.19774553401757e-05\\
-7.91202122068705e-05	-2.87617542247662e-05\\
-0.000107881966431637	-2.5871362640644e-05\\
-0.000133753329072281	-2.3272877340421e-05\\
-0.000157026206412702	-2.09366526252073e-05\\
};
\addplot [color=black,solid,mark=x,mark options={solid},forget plot]
  table[row sep=crcr]{%
0.000418558451533631	-7.90520847908361e-05\\
0.000339506366742795	-7.1029342983542e-05\\
0.000268477023759253	-6.38318842115535e-05\\
0.0002046451395477	-5.73726608580559e-05\\
0.000147272478689644	-5.15743779410008e-05\\
9.56981007486431e-05	-4.63679165444788e-05\\
4.93301842041643e-05	-4.16919015316666e-05\\
7.63828267249773e-06	-3.74912338331978e-05\\
-2.98529511607001e-05	-3.37170227652779e-05\\
-6.3569973925978e-05	-3.03252189999585e-05\\
-9.38951929259365e-05	-2.72767553843909e-05\\
-0.000121171948310327	-2.45363359720394e-05\\
-0.000145708284282367	-2.20726589576827e-05\\
-0.00016778094324005	-1.98574004789403e-05\\
};
\addplot [color=black,solid,mark=x,mark options={solid},forget plot]
  table[row sep=crcr]{%
0.000155387317498033	-5.23940349002022e-05\\
0.00010299328259783	-4.71040463082288e-05\\
5.58892362896015e-05	-4.23530422361666e-05\\
1.35361940534349e-05	-3.80852929236416e-05\\
-2.45490988702067e-05	-3.42507390938835e-05\\
-5.87998379640902e-05	-3.08049658277739e-05\\
-8.96048037918641e-05	-2.77078973141848e-05\\
-0.000117312701106049	-2.49239983797238e-05\\
-0.000142236699485773	-2.242113322044e-05\\
-0.000164657832706213	-2.01707969105891e-05\\
-0.000184828629616802	-1.81471842380484e-05\\
-0.00020297581385485	-1.63273723786176e-05\\
-0.000219303186233468	-1.4690625121218e-05\\
-0.000233993811354686	-1.3218461052448e-05\\
};
\addplot [color=black,solid,mark=x,mark options={solid},forget plot]
  table[row sep=crcr]{%
-0.000106166310200297	-2.60436963243647e-05\\
-0.000132210006524662	-2.34278393403114e-05\\
-0.000155637845864973	-2.10759860119969e-05\\
-0.00017671383187697	-1.89611787767132e-05\\
-0.000195675010653683	-1.70594241749116e-05\\
-0.000212734434828595	-1.53490382298248e-05\\
-0.000228083473058419	-1.38106874118979e-05\\
-0.000241894160470317	-1.24269365807075e-05\\
-0.000254321097051025	-1.11821832807186e-05\\
-0.000265503280331743	-1.00623951374506e-05\\
-0.000275565675469194	-9.05496711520149e-06\\
-0.000284620642584395	-8.14859467292736e-06\\
-0.000292769237257323	-7.33308712412556e-06\\
-0.000300102324381448	-6.59933209041233e-06\\
};
\addplot [color=black,solid,mark=x,mark options={solid},forget plot]
  table[row sep=crcr]{%
-0.000366106547876789	0\\
-0.000366106547876789	0\\
};
\addplot [color=black,solid,mark=x,mark options={solid},forget plot]
  table[row sep=crcr]{%
-0.000106166310200297	-2.60436963243647e-05\\
-0.000132210006524662	-2.34278393403114e-05\\
-0.000155637845864973	-2.10759860119969e-05\\
-0.00017671383187697	-1.89611787767132e-05\\
-0.000195675010653683	-1.70594241749113e-05\\
-0.000212734434828594	-1.53490382298251e-05\\
-0.000228083473058419	-1.38106874118979e-05\\
-0.000241894160470317	-1.24269365807075e-05\\
-0.000254321097051025	-1.11821832807186e-05\\
-0.000265503280331743	-1.00623951374506e-05\\
-0.000275565675469194	-9.05496711520149e-06\\
-0.000284620642584395	-8.14859467292736e-06\\
-0.000292769237257323	-7.33308712412556e-06\\
-0.000300102324381448	-6.59933209041233e-06\\
};
\addplot [color=black,solid,mark=x,mark options={solid},forget plot]
  table[row sep=crcr]{%
0.000155387317498033	-5.23940349002022e-05\\
0.00010299328259783	-4.71040463082288e-05\\
5.58892362896015e-05	-4.23530422361666e-05\\
1.35361940534349e-05	-3.80852929236416e-05\\
-2.45490988702067e-05	-3.42507390938835e-05\\
-5.87998379640902e-05	-3.08049658277739e-05\\
-8.96048037918641e-05	-2.77078973141848e-05\\
-0.000117312701106049	-2.49239983797238e-05\\
-0.000142236699485773	-2.242113322044e-05\\
-0.000164657832706213	-2.01707969105891e-05\\
-0.000184828629616802	-1.81471842380484e-05\\
-0.00020297581385485	-1.63273723786176e-05\\
-0.000219303186233468	-1.4690625121218e-05\\
-0.000233993811354686	-1.3218461052448e-05\\
};
\addplot [color=black,solid,mark=x,mark options={solid},forget plot]
  table[row sep=crcr]{%
0.000418558451533631	-7.90520847908361e-05\\
0.000339506366742795	-7.1029342983542e-05\\
0.000268477023759253	-6.38318842115535e-05\\
0.0002046451395477	-5.73726608580559e-05\\
0.000147272478689644	-5.15743779410008e-05\\
9.56981007486431e-05	-4.63679165444788e-05\\
4.93301842041643e-05	-4.16919015316666e-05\\
7.63828267249773e-06	-3.74912338331978e-05\\
-2.98529511607001e-05	-3.37170227652779e-05\\
-6.3569973925978e-05	-3.03252189999585e-05\\
-9.38951929259365e-05	-2.72767553843909e-05\\
-0.000121171948310327	-2.45363359720394e-05\\
-0.000145708284282367	-2.20726589576827e-05\\
-0.00016778094324005	-1.98574004789403e-05\\
};
\addplot [color=black,solid,mark=x,mark options={solid},forget plot]
  table[row sep=crcr]{%
0.000461418164064298	-8.34072163901125e-05\\
0.000378010947674185	-7.49354919964999e-05\\
0.000303075455677685	-6.73364530927442e-05\\
0.000235739002584941	-6.0518049502593e-05\\
0.000175220953082348	-5.43980621510351e-05\\
0.000120822890931313	-4.89036080328704e-05\\
7.19192828984425e-05	-4.39693270557186e-05\\
2.79499558427239e-05	-3.95373043165039e-05\\
-1.158734847378e-05	-3.55554083929148e-05\\
-4.71427568666948e-05	-3.19774553401757e-05\\
-7.91202122068705e-05	-2.87617542247662e-05\\
-0.000107881966431637	-2.5871362640644e-05\\
-0.000133753329072281	-2.3272877340421e-05\\
-0.000157026206412702	-2.09366526252073e-05\\
};
\addplot [color=black,solid,mark=x,mark options={solid},forget plot]
  table[row sep=crcr]{%
0.000217367604243283	-5.86593576913889e-05\\
0.000158708246551894	-5.27295410109608e-05\\
0.000105978705540933	-4.74053161031186e-05\\
5.85733894378144e-05	-4.26236654252706e-05\\
1.59497240125438e-05	-3.83283656159727e-05\\
-2.23786416034289e-05	-3.44692127703454e-05\\
-5.68478543737744e-05	-3.10012473789748e-05\\
-8.78491017527492e-05	-2.7884388615828e-05\\
-0.000115733490368577	-2.50825955374722e-05\\
-0.000140816085906049	-2.25637851415939e-05\\
-0.000163379871047643	-2.02990102011313e-05\\
-0.000183678881248775	-1.82625321474394e-05\\
-0.000201941413396214	-1.64310680232979e-05\\
-0.000218372481419512	-1.47839277266501e-05\\
};
\addplot [color=black,solid,mark=x,mark options={solid},forget plot]
  table[row sep=crcr]{%
1.97675465916394e-06	-3.69210919248305e-05\\
-3.49443372656666e-05	-3.32046522426579e-05\\
-6.81489895083244e-05	-2.98648178730129e-05\\
-9.80138073813373e-05	-2.6862835379784e-05\\
-0.000124876642761121	-2.41642931112838e-05\\
-0.000149040935872405	-2.17380976391696e-05\\
-0.000170779033511575	-1.95566034555e-05\\
-0.000190335636967075	-1.75948647839582e-05\\
-0.000207930501751033	-1.5830623791673e-05\\
-0.000223761125542706	-1.42438415314309e-05\\
-0.000238004967074137	-1.28165689263602e-05\\
-0.000250821536000497	-1.15326893479452e-05\\
-0.000262354225348442	-1.03777117034408e-05\\
-0.000272731937051883	-9.33866385441042e-06\\
};
\addplot [color=black,solid,mark=x,mark options={solid},forget plot]
  table[row sep=crcr]{%
-0.000106166310200297	-2.60436963243647e-05\\
-0.000132210006524662	-2.34278393403114e-05\\
-0.000155637845864973	-2.10759860119969e-05\\
-0.00017671383187697	-1.89611787767132e-05\\
-0.000195675010653683	-1.70594241749113e-05\\
-0.000212734434828594	-1.53490382298251e-05\\
-0.000228083473058419	-1.38106874118979e-05\\
-0.000241894160470317	-1.24269365807075e-05\\
-0.000254321097051025	-1.11821832807186e-05\\
-0.000265503280331743	-1.00623951374506e-05\\
-0.000275565675469194	-9.05496711520149e-06\\
-0.000284620642584395	-8.14859467292736e-06\\
-0.000292769237257323	-7.33308712412556e-06\\
-0.000300102324381448	-6.59933209041233e-06\\
};
\addplot [color=black,solid,mark=x,mark options={solid},forget plot]
  table[row sep=crcr]{%
1.97675465916394e-06	-3.69210919248311e-05\\
-3.49443372656672e-05	-3.32046522426572e-05\\
-6.81489895083244e-05	-2.98648178730129e-05\\
-9.80138073813373e-05	-2.6862835379784e-05\\
-0.000124876642761121	-2.41642931112838e-05\\
-0.000149040935872405	-2.17380976391696e-05\\
-0.000170779033511575	-1.95566034555e-05\\
-0.000190335636967075	-1.75948647839582e-05\\
-0.000207930501751033	-1.5830623791673e-05\\
-0.000223761125542706	-1.42438415314309e-05\\
-0.000238004967074137	-1.28165689263602e-05\\
-0.000250821536000497	-1.15326893479451e-05\\
-0.000262354225348442	-1.03777117034409e-05\\
-0.000272731937051883	-9.33866385441042e-06\\
};
\addplot [color=black,solid,mark=x,mark options={solid},forget plot]
  table[row sep=crcr]{%
0.000217367556605077	-5.86593485140504e-05\\
0.000158708208091026	-5.27295105645025e-05\\
0.000105978697526524	-4.74053452644906e-05\\
5.85733522620332e-05	-4.26236286641168e-05\\
1.59497235979164e-05	-3.83283927878724e-05\\
-2.2378669189956e-05	-3.4469179231394e-05\\
-5.684784842135e-05	-3.10012699462513e-05\\
-8.78491183676013e-05	-2.78843622376553e-05\\
-0.000115733480605257	-2.50826127787188e-05\\
-0.000140816093383975	-2.25637667438228e-05\\
-0.000163379860127798	-2.02990229048744e-05\\
-0.000183678883032673	-1.82625210246146e-05\\
-0.000201941404057287	-1.64310749405035e-05\\
-0.000218372478997791	-1.47839251019278e-05\\
};
\addplot [color=black,solid,mark=x,mark options={solid},forget plot]
  table[row sep=crcr]{%
0.000461418072760261	-8.34072030174934e-05\\
0.000378010869742768	-7.49354134861679e-05\\
0.0003030754562566	-6.73365131690526e-05\\
0.000235738943087547	-6.05179764424599e-05\\
0.000175220966645087	-5.43981121355097e-05\\
0.000120822854509577	-4.89035500177586e-05\\
7.19193044918188e-05	-4.39693661480043e-05\\
2.79499383438145e-05	-3.95372637155969e-05\\
-1.15873253717824e-05	-3.55554374705956e-05\\
-4.71427628423779e-05	-3.1977430832625e-05\\
-7.91201936750029e-05	-2.87617746709899e-05\\
-0.000107881968345993	-2.58713513037148e-05\\
-0.000133753319649708	-2.32728833215416e-05\\
-0.000157026202971249	-2.09366565104994e-05\\
};
\addplot [color=black,solid,mark=x,mark options={solid},forget plot]
  table[row sep=crcr]{%
0.000578710165056088	-9.5344987611527e-05\\
0.000483365177444561	-8.56387943049932e-05\\
0.000397726383139568	-7.69367283017925e-05\\
0.000320789654837775	-6.91317458155506e-05\\
0.000251657909022225	-6.21291165064377e-05\\
0.000189528792515787	-5.58442828624089e-05\\
0.000133684509653378	-5.02021087105354e-05\\
8.34824009428427e-05	-4.51355903503061e-05\\
3.83468105925366e-05	-4.05848923344421e-05\\
-2.23808174190555e-06	-3.64967242658877e-05\\
-3.87348060077932e-05	-3.28232825375152e-05\\
-7.15580885453084e-05	-2.95220459518086e-05\\
-0.000101080134497117	-2.65547274885353e-05\\
-0.000127634861985652	-2.38873117938619e-05\\
};
\addplot [color=black,solid,mark=x,mark options={solid},forget plot]
  table[row sep=crcr]{%
0.000373290282107307	-7.44563778247475e-05\\
0.00029883390428256	-6.69066435008107e-05\\
0.000231927260781749	-6.01323716310199e-05\\
0.000171794889150729	-5.40518117742342e-05\\
0.000117743077376495	-4.85927107311965e-05\\
6.91503666452986e-05	-4.36901030412838e-05\\
2.54602636040149e-05	-3.92864514200657e-05\\
-1.38261878160508e-05	-3.53300447725267e-05\\
-4.91562325885775e-05	-3.17748957046917e-05\\
-8.09311282932692e-05	-2.85797344423461e-05\\
-0.000109510862735615	-2.57076982350396e-05\\
-0.000135218560970655	-2.31257728636867e-05\\
-0.000158344333834342	-2.08043437527596e-05\\
-0.000179148677587101	-1.87169547916592e-05\\
};
\addplot [color=black,solid,mark=x,mark options={solid},forget plot]
  table[row sep=crcr]{%
0.000217367556605077	-5.86593485140504e-05\\
0.000158708208091026	-5.27295105645025e-05\\
0.000105978697526524	-4.74053452644906e-05\\
5.85733522620332e-05	-4.26236286641168e-05\\
1.59497235979164e-05	-3.83283927878724e-05\\
-2.2378669189956e-05	-3.4469179231394e-05\\
-5.684784842135e-05	-3.10012699462513e-05\\
-8.78491183676013e-05	-2.78843622376553e-05\\
-0.000115733480605257	-2.50826127787188e-05\\
-0.000140816093383975	-2.25637667438228e-05\\
-0.000163379860127798	-2.02990229048744e-05\\
-0.000183678883032673	-1.82625210246146e-05\\
-0.000201941404057287	-1.64310749405035e-05\\
-0.000218372478997791	-1.47839251019278e-05\\
};
\addplot [color=black,solid,mark=x,mark options={solid},forget plot]
  table[row sep=crcr]{%
0.000155387317498033	-5.23940349002022e-05\\
0.00010299328259783	-4.71040463082288e-05\\
5.58892362896015e-05	-4.23530422361666e-05\\
1.35361940534349e-05	-3.80852929236416e-05\\
-2.45490988702067e-05	-3.42507390938835e-05\\
-5.87998379640902e-05	-3.08049658277739e-05\\
-8.96048037918641e-05	-2.77078973141848e-05\\
-0.000117312701106049	-2.49239983797238e-05\\
-0.000142236699485773	-2.242113322044e-05\\
-0.000164657832706213	-2.01707969105891e-05\\
-0.000184828629616802	-1.81471842380484e-05\\
-0.00020297581385485	-1.63273723786176e-05\\
-0.000219303186233468	-1.4690625121218e-05\\
-0.000233993811354686	-1.3218461052448e-05\\
};
\addplot [color=black,solid,mark=x,mark options={solid},forget plot]
  table[row sep=crcr]{%
0.000217367604243283	-5.86593576913889e-05\\
0.000158708246551894	-5.27295410109608e-05\\
0.000105978705540933	-4.74053161031186e-05\\
5.85733894378144e-05	-4.26236654252706e-05\\
1.59497240125438e-05	-3.83283656159727e-05\\
-2.23786416034289e-05	-3.44692127703454e-05\\
-5.68478543737744e-05	-3.10012473789748e-05\\
-8.78491017527492e-05	-2.7884388615828e-05\\
-0.000115733490368577	-2.50825955374722e-05\\
-0.000140816085906049	-2.25637851415939e-05\\
-0.000163379871047643	-2.02990102011313e-05\\
-0.000183678881248775	-1.82625321474394e-05\\
-0.000201941413396214	-1.64310680232979e-05\\
-0.000218372481419512	-1.47839277266501e-05\\
};
\addplot [color=black,solid,mark=x,mark options={solid},forget plot]
  table[row sep=crcr]{%
0.000373290282107307	-7.44563778247475e-05\\
0.00029883390428256	-6.69066435008107e-05\\
0.000231927260781749	-6.01323716310199e-05\\
0.000171794889150729	-5.40518117742342e-05\\
0.000117743077376495	-4.85927107311965e-05\\
6.91503666452986e-05	-4.36901030412838e-05\\
2.54602636040149e-05	-3.92864514200657e-05\\
-1.38261878160508e-05	-3.53300447725267e-05\\
-4.91562325885775e-05	-3.17748957046917e-05\\
-8.09311282932692e-05	-2.85797344423461e-05\\
-0.000109510862735615	-2.57076982350396e-05\\
-0.000135218560970655	-2.31257728636867e-05\\
-0.000158344333834342	-2.08043437527596e-05\\
-0.000179148677587101	-1.87169547916592e-05\\
};
\addplot [color=black,solid,mark=x,mark options={solid},forget plot]
  table[row sep=crcr]{%
0.000578710120461909	-9.53449764870871e-05\\
0.000483365143974822	-8.56387740841693e-05\\
0.000397726369890652	-7.6936755316778e-05\\
0.000320789614573874	-6.91317113596701e-05\\
0.000251657903214204	-6.21291440737971e-05\\
0.000189528759140407	-5.58442475839709e-05\\
0.000133684511556436	-5.02021331237162e-05\\
8.34823784327201e-05	-4.51355605262717e-05\\
3.83468179064484e-05	-4.05849118386165e-05\\
-2.23809393216808e-06	-3.64967020327142e-05\\
-3.87347959648823e-05	-3.28232967700756e-05\\
-7.15580927349579e-05	-2.95220312163224e-05\\
-0.00010108012395128	-2.65547365305851e-05\\
-0.000127634860481865	-2.38873058075043e-05\\
};
\addplot [color=black,solid,mark=x,mark options={solid},forget plot]
  table[row sep=crcr]{%
0.000747845590505265	-0.000112608897273819\\
0.000635236693231446	-0.000101108221398966\\
0.00053412847183248	-9.08043205732629e-05\\
0.000443324151259217	-8.15681524104606e-05\\
0.000361755998848757	-7.32861535026762e-05\\
0.00028846984534608	-6.58566786156994e-05\\
0.000222613166730381	-5.91900306238005e-05\\
0.000163423136106581	-5.32058973094306e-05\\
0.00011021723879715	-4.78330778977585e-05\\
6.23841608993914e-05	-4.30078709353128e-05\\
1.93762899640786e-05	-3.86735287587238e-05\\
-1.92972387946452e-05	-3.47793536868049e-05\\
-5.40765924814501e-05	-3.12799716485995e-05\\
-8.53565641300495e-05	-2.81349280575937e-05\\
};
\addplot [color=black,solid,mark=x,mark options={solid},forget plot]
  table[row sep=crcr]{%
0.000578710120461909	-9.53449764870871e-05\\
0.000483365143974822	-8.56387740841693e-05\\
0.000397726369890652	-7.6936755316778e-05\\
0.000320789614573874	-6.91317113596701e-05\\
0.000251657903214204	-6.21291440737971e-05\\
0.000189528759140407	-5.58442475839709e-05\\
0.000133684511556436	-5.02021331237162e-05\\
8.34823784327201e-05	-4.51355605262717e-05\\
3.83468179064484e-05	-4.05849118386165e-05\\
-2.23809393216808e-06	-3.64967020327142e-05\\
-3.87347959648823e-05	-3.28232967700756e-05\\
-7.15580927349579e-05	-2.95220312163224e-05\\
-0.00010108012395128	-2.65547365305851e-05\\
-0.000127634860481865	-2.38873058075043e-05\\
};
\addplot [color=black,solid,mark=x,mark options={solid},forget plot]
  table[row sep=crcr]{%
0.000461418072760261	-8.34072030174934e-05\\
0.000378010869742768	-7.49354134861692e-05\\
0.000303075456256598	-6.73365131690532e-05\\
0.000235738943087545	-6.05179764424581e-05\\
0.000175220966645087	-5.4398112135511e-05\\
0.000120822854509576	-4.89035500177572e-05\\
7.19193044918188e-05	-4.39693661480043e-05\\
2.79499383438145e-05	-3.95372637155969e-05\\
-1.15873253717824e-05	-3.55554374705956e-05\\
-4.71427628423779e-05	-3.1977430832625e-05\\
-7.91201936750029e-05	-2.87617746709899e-05\\
-0.000107881968345993	-2.58713513037148e-05\\
-0.000133753319649708	-2.32728833215416e-05\\
-0.000157026202971249	-2.09366565104994e-05\\
};
\addplot [color=black,solid,mark=x,mark options={solid},forget plot]
  table[row sep=crcr]{%
0.000418558451533631	-7.90520847908361e-05\\
0.000339506366742795	-7.1029342983542e-05\\
0.000268477023759253	-6.38318842115535e-05\\
0.0002046451395477	-5.73726608580559e-05\\
0.000147272478689644	-5.15743779410008e-05\\
9.56981007486431e-05	-4.63679165444788e-05\\
4.93301842041643e-05	-4.16919015316666e-05\\
7.63828267249773e-06	-3.74912338331978e-05\\
-2.98529511607001e-05	-3.37170227652779e-05\\
-6.3569973925978e-05	-3.03252189999585e-05\\
-9.38951929259365e-05	-2.72767553843909e-05\\
-0.000121171948310327	-2.45363359720394e-05\\
-0.000145708284282367	-2.20726589576827e-05\\
-0.00016778094324005	-1.98574004789403e-05\\
};
\addplot [color=black,solid,mark=x,mark options={solid},forget plot]
  table[row sep=crcr]{%
0.000461418164064298	-8.34072163901125e-05\\
0.000378010947674185	-7.49354919964999e-05\\
0.000303075455677685	-6.73364530927442e-05\\
0.000235739002584941	-6.0518049502593e-05\\
0.000175220953082348	-5.43980621510351e-05\\
0.000120822890931313	-4.89036080328704e-05\\
7.19192828984425e-05	-4.39693270557186e-05\\
2.79499558427239e-05	-3.95373043165039e-05\\
-1.158734847378e-05	-3.55554083929148e-05\\
-4.71427568666948e-05	-3.19774553401757e-05\\
-7.91202122068705e-05	-2.87617542247662e-05\\
-0.000107881966431637	-2.5871362640644e-05\\
-0.000133753329072281	-2.3272877340421e-05\\
-0.000157026206412702	-2.09366526252073e-05\\
};
\addplot [color=black,solid,mark=x,mark options={solid},forget plot]
  table[row sep=crcr]{%
0.000578710165056088	-9.5344987611527e-05\\
0.000483365177444561	-8.56387943049932e-05\\
0.000397726383139568	-7.69367283017925e-05\\
0.000320789654837775	-6.91317458155506e-05\\
0.000251657909022225	-6.21291165064377e-05\\
0.000189528792515787	-5.58442828624089e-05\\
0.000133684509653378	-5.02021087105354e-05\\
8.34824009428427e-05	-4.51355903503061e-05\\
3.83468105925366e-05	-4.05848923344421e-05\\
-2.23808174190555e-06	-3.64967242658877e-05\\
-3.87348060077932e-05	-3.28232825375152e-05\\
-7.15580885453084e-05	-2.95220459518086e-05\\
-0.000101080134497117	-2.65547274885353e-05\\
-0.000127634861985652	-2.38873117938619e-05\\
};
\addplot [color=black,solid,mark=x,mark options={solid},forget plot]
  table[row sep=crcr]{%
0.000747845590505265	-0.000112608897273819\\
0.000635236693231446	-0.000101108221398966\\
0.00053412847183248	-9.08043205732629e-05\\
0.000443324151259217	-8.15681524104606e-05\\
0.000361755998848757	-7.32861535026762e-05\\
0.00028846984534608	-6.58566786156994e-05\\
0.000222613166730381	-5.91900306238005e-05\\
0.000163423136106581	-5.32058973094306e-05\\
0.00011021723879715	-4.78330778977585e-05\\
6.23841608993914e-05	-4.30078709353128e-05\\
1.93762899640786e-05	-3.86735287587238e-05\\
-1.92972387946452e-05	-3.47793536868049e-05\\
-5.40765924814501e-05	-3.12799716485995e-05\\
};
\end{axis}
\end{tikzpicture}%

%% file: trs2.tex
%
%
\begin{tikzpicture}

\begin{axis}[%
width=0.951\fwidth,
height=\fheight,
at={(0\fwidth,0\fheight)},
scale only axis,
xmin=0,
xmax=15,
xlabel={iteration},
ymin=-0.1,
ymax=0.15,
ylabel={coordinate 1},
axis background/.style={fill=white}
]
\addplot [color=black,solid,mark=x,mark options={solid},forget plot]
  table[row sep=crcr]{%
1	0.124662260489931\\
2	0.105890899311022\\
3	0.0890366454466648\\
4	0.0739000022563686\\
5	0.0603029838443122\\
6	0.0480865347881866\\
7	0.0371085434369971\\
8	0.0272418502188878\\
9	0.0183726832221184\\
10	0.0103991393614174\\
11	0.00322993427720614\\
12	-0.00321675682774188\\
13	-0.00901430769121176\\
14	-0.0142285284117496\\
15	-0.018918485503041\\
};
\addplot [color=black,solid,mark=x,mark options={solid},forget plot]
  table[row sep=crcr]{%
1	0.0964681991313894\\
2	0.0805746486004475\\
3	0.0662990737769338\\
4	0.0534740914723364\\
5	0.0419501621821212\\
6	0.0315935374934579\\
7	0.0222845643237912\\
8	0.0139161144289562\\
9	0.00639222875916818\\
10	-0.000373077454289309\\
11	-0.00645690572743504\\
12	-0.0119283889450626\\
13	-0.016849571912998\\
14	-0.0212761171750733\\
15	-0.0252580170793553\\
};
\addplot [color=black,solid,mark=x,mark options={solid},forget plot]
  table[row sep=crcr]{%
1	0.0769161871028837\\
2	0.0630126055769977\\
3	0.0505212197323644\\
4	0.0392965570981376\\
5	0.0292084895248065\\
6	0.020140594386937\\
7	0.0119885983052695\\
8	0.00465912311336039\\
9	-0.00193155521964103\\
10	-0.00785847066739916\\
11	-0.0131889585622717\\
12	-0.0179834045586608\\
13	-0.0222960361896883\\
14	-0.0261755128279094\\
15	-0.0296655520505615\\
};
\addplot [color=black,solid,mark=x,mark options={solid},forget plot]
  table[row sep=crcr]{%
1	0.0697716793116274\\
2	0.0565940772617053\\
3	0.0447538276568898\\
4	0.0341133597873559\\
5	0.024549613362034\\
6	0.0159524127913342\\
7	0.00822310427627312\\
8	0.00127326495777197\\
9	-0.00497634327358142\\
10	-0.0105968086855259\\
11	-0.0156518767348446\\
12	-0.0201987805719745\\
13	-0.0242888696829477\\
14	-0.027968275693524\\
15	-0.0312784087779299\\
};
\addplot [color=black,solid,mark=x,mark options={solid},forget plot]
  table[row sep=crcr]{%
1	0.0769161718829402\\
2	0.0630125925862055\\
3	0.0505212198288667\\
4	0.0392965471802084\\
5	0.0292084917856498\\
6	0.020140588315609\\
7	0.0119886019047815\\
8	0.00465912019637643\\
9	-0.00193155136864922\\
10	-0.00785847166351677\\
11	-0.0131889554730986\\
12	-0.0179834048777747\\
13	-0.0222960346189907\\
14	-0.0261755122542358\\
15	-0.0296655521245474\\
};
\addplot [color=black,solid,mark=x,mark options={solid},forget plot]
  table[row sep=crcr]{%
1	0.0964681916977544\\
2	0.0805746430212031\\
3	0.0662990715684034\\
4	0.053474084760538\\
5	0.0419501612139522\\
6	0.0315935319299427\\
7	0.0222845646410219\\
8	0.0139161106766271\\
9	0.00639222997836208\\
10	-0.0003730794863474\\
11	-0.00645690405333013\\
12	-0.011928389643457\\
13	-0.0168495701550578\\
14	-0.0212761169243992\\
15	-0.0252580158307842\\
};
\addplot [color=black,solid,mark=x,mark options={solid},forget plot]
  table[row sep=crcr]{%
1	0.124662260489932\\
2	0.105890899311017\\
3	0.0890366454466659\\
4	0.0739000022563686\\
5	0.0603029838443122\\
6	0.0480865347881866\\
7	0.0371085434369971\\
8	0.0272418502188878\\
9	0.0183726832221184\\
10	0.0103991393614174\\
11	0.00322993427720614\\
12	-0.00321675682774188\\
13	-0.00901430769121176\\
14	-0.0142285284117496\\
15	-0.018918485503041\\
};
\addplot [color=black,solid,mark=x,mark options={solid},forget plot]
  table[row sep=crcr]{%
1	0.0964681916977544\\
2	0.0805746430212031\\
3	0.0662990715684034\\
4	0.053474084760538\\
5	0.0419501612139522\\
6	0.0315935319299427\\
7	0.0222845646410219\\
8	0.0139161106766271\\
9	0.00639222997836208\\
10	-0.0003730794863474\\
11	-0.00645690405333013\\
12	-0.011928389643457\\
13	-0.0168495701550578\\
14	-0.0212761169243992\\
15	-0.0252580158307842\\
};
\addplot [color=black,solid,mark=x,mark options={solid},forget plot]
  table[row sep=crcr]{%
1	0.0622256933479821\\
2	0.0498141735297609\\
3	0.0386611580857976\\
4	0.0286373811573559\\
5	0.0196272042907653\\
6	0.0115270332929105\\
7	0.00424410340027763\\
8	-0.00230475896226483\\
9	-0.00819410737919939\\
10	-0.0134908295577743\\
11	-0.0182549337325114\\
12	-0.0225402832948559\\
13	-0.0263952383248323\\
14	-0.0298632222952617\\
15	-0.0329832485726643\\
};
\addplot [color=black,solid,mark=x,mark options={solid},forget plot]
  table[row sep=crcr]{%
1	0.0362341334177562\\
2	0.026455900823304\\
3	0.0176661401287442\\
4	0.00976390210035426\\
5	0.00265874222545176\\
6	-0.00373041184490054\\
7	-0.00947626371033242\\
8	-0.0146440224366485\\
9	-0.0192922158088662\\
10	-0.0234733637602012\\
11	-0.0272346381418015\\
12	-0.0306183854413898\\
13	-0.0336626616511775\\
14	-0.0364016416064705\\
15	-0.0388660512021449\\
};
\addplot [color=black,solid,mark=x,mark options={solid},forget plot]
  table[row sep=crcr]{%
1	0.025902317933952\\
2	0.0171684844932351\\
3	0.00931646668961484\\
4	0.00225641839780251\\
5	-0.00409221662466866\\
6	-0.00980164997976756\\
7	-0.014936689516219\\
8	-0.0195554626379211\\
9	-0.0237101732064016\\
10	-0.027447668199342\\
11	-0.0308100429604989\\
12	-0.0338350912289709\\
13	-0.036556785619392\\
14	-0.0390056421198146\\
15	-0.0412090959556298\\
};
\addplot [color=black,solid,mark=x,mark options={solid},forget plot]
  table[row sep=crcr]{%
1	0.0362341254766966\\
2	0.0264558944120625\\
3	0.0176661387927809\\
4	0.00976389590333047\\
5	0.00265874215633537\\
6	-0.00373041644344183\\
7	-0.00947626271809193\\
8	-0.0146440252062644\\
9	-0.0192922141813677\\
10	-0.0234733650067354\\
11	-0.0272346363215159\\
12	-0.030618385738757\\
13	-0.0336626600944234\\
14	-0.0364016412027812\\
15	-0.0388660503609271\\
};
\addplot [color=black,solid,mark=x,mark options={solid},forget plot]
  table[row sep=crcr]{%
1	0.0622256933479821\\
2	0.0498141735297609\\
3	0.0386611580857976\\
4	0.0286373811573559\\
5	0.0196272042907653\\
6	0.0115270332929105\\
7	0.00424410340027763\\
8	-0.00230475896226483\\
9	-0.00819410737919939\\
10	-0.0134908295577743\\
11	-0.0182549337325114\\
12	-0.0225402832948559\\
13	-0.0263952383248323\\
14	-0.0298632222952617\\
15	-0.0329832485726643\\
};
\addplot [color=black,solid,mark=x,mark options={solid},forget plot]
  table[row sep=crcr]{%
1	0.0964681991313894\\
2	0.0805746486004475\\
3	0.0662990737769338\\
4	0.0534740914723364\\
5	0.0419501621821212\\
6	0.0315935374934579\\
7	0.0222845643237912\\
8	0.0139161144289562\\
9	0.00639222875916818\\
10	-0.000373077454289309\\
11	-0.00645690572743504\\
12	-0.0119283889450626\\
13	-0.016849571912998\\
14	-0.0212761171750733\\
15	-0.0252580170793553\\
};
\addplot [color=black,solid,mark=x,mark options={solid},forget plot]
  table[row sep=crcr]{%
1	0.0769161718829402\\
2	0.0630125925862055\\
3	0.0505212198288665\\
4	0.0392965471802081\\
5	0.0292084917856498\\
6	0.0201405883156088\\
7	0.0119886019047815\\
8	0.00465912019637643\\
9	-0.00193155136864922\\
10	-0.00785847166351677\\
11	-0.0131889554730986\\
12	-0.0179834048777747\\
13	-0.0222960346189907\\
14	-0.0261755122542358\\
15	-0.0296655521245474\\
};
\addplot [color=black,solid,mark=x,mark options={solid},forget plot]
  table[row sep=crcr]{%
1	0.0362341254766966\\
2	0.0264558944120625\\
3	0.0176661387927809\\
4	0.00976389590333047\\
5	0.00265874215633537\\
6	-0.00373041644344183\\
7	-0.00947626271809193\\
8	-0.0146440252062644\\
9	-0.0192922141813677\\
10	-0.0234733650067354\\
11	-0.0272346363215159\\
12	-0.030618385738757\\
13	-0.0336626600944234\\
14	-0.0364016412027812\\
15	-0.0388660503609271\\
};
\addplot [color=black,solid,mark=x,mark options={solid},forget plot]
  table[row sep=crcr]{%
1	0.000329515487386759\\
2	-0.00582505283198549\\
3	-0.0113601085438941\\
4	-0.0163384299413025\\
5	-0.0208163353058982\\
6	-0.0248444066626587\\
7	-0.0284680429117123\\
8	-0.0317280345800735\\
9	-0.0346610138539113\\
10	-0.0372999026457838\\
11	-0.0396742824721981\\
12	-0.0418107428249369\\
13	-0.0437331866313805\\
14	-0.0454631012234953\\
15	-0.0470198115402064\\
};
\addplot [color=black,solid,mark=x,mark options={solid},forget plot]
  table[row sep=crcr]{%
1	-0.0176974129225024\\
2	-0.0220387717491488\\
3	-0.0259440798069057\\
4	-0.0294573452345116\\
5	-0.032618082474751\\
6	-0.0354618063761839\\
7	-0.0380204171728105\\
8	-0.0403225922923079\\
9	-0.0423941028083615\\
10	-0.0442581189404153\\
11	-0.0459354717786262\\
12	-0.0474448912143687\\
13	-0.0488032227264353\\
14	-0.050025613055227\\
15	-0.0511256899515268\\
};
\addplot [color=black,solid,mark=x,mark options={solid},forget plot]
  table[row sep=crcr]{%
1	0.000329515487386938\\
2	-0.00582505283198549\\
3	-0.0113601085438941\\
4	-0.0163384299413024\\
5	-0.0208163353058982\\
6	-0.0248444066626587\\
7	-0.0284680429117123\\
8	-0.0317280345800735\\
9	-0.0346610138539113\\
10	-0.0372999026457838\\
11	-0.0396742824721981\\
12	-0.0418107428249369\\
13	-0.0437331866313805\\
14	-0.0454631012234953\\
15	-0.0470198115402064\\
};
\addplot [color=black,solid,mark=x,mark options={solid},forget plot]
  table[row sep=crcr]{%
1	0.0362341334177562\\
2	0.026455900823304\\
3	0.0176661401287442\\
4	0.00976390210035426\\
5	0.00265874222545176\\
6	-0.00373041184490054\\
7	-0.00947626371033242\\
8	-0.0146440224366485\\
9	-0.0192922158088662\\
10	-0.0234733637602012\\
11	-0.0272346381418015\\
12	-0.0306183854413898\\
13	-0.0336626616511775\\
14	-0.0364016416064705\\
15	-0.0388660512021449\\
};
\addplot [color=black,solid,mark=x,mark options={solid},forget plot]
  table[row sep=crcr]{%
1	0.0769161871028837\\
2	0.0630126055769977\\
3	0.0505212197323644\\
4	0.0392965570981376\\
5	0.0292084895248065\\
6	0.020140594386937\\
7	0.0119885983052695\\
8	0.00465912311336039\\
9	-0.00193155521964103\\
10	-0.00785847066739916\\
11	-0.0131889585622717\\
12	-0.0179834045586608\\
13	-0.0222960361896883\\
14	-0.0261755128279094\\
15	-0.0296655520505615\\
};
\addplot [color=black,solid,mark=x,mark options={solid},forget plot]
  table[row sep=crcr]{%
1	0.0697716793116274\\
2	0.0565940772617053\\
3	0.0447538276568898\\
4	0.0341133597873559\\
5	0.024549613362034\\
6	0.0159524127913342\\
7	0.00822310427627312\\
8	0.00127326495777197\\
9	-0.00497634327358142\\
10	-0.0105968086855259\\
11	-0.0156518767348446\\
12	-0.0201987805719745\\
13	-0.0242888696829477\\
14	-0.027968275693524\\
15	-0.0312784087779299\\
};
\addplot [color=black,solid,mark=x,mark options={solid},forget plot]
  table[row sep=crcr]{%
1	0.025902317933952\\
2	0.0171684844932351\\
3	0.00931646668961484\\
4	0.00225641839780251\\
5	-0.00409221662466866\\
6	-0.00980164997976756\\
7	-0.014936689516219\\
8	-0.0195554626379211\\
9	-0.0237101732064016\\
10	-0.027447668199342\\
11	-0.0308100429604989\\
12	-0.0338350912289709\\
13	-0.036556785619392\\
14	-0.0390056421198146\\
15	-0.0412090959556298\\
};
\addplot [color=black,solid,mark=x,mark options={solid},forget plot]
  table[row sep=crcr]{%
1	-0.0176974129225024\\
2	-0.0220387717491488\\
3	-0.0259440798069057\\
4	-0.0294573452345116\\
5	-0.032618082474751\\
6	-0.0354618063761839\\
7	-0.0380204171728105\\
8	-0.0403225922923079\\
9	-0.0423941028083615\\
10	-0.0442581189404153\\
11	-0.0459354717786262\\
12	-0.0474448912143687\\
13	-0.0488032227264353\\
14	-0.050025613055227\\
15	-0.0511256899515268\\
};
\addplot [color=black,solid,mark=x,mark options={solid},forget plot]
  table[row sep=crcr]{%
1	-0.061028199427706\\
2	-0.061028199427706\\
3	-0.061028199427706\\
4	-0.061028199427706\\
5	-0.061028199427706\\
6	-0.061028199427706\\
7	-0.061028199427706\\
8	-0.061028199427706\\
9	-0.061028199427706\\
10	-0.061028199427706\\
11	-0.061028199427706\\
12	-0.061028199427706\\
13	-0.061028199427706\\
14	-0.061028199427706\\
15	-0.061028199427706\\
};
\addplot [color=black,solid,mark=x,mark options={solid},forget plot]
  table[row sep=crcr]{%
1	-0.0176974129225024\\
2	-0.0220387717491488\\
3	-0.0259440798069057\\
4	-0.0294573452345116\\
5	-0.032618082474751\\
6	-0.0354618063761839\\
7	-0.0380204171728105\\
8	-0.0403225922923079\\
9	-0.0423941028083615\\
10	-0.0442581189404153\\
11	-0.0459354717786262\\
12	-0.0474448912143687\\
13	-0.0488032227264353\\
14	-0.050025613055227\\
15	-0.0511256899515268\\
};
\addplot [color=black,solid,mark=x,mark options={solid},forget plot]
  table[row sep=crcr]{%
1	0.025902317933952\\
2	0.0171684844932351\\
3	0.00931646668961484\\
4	0.00225641839780251\\
5	-0.00409221662466866\\
6	-0.00980164997976756\\
7	-0.014936689516219\\
8	-0.0195554626379211\\
9	-0.0237101732064016\\
10	-0.027447668199342\\
11	-0.0308100429604989\\
12	-0.0338350912289709\\
13	-0.036556785619392\\
14	-0.0390056421198146\\
15	-0.0412090959556298\\
};
\addplot [color=black,solid,mark=x,mark options={solid},forget plot]
  table[row sep=crcr]{%
1	0.0697716793116274\\
2	0.0565940772617053\\
3	0.0447538276568898\\
4	0.0341133597873559\\
5	0.024549613362034\\
6	0.0159524127913342\\
7	0.00822310427627312\\
8	0.00127326495777197\\
9	-0.00497634327358142\\
10	-0.0105968086855259\\
11	-0.0156518767348446\\
12	-0.0201987805719745\\
13	-0.0242888696829477\\
14	-0.027968275693524\\
15	-0.0312784087779299\\
};
\addplot [color=black,solid,mark=x,mark options={solid},forget plot]
  table[row sep=crcr]{%
1	0.0769161871028837\\
2	0.0630126055769977\\
3	0.0505212197323644\\
4	0.0392965570981376\\
5	0.0292084895248065\\
6	0.020140594386937\\
7	0.0119885983052695\\
8	0.00465912311336039\\
9	-0.00193155521964103\\
10	-0.00785847066739916\\
11	-0.0131889585622717\\
12	-0.0179834045586608\\
13	-0.0222960361896883\\
14	-0.0261755128279094\\
15	-0.0296655520505615\\
};
\addplot [color=black,solid,mark=x,mark options={solid},forget plot]
  table[row sep=crcr]{%
1	0.0362341334177562\\
2	0.026455900823304\\
3	0.0176661401287443\\
4	0.00976390210035426\\
5	0.00265874222545176\\
6	-0.00373041184490054\\
7	-0.00947626371033242\\
8	-0.0146440224366485\\
9	-0.0192922158088662\\
10	-0.0234733637602012\\
11	-0.0272346381418015\\
12	-0.0306183854413898\\
13	-0.0336626616511775\\
14	-0.0364016416064705\\
15	-0.0388660512021449\\
};
\addplot [color=black,solid,mark=x,mark options={solid},forget plot]
  table[row sep=crcr]{%
1	0.000329515487386759\\
2	-0.00582505283198549\\
3	-0.0113601085438941\\
4	-0.0163384299413025\\
5	-0.0208163353058982\\
6	-0.0248444066626587\\
7	-0.0284680429117123\\
8	-0.0317280345800735\\
9	-0.0346610138539113\\
10	-0.0372999026457838\\
11	-0.0396742824721981\\
12	-0.0418107428249369\\
13	-0.0437331866313805\\
14	-0.0454631012234953\\
15	-0.0470198115402064\\
};
\addplot [color=black,solid,mark=x,mark options={solid},forget plot]
  table[row sep=crcr]{%
1	-0.0176974129225024\\
2	-0.0220387717491488\\
3	-0.0259440798069057\\
4	-0.0294573452345116\\
5	-0.032618082474751\\
6	-0.0354618063761839\\
7	-0.0380204171728105\\
8	-0.0403225922923079\\
9	-0.0423941028083615\\
10	-0.0442581189404153\\
11	-0.0459354717786262\\
12	-0.0474448912143687\\
13	-0.0488032227264353\\
14	-0.050025613055227\\
15	-0.0511256899515268\\
};
\addplot [color=black,solid,mark=x,mark options={solid},forget plot]
  table[row sep=crcr]{%
1	0.000329515487386759\\
2	-0.00582505283198559\\
3	-0.0113601085438941\\
4	-0.0163384299413025\\
5	-0.0208163353058982\\
6	-0.0248444066626587\\
7	-0.0284680429117123\\
8	-0.0317280345800735\\
9	-0.0346610138539113\\
10	-0.0372999026457838\\
11	-0.0396742824721981\\
12	-0.0418107428249369\\
13	-0.0437331866313805\\
14	-0.0454631012234953\\
15	-0.0470198115402064\\
};
\addplot [color=black,solid,mark=x,mark options={solid},forget plot]
  table[row sep=crcr]{%
1	0.0362341254766966\\
2	0.0264558944120625\\
3	0.0176661387927809\\
4	0.00976389590333047\\
5	0.00265874215633537\\
6	-0.00373041644344183\\
7	-0.00947626271809193\\
8	-0.0146440252062644\\
9	-0.0192922141813677\\
10	-0.0234733650067354\\
11	-0.0272346363215159\\
12	-0.030618385738757\\
13	-0.0336626600944234\\
14	-0.0364016412027812\\
15	-0.0388660503609271\\
};
\addplot [color=black,solid,mark=x,mark options={solid},forget plot]
  table[row sep=crcr]{%
1	0.0769161718829402\\
2	0.0630125925862055\\
3	0.0505212198288667\\
4	0.0392965471802084\\
5	0.0292084917856498\\
6	0.020140588315609\\
7	0.0119886019047815\\
8	0.00465912019637643\\
9	-0.00193155136864922\\
10	-0.00785847166351677\\
11	-0.0131889554730986\\
12	-0.0179834048777747\\
13	-0.0222960346189907\\
14	-0.0261755122542358\\
15	-0.0296655521245474\\
};
\addplot [color=black,solid,mark=x,mark options={solid},forget plot]
  table[row sep=crcr]{%
1	0.0964681991313894\\
2	0.0805746486004475\\
3	0.0662990737769338\\
4	0.0534740914723364\\
5	0.0419501621821212\\
6	0.0315935374934579\\
7	0.0222845643237912\\
8	0.0139161144289562\\
9	0.00639222875916818\\
10	-0.000373077454289309\\
11	-0.00645690572743504\\
12	-0.0119283889450626\\
13	-0.016849571912998\\
14	-0.0212761171750733\\
15	-0.0252580170793553\\
};
\addplot [color=black,solid,mark=x,mark options={solid},forget plot]
  table[row sep=crcr]{%
1	0.0622256933479821\\
2	0.0498141735297609\\
3	0.0386611580857976\\
4	0.0286373811573559\\
5	0.0196272042907653\\
6	0.0115270332929105\\
7	0.00424410340027763\\
8	-0.00230475896226483\\
9	-0.00819410737919939\\
10	-0.0134908295577743\\
11	-0.0182549337325114\\
12	-0.0225402832948559\\
13	-0.0263952383248323\\
14	-0.0298632222952617\\
15	-0.0329832485726643\\
};
\addplot [color=black,solid,mark=x,mark options={solid},forget plot]
  table[row sep=crcr]{%
1	0.0362341254766966\\
2	0.0264558944120625\\
3	0.0176661387927809\\
4	0.00976389590333047\\
5	0.00265874215633537\\
6	-0.00373041644344183\\
7	-0.00947626271809193\\
8	-0.0146440252062644\\
9	-0.0192922141813677\\
10	-0.0234733650067354\\
11	-0.0272346363215159\\
12	-0.030618385738757\\
13	-0.0336626600944234\\
14	-0.0364016412027812\\
15	-0.0388660503609271\\
};
\addplot [color=black,solid,mark=x,mark options={solid},forget plot]
  table[row sep=crcr]{%
1	0.025902317933952\\
2	0.0171684844932351\\
3	0.00931646668961484\\
4	0.00225641839780251\\
5	-0.00409221662466866\\
6	-0.00980164997976756\\
7	-0.014936689516219\\
8	-0.0195554626379211\\
9	-0.0237101732064016\\
10	-0.027447668199342\\
11	-0.0308100429604989\\
12	-0.0338350912289709\\
13	-0.036556785619392\\
14	-0.0390056421198146\\
15	-0.0412090959556298\\
};
\addplot [color=black,solid,mark=x,mark options={solid},forget plot]
  table[row sep=crcr]{%
1	0.0362341334177562\\
2	0.026455900823304\\
3	0.0176661401287443\\
4	0.00976390210035426\\
5	0.00265874222545176\\
6	-0.00373041184490054\\
7	-0.00947626371033242\\
8	-0.0146440224366485\\
9	-0.0192922158088662\\
10	-0.0234733637602012\\
11	-0.0272346381418015\\
12	-0.0306183854413898\\
13	-0.0336626616511775\\
14	-0.0364016416064705\\
15	-0.0388660512021449\\
};
\addplot [color=black,solid,mark=x,mark options={solid},forget plot]
  table[row sep=crcr]{%
1	0.0622256933479821\\
2	0.0498141735297609\\
3	0.0386611580857976\\
4	0.0286373811573559\\
5	0.0196272042907653\\
6	0.0115270332929105\\
7	0.00424410340027763\\
8	-0.00230475896226483\\
9	-0.00819410737919939\\
10	-0.0134908295577743\\
11	-0.0182549337325114\\
12	-0.0225402832948559\\
13	-0.0263952383248323\\
14	-0.0298632222952617\\
15	-0.0329832485726643\\
};
\addplot [color=black,solid,mark=x,mark options={solid},forget plot]
  table[row sep=crcr]{%
1	0.0964681916977544\\
2	0.0805746430212031\\
3	0.0662990715684034\\
4	0.053474084760538\\
5	0.0419501612139522\\
6	0.0315935319299427\\
7	0.0222845646410219\\
8	0.0139161106766271\\
9	0.00639222997836208\\
10	-0.0003730794863474\\
11	-0.00645690405333013\\
12	-0.011928389643457\\
13	-0.0168495701550578\\
14	-0.0212761169243992\\
15	-0.0252580158307842\\
};
\addplot [color=black,solid,mark=x,mark options={solid},forget plot]
  table[row sep=crcr]{%
1	0.124662260489932\\
2	0.105890899311017\\
3	0.0890366454466659\\
4	0.0739000022563686\\
5	0.0603029838443122\\
6	0.0480865347881866\\
7	0.0371085434369971\\
8	0.0272418502188878\\
9	0.0183726832221184\\
10	0.0103991393614174\\
11	0.00322993427720614\\
12	-0.00321675682774188\\
13	-0.00901430769121176\\
14	-0.0142285284117496\\
15	-0.018918485503041\\
};
\addplot [color=black,solid,mark=x,mark options={solid},forget plot]
  table[row sep=crcr]{%
1	0.0964681916977544\\
2	0.0805746430212031\\
3	0.0662990715684034\\
4	0.053474084760538\\
5	0.0419501612139522\\
6	0.0315935319299427\\
7	0.0222845646410219\\
8	0.0139161106766271\\
9	0.00639222997836208\\
10	-0.0003730794863474\\
11	-0.00645690405333013\\
12	-0.011928389643457\\
13	-0.0168495701550578\\
14	-0.0212761169243992\\
15	-0.0252580158307842\\
};
\addplot [color=black,solid,mark=x,mark options={solid},forget plot]
  table[row sep=crcr]{%
1	0.0769161718829402\\
2	0.0630125925862055\\
3	0.0505212198288665\\
4	0.0392965471802081\\
5	0.0292084917856498\\
6	0.0201405883156088\\
7	0.0119886019047815\\
8	0.00465912019637643\\
9	-0.00193155136864922\\
10	-0.00785847166351677\\
11	-0.0131889554730986\\
12	-0.0179834048777747\\
13	-0.0222960346189907\\
14	-0.0261755122542358\\
15	-0.0296655521245474\\
};
\addplot [color=black,solid,mark=x,mark options={solid},forget plot]
  table[row sep=crcr]{%
1	0.0697716793116274\\
2	0.0565940772617053\\
3	0.0447538276568898\\
4	0.0341133597873559\\
5	0.024549613362034\\
6	0.0159524127913342\\
7	0.00822310427627312\\
8	0.00127326495777197\\
9	-0.00497634327358142\\
10	-0.0105968086855259\\
11	-0.0156518767348446\\
12	-0.0201987805719745\\
13	-0.0242888696829477\\
14	-0.027968275693524\\
15	-0.0312784087779299\\
};
\addplot [color=black,solid,mark=x,mark options={solid},forget plot]
  table[row sep=crcr]{%
1	0.0769161871028837\\
2	0.0630126055769977\\
3	0.0505212197323644\\
4	0.0392965570981376\\
5	0.0292084895248065\\
6	0.020140594386937\\
7	0.0119885983052695\\
8	0.00465912311336039\\
9	-0.00193155521964103\\
10	-0.00785847066739916\\
11	-0.0131889585622717\\
12	-0.0179834045586608\\
13	-0.0222960361896883\\
14	-0.0261755128279094\\
15	-0.0296655520505615\\
};
\addplot [color=black,solid,mark=x,mark options={solid},forget plot]
  table[row sep=crcr]{%
1	0.0964681991313894\\
2	0.0805746486004475\\
3	0.0662990737769338\\
4	0.0534740914723364\\
5	0.0419501621821212\\
6	0.0315935374934579\\
7	0.0222845643237912\\
8	0.0139161144289562\\
9	0.00639222875916818\\
10	-0.000373077454289309\\
11	-0.00645690572743504\\
12	-0.0119283889450626\\
13	-0.016849571912998\\
14	-0.0212761171750733\\
15	-0.0252580170793553\\
};
\addplot [color=black,solid,mark=x,mark options={solid},forget plot]
  table[row sep=crcr]{%
1	0.124662260489932\\
2	0.105890899311017\\
3	0.0890366454466659\\
4	0.0739000022563686\\
5	0.0603029838443122\\
6	0.0480865347881866\\
7	0.0371085434369971\\
8	0.0272418502188878\\
9	0.0183726832221184\\
10	0.0103991393614174\\
11	0.00322993427720614\\
12	-0.00321675682774188\\
13	-0.00901430769121176\\
14	-0.0142285284117496\\
};
\end{axis}
\end{tikzpicture}%

%% file: inis2.tex
%
%
\begin{tikzpicture}

\begin{axis}[%
width=0.951\fwidth,
height=\fheight,
at={(0\fwidth,0\fheight)},
scale only axis,
colormap/blackwhite,
xmin=-1,
xmax=1,
tick align=outside,
xlabel={x},
xmajorgrids,
ymin=-1,
ymax=1,
ylabel={y},
ymajorgrids,
zmin=-0.1,
zmax=0.15,
zlabel={coordinate 1},
zmajorgrids,
view={-37.5}{30},
axis background/.style={fill=white},
axis x line*=bottom,
axis y line*=left,
axis z line*=left
]

\addplot3[area legend,solid,table/row sep=crcr,patch,shader=interp,forget plot,patch table={%
127	84	63\\
145	138	135\\
81	102	61\\
63	42	85\\
13	36	57\\
1	6	9\\
9	12	7\\
7	30	9\\
135	132	139\\
99	61	102\\
103	99	120\\
13	34	54\\
106	127	63\\
106	130	127\\
45	85	42\\
45	24	43\\
16	34	13\\
13	57	16\\
9	30	4\\
4	1	9\\
6	1	27\\
43	24	27\\
33	7	12\\
33	36	13\\
142	145	135\\
111	139	132\\
111	108	133\\
84	127	87\\
127	130	87\\
133	108	87\\
87	130	133\\
103	120	117\\
138	145	117\\
118	139	146\\
112	134	137\\
133	130	112\\
112	90	133\\
72	77	95\\
116	119	137\\
100	96	103\\
58	61	20\\
83	59	62\\
62	100	83\\
96	100	62\\
31	54	8\\
8	54	34\\
32	56	53\\
53	77	72\\
53	56	77\\
63	85	66\\
66	106	63\\
66	85	88\\
19	57	60\\
19	60	81\\
81	40	19\\
19	16	57\\
19	40	16\\
81	61	78\\
78	40	81\\
78	61	58\\
85	45	64\\
64	45	43\\
43	48	64\\
23	2	26\\
1	4	22\\
22	27	1\\
43	27	22\\
13	54	10\\
10	33	13\\
7	33	10\\
10	2	7\\
135	139	114\\
114	142	135\\
114	139	118\\
133	90	136\\
136	111	133\\
139	111	136\\
136	146	139\\
145	142	124\\
124	117	145\\
103	117	124\\
124	100	103\\
80	59	83\\
83	100	104\\
143	122	146\\
143	136	90\\
146	136	143\\
113	70	92\\
92	116	113\\
95	116	92\\
113	116	91\\
91	69	113\\
134	69	91\\
137	134	91\\
91	116	137\\
11	8	34\\
34	14	11\\
29	11	32\\
8	11	29\\
31	8	29\\
29	26	31\\
77	56	76\\
76	80	77\\
59	80	76\\
35	56	32\\
35	76	56\\
109	112	130\\
134	112	109\\
130	106	109\\
106	66	109\\
88	134	109\\
109	66	88\\
34	16	37\\
37	14	34\\
16	40	37\\
40	78	37\\
58	14	37\\
37	78	58\\
67	48	86\\
86	89	67\\
7	2	28\\
28	2	23\\
30	7	28\\
23	44	28\\
86	48	46\\
46	48	43\\
43	22	46\\
25	4	30\\
25	22	4\\
30	28	25\\
25	28	44\\
44	46	25\\
25	46	22\\
5	54	31\\
5	10	54\\
2	10	5\\
31	26	5\\
5	26	2\\
77	80	98\\
98	95	77\\
121	114	118\\
118	104	121\\
142	114	121\\
121	124	142\\
100	124	121\\
121	104	100\\
137	119	140\\
140	112	137\\
90	112	140\\
119	122	115\\
122	143	115\\
115	140	119\\
115	143	90\\
90	140	115\\
79	62	59\\
96	62	79\\
79	41	96\\
82	41	20\\
61	99	82\\
82	20	61\\
96	41	82\\
82	99	103\\
103	96	82\\
32	11	55\\
55	11	14\\
55	35	32\\
52	53	50\\
26	29	52\\
32	53	52\\
52	29	32\\
50	47	52\\
52	47	26\\
75	55	14\\
35	55	75\\
85	64	128\\
88	85	128\\
110	67	89\\
113	69	110\\
110	70	113\\
110	89	70\\
107	64	48\\
48	67	107\\
107	128	64\\
67	110	107\\
107	110	128\\
86	46	65\\
65	46	44\\
65	89	86\\
26	47	51\\
23	26	51\\
51	44	23\\
71	53	72\\
50	53	71\\
71	92	70\\
72	95	71\\
95	92	71\\
94	122	119\\
95	98	94\\
119	116	94\\
94	116	95\\
125	104	118\\
125	93	104\\
122	93	125\\
118	146	125\\
146	122	125\\
104	93	97\\
97	80	83\\
83	104	97\\
17	14	58\\
17	75	14\\
58	20	17\\
131	110	69\\
131	134	88\\
131	69	134\\
88	128	131\\
128	110	131\\
49	65	44\\
44	51	49\\
49	51	47\\
101	97	93\\
101	93	122\\
101	98	80\\
80	97	101\\
122	94	101\\
101	94	98\\
75	17	38\\
41	79	38\\
20	41	38\\
38	17	20\\
38	79	59\\
35	75	38\\
59	76	38\\
76	35	38\\
89	65	68\\
65	49	68\\
70	89	68\\
68	71	70\\
68	49	47\\
68	47	50\\
50	71	68\\
0	3	24\\
24	21	0\\
18	39	15\\
39	36	15\\
39	60	57\\
57	36	39\\
108	129	126\\
126	105	108\\
102	123	120\\
120	99	102\\
45	42	21\\
21	24	45\\
3	6	27\\
27	24	3\\
12	15	33\\
33	15	36\\
111	132	129\\
129	108	111\\
87	105	84\\
87	108	105\\
141	123	144\\
141	120	123\\
117	141	138\\
120	141	117\\
}]
table[row sep=crcr, point meta=\thisrow{c}] {%
x	y	z	c\\
-1	-1	0.124662260489931	0.124662260489931\\
-0.214046063750064	-0.807601914223881	0.0480865347881866	0.0480865347881866\\
0.130338581691294	-0.476066518498105	0.00322993427720614	0.00322993427720614\\
-0.666666666666667	-1	0.0964681991313894	0.0964681991313894\\
-0.0437716624634227	-0.708676086054838	0.0315935374934579	0.0315935374934579\\
0.187960012643871	-0.374998886647866	-0.00645690572743504	-0.00645690572743504\\
-0.333333333333333	-1	0.0769161871028837	0.0769161871028837\\
0.126502847153089	-0.609750336434697	0.020140594386937	0.020140594386937\\
0.245581343461179	-0.273931284885139	-0.0131889585622717	-0.0131889585622717\\
0	-1	0.0697716793116274	0.0697716793116274\\
0.29677787712344	-0.510823950487219	0.0159524127913342	0.0159524127913342\\
0.303202549085561	-0.172864238757546	-0.0156518767348446	-0.0156518767348446\\
0.333333333333333	-1	0.0769161718829402	0.0769161718829402\\
0.467052588232566	-0.411898013908003	0.020140588315609	0.020140588315609\\
0.360823872895834	-0.0717964517051722	-0.0131889554730986	-0.0131889554730986\\
0.666666666666667	-1	0.0964681916977544	0.0964681916977544\\
0.637327241808151	-0.312971975200472	0.0315935319299427	0.0315935319299427\\
0.418445214093429	0.0292710292198177	-0.00645690405333013	-0.00645690405333013\\
1	-1	0.124662260489932	0.124662260489932\\
0.807601914223881	-0.214046063750064	0.0480865347881866	0.0480865347881866\\
0.476066518498105	0.130338581691294	0.00322993427720614	0.00322993427720614\\
-1	-0.666666666666667	0.0964681916977544	0.0964681916977544\\
-0.312971975200472	-0.637327241808151	0.0315935319299427	0.0315935319299427\\
0.0292710292198177	-0.418445214093429	-0.00645690405333013	-0.00645690405333013\\
-0.666666666666667	-0.666666666666667	0.0622256933479821	0.0622256933479821\\
-0.142697375833376	-0.538401276149254	0.0115270332929105	0.0115270332929105\\
0.086892387794196	-0.317377678998736	-0.0182549337325114	-0.0182549337325114\\
-0.333333333333333	-0.666666666666667	0.0362341334177562	0.0362341334177562\\
0.0275770439388151	-0.439475473259259	-0.00373041184490054	-0.00373041184490054\\
0.144513798581851	-0.216310047311358	-0.0272346381418015	-0.0272346381418015\\
0	-0.666666666666667	0.025902317933952	0.025902317933952\\
0.197851728798922	-0.340549656596234	-0.00980164997976756	-0.00980164997976756\\
0.20213505600989	-0.115242546172719	-0.0308100429604989	-0.0308100429604989\\
0.333333333333333	-0.666666666666667	0.0362341254766966	0.0362341254766966\\
0.368126612251134	-0.241623311031687	-0.00373041644344183	-0.00373041644344183\\
0.259756364766711	-0.0141751523334529	-0.0272346363215159	-0.0272346363215159\\
0.666666666666667	-0.666666666666667	0.0622256933479821	0.0622256933479821\\
0.538401276149254	-0.142697375833376	0.0115270332929105	0.0115270332929105\\
0.317377678998736	0.086892387794196	-0.0182549337325114	-0.0182549337325114\\
1	-0.666666666666667	0.0964681991313894	0.0964681991313894\\
0.708676086054838	-0.0437716624634227	0.0315935374934579	0.0315935374934579\\
0.374998886647866	0.187960012643871	-0.00645690572743504	-0.00645690572743504\\
-1	-0.333333333333333	0.0769161718829402	0.0769161718829402\\
-0.411898013908003	-0.467052588232566	0.0201405883156088	0.0201405883156088\\
-0.0717964517051721	-0.360823872895834	-0.0131889554730986	-0.0131889554730986\\
-0.666666666666667	-0.333333333333333	0.0362341254766966	0.0362341254766966\\
-0.241623311031687	-0.368126612251134	-0.00373041644344183	-0.00373041644344183\\
-0.0141751523334529	-0.259756364766711	-0.0272346363215159	-0.0272346363215159\\
-0.333333333333333	-0.333333333333333	0.000329515487386759	0.000329515487386759\\
-0.0713486879166881	-0.269200638074627	-0.0248444066626587	-0.0248444066626587\\
0.043446193897098	-0.158688839499368	-0.0396742824721981	-0.0396742824721981\\
0	-0.333333333333333	-0.0176974129225024	-0.0176974129225024\\
0.0989259634357864	-0.170274686489111	-0.0354618063761839	-0.0354618063761839\\
0.101067518641138	-0.0576213082003527	-0.0459354717786262	-0.0459354717786262\\
0.333333333333333	-0.333333333333333	0.000329515487386938	0.000329515487386938\\
0.269200638074627	-0.0713486879166882	-0.0248444066626587	-0.0248444066626587\\
0.158688839499368	0.043446193897098	-0.0396742824721981	-0.0396742824721981\\
0.666666666666667	-0.333333333333333	0.0362341334177562	0.0362341334177562\\
0.439475473259259	0.0275770439388151	-0.00373041184490054	-0.00373041184490054\\
0.216310047311358	0.144513798581851	-0.0272346381418015	-0.0272346381418015\\
1	-0.333333333333333	0.0769161871028837	0.0769161871028837\\
0.609750336434697	0.126502847153089	0.020140594386937	0.020140594386937\\
0.273931284885139	0.245581343461179	-0.0131889585622717	-0.0131889585622717\\
-1	0	0.0697716793116274	0.0697716793116274\\
-0.510823950487219	-0.29677787712344	0.0159524127913342	0.0159524127913342\\
-0.172864238757546	-0.303202549085561	-0.0156518767348446	-0.0156518767348446\\
-0.666666666666667	0	0.025902317933952	0.025902317933952\\
-0.340549656596234	-0.197851728798922	-0.00980164997976756	-0.00980164997976756\\
-0.115242546172719	-0.20213505600989	-0.0308100429604989	-0.0308100429604989\\
-0.333333333333333	0	-0.0176974129225024	-0.0176974129225024\\
-0.170274686489111	-0.0989259634357864	-0.0354618063761839	-0.0354618063761839\\
-0.0576213082003527	-0.101067518641138	-0.0459354717786262	-0.0459354717786262\\
0	0	-0.061028199427706	-0.061028199427706\\
0	0	-0.061028199427706	-0.061028199427706\\
0	0	-0.061028199427706	-0.061028199427706\\
0.333333333333333	0	-0.0176974129225024	-0.0176974129225024\\
0.170274686489112	0.0989259634357863	-0.0354618063761839	-0.0354618063761839\\
0.0576213082003528	0.101067518641138	-0.0459354717786262	-0.0459354717786262\\
0.666666666666667	0	0.025902317933952	0.025902317933952\\
0.340549656596234	0.197851728798922	-0.00980164997976756	-0.00980164997976756\\
0.115242546172719	0.20213505600989	-0.0308100429604989	-0.0308100429604989\\
1	0	0.0697716793116274	0.0697716793116274\\
0.510823950487219	0.29677787712344	0.0159524127913342	0.0159524127913342\\
0.172864238757546	0.303202549085561	-0.0156518767348446	-0.0156518767348446\\
-1	0.333333333333333	0.0769161871028837	0.0769161871028837\\
-0.609750336434698	-0.126502847153089	0.020140594386937	0.020140594386937\\
-0.273931284885139	-0.245581343461179	-0.0131889585622717	-0.0131889585622717\\
-0.666666666666667	0.333333333333333	0.0362341334177562	0.0362341334177562\\
-0.439475473259259	-0.0275770439388151	-0.00373041184490054	-0.00373041184490054\\
-0.216310047311358	-0.144513798581851	-0.0272346381418015	-0.0272346381418015\\
-0.333333333333333	0.333333333333333	0.000329515487386759	0.000329515487386759\\
-0.269200638074627	0.071348687916688	-0.0248444066626587	-0.0248444066626587\\
-0.158688839499368	-0.0434461938970981	-0.0396742824721981	-0.0396742824721981\\
0	0.333333333333333	-0.0176974129225024	-0.0176974129225024\\
-0.0989259634357863	0.170274686489112	-0.0354618063761839	-0.0354618063761839\\
-0.101067518641138	0.0576213082003528	-0.0459354717786262	-0.0459354717786262\\
0.333333333333333	0.333333333333333	0.000329515487386759	0.000329515487386759\\
0.0713486879166881	0.269200638074627	-0.0248444066626587	-0.0248444066626587\\
-0.043446193897098	0.158688839499368	-0.0396742824721981	-0.0396742824721981\\
0.666666666666667	0.333333333333333	0.0362341254766966	0.0362341254766966\\
0.241623311031687	0.368126612251134	-0.00373041644344183	-0.00373041644344183\\
0.0141751523334529	0.259756364766711	-0.0272346363215159	-0.0272346363215159\\
1	0.333333333333333	0.0769161718829402	0.0769161718829402\\
0.411898013908003	0.467052588232566	0.020140588315609	0.020140588315609\\
0.0717964517051722	0.360823872895834	-0.0131889554730986	-0.0131889554730986\\
-1	0.666666666666667	0.0964681991313894	0.0964681991313894\\
-0.708676086054838	0.0437716624634227	0.0315935374934579	0.0315935374934579\\
-0.374998886647866	-0.187960012643871	-0.00645690572743504	-0.00645690572743504\\
-0.666666666666667	0.666666666666667	0.0622256933479821	0.0622256933479821\\
-0.538401276149254	0.142697375833376	0.0115270332929105	0.0115270332929105\\
-0.317377678998736	-0.086892387794196	-0.0182549337325114	-0.0182549337325114\\
-0.333333333333333	0.666666666666667	0.0362341254766966	0.0362341254766966\\
-0.368126612251134	0.241623311031687	-0.00373041644344183	-0.00373041644344183\\
-0.259756364766711	0.0141751523334529	-0.0272346363215159	-0.0272346363215159\\
0	0.666666666666667	0.025902317933952	0.025902317933952\\
-0.197851728798922	0.340549656596234	-0.00980164997976756	-0.00980164997976756\\
-0.20213505600989	0.115242546172719	-0.0308100429604989	-0.0308100429604989\\
0.333333333333333	0.666666666666667	0.0362341334177562	0.0362341334177562\\
-0.0275770439388151	0.439475473259259	-0.00373041184490054	-0.00373041184490054\\
-0.144513798581851	0.216310047311358	-0.0272346381418015	-0.0272346381418015\\
0.666666666666667	0.666666666666667	0.0622256933479821	0.0622256933479821\\
0.142697375833376	0.538401276149254	0.0115270332929105	0.0115270332929105\\
-0.086892387794196	0.317377678998736	-0.0182549337325114	-0.0182549337325114\\
1	0.666666666666667	0.0964681916977544	0.0964681916977544\\
0.312971975200472	0.637327241808151	0.0315935319299427	0.0315935319299427\\
-0.0292710292198177	0.418445214093429	-0.00645690405333013	-0.00645690405333013\\
-1	1	0.124662260489932	0.124662260489932\\
-0.807601914223881	0.214046063750064	0.0480865347881866	0.0480865347881866\\
-0.476066518498105	-0.130338581691294	0.00322993427720614	0.00322993427720614\\
-0.666666666666667	1	0.0964681916977544	0.0964681916977544\\
-0.637327241808151	0.312971975200472	0.0315935319299427	0.0315935319299427\\
-0.418445214093429	-0.0292710292198177	-0.00645690405333013	-0.00645690405333013\\
-0.333333333333333	1	0.0769161718829402	0.0769161718829402\\
-0.467052588232566	0.411898013908003	0.0201405883156088	0.0201405883156088\\
-0.360823872895834	0.0717964517051721	-0.0131889554730986	-0.0131889554730986\\
0	1	0.0697716793116274	0.0697716793116274\\
-0.29677787712344	0.510823950487219	0.0159524127913342	0.0159524127913342\\
-0.303202549085561	0.172864238757546	-0.0156518767348446	-0.0156518767348446\\
0.333333333333333	1	0.0769161871028837	0.0769161871028837\\
-0.126502847153089	0.609750336434698	0.020140594386937	0.020140594386937\\
-0.245581343461179	0.273931284885139	-0.0131889585622717	-0.0131889585622717\\
0.666666666666667	1	0.0964681991313894	0.0964681991313894\\
0.0437716624634227	0.708676086054838	0.0315935374934579	0.0315935374934579\\
-0.187960012643871	0.374998886647866	-0.00645690572743504	-0.00645690572743504\\
1	1	0.124662260489932	0.124662260489932\\
0.214046063750064	0.807601914223881	0.0480865347881866	0.0480865347881866\\
-0.130338581691294	0.476066518498105	0.00322993427720614	0.00322993427720614\\
};
\end{axis}
\end{tikzpicture}%

%% file: mucos2.tex
%
%
\begin{tikzpicture}

\begin{axis}[%
width=0.75\fwidth,
height=\fheight,
at={(0\fwidth,0\fheight)},
scale only axis,
xmin=-0.0004,
xmax=0.0008,
xlabel={coordinate 1},
ymin=0,
ymax=1.5,
ylabel={observer coordinate 1},
axis background/.style={fill=white}
]
\addplot [color=black,solid,mark=x,mark options={solid},forget plot]
  table[row sep=crcr]{%
0.000747845590505263	1.4142135623731\\
0.000635236693231472	1.27291613969195\\
0.000534128471832474	1.14573672105966\\
0.000443324151259217	1.03126298896983\\
0.000361755998848757	0.928227461592727\\
0.00028846984534608	0.835485828284941\\
0.000222613166730381	0.752010643579825\\
0.000163423136106581	0.676875430935378\\
0.00011021723879715	0.609247224834565\\
6.23841608993914e-05	0.548375965027024\\
1.93762899640786e-05	0.493586340889013\\
-1.92972387946452e-05	0.444271102238922\\
-5.40765924814501e-05	0.399882787028542\\
-8.53565641300495e-05	0.359929719098952\\
-0.000113491492187643	0.323968134020426\\
};
\addplot [color=black,solid,mark=x,mark options={solid},forget plot]
  table[row sep=crcr]{%
0.000578710165056088	1.20185042515466\\
0.000483365177444561	1.08177088021137\\
0.000397726383139568	0.973688844563822\\
0.000320789654837775	0.876405207048611\\
0.000251657909022225	0.788841529789667\\
0.000189528792515787	0.710026586390127\\
0.000133684509653378	0.639086052674395\\
8.34824009428427e-05	0.575233626037156\\
3.83468105925366e-05	0.517760501926382\\
-2.23808174190555e-06	0.466030036492987\\
-3.87348060077932e-05	0.419467676156606\\
-7.15580885453084e-05	0.377557879659535\\
-0.000101080134497117	0.339834989427531\\
-0.000127634861985652	0.305881459711535\\
-0.000151522173779514	0.275319956648929\\
};
\addplot [color=black,solid,mark=x,mark options={solid},forget plot]
  table[row sep=crcr]{%
0.000461418164064298	1.05409255338946\\
0.000378010947674185	0.948776138808791\\
0.000303075455677685	0.853981406060374\\
0.000235739002584941	0.768658437746888\\
0.000175220953082348	0.691859641109652\\
0.000120822890931313	0.622734649044088\\
7.19192828984425e-05	0.56051544979122\\
2.79499558427239e-05	0.504513333105839\\
-1.158734847378e-05	0.45410594467647\\
-4.71427568666948e-05	0.408735399005072\\
-7.91202122068705e-05	0.367897465464252\\
-0.000107881966431637	0.331140151879988\\
-0.000133753329072281	0.298055011638571\\
-0.000157026206412702	0.268275754993234\\
-0.000177962859037909	0.241471612237471\\
};
\addplot [color=black,solid,mark=x,mark options={solid},forget plot]
  table[row sep=crcr]{%
0.000418558451533631	1\\
0.000339506366742795	0.900087675089919\\
0.000268477023759253	0.810158131282656\\
0.0002046451395477	0.729213107024132\\
0.000147272478689644	0.656355973169123\\
9.56981007486431e-05	0.590777637306343\\
4.93301842041643e-05	0.531751930797434\\
7.63828267249773e-06	0.478623079771031\\
-2.98529511607001e-05	0.430802976202164\\
-6.3569973925978e-05	0.387760204026884\\
-9.38951929259365e-05	0.349018381769797\\
-0.000121171948310327	0.314146946063971\\
-0.000145708284282367	0.282759948454458\\
-0.00016778094324005	0.254508605332817\\
-0.00018763834371899	0.229080159632866\\
};
\addplot [color=black,solid,mark=x,mark options={solid},forget plot]
  table[row sep=crcr]{%
0.000461418072760261	1.05409255338946\\
0.000378010869742768	0.948775672669272\\
0.0003030754562566	0.853981730869042\\
0.000235738943087547	0.768658048672053\\
0.000175220966645087	0.691859900901138\\
0.000120822854509577	0.622734368760948\\
7.19193044918188e-05	0.560515639615951\\
2.79499383438145e-05	0.504513158623656\\
-1.15873253717824e-05	0.454106069944395\\
-4.71427628423779e-05	0.408735311675825\\
-7.91201936750029e-05	0.367897536997738\\
-0.000107881968345993	0.33114012834833\\
-0.000133753319649708	0.298055031609297\\
-0.000157026202971249	0.268275773009221\\
-0.000177962859481749	0.241471590071145\\
};
\addplot [color=black,solid,mark=x,mark options={solid},forget plot]
  table[row sep=crcr]{%
0.000578710120461909	1.20185042515466\\
0.000483365143974822	1.08177069696277\\
0.000397726369890652	0.973689007115459\\
0.000320789614573874	0.8764049853992\\
0.000251657903214204	0.78884168287072\\
0.000189528759140407	0.710026387123514\\
0.000133684511556436	0.639086182201942\\
8.34823784327201e-05	0.575233473034235\\
3.83468179064484e-05	0.5177606019741\\
-2.23809393216808e-06	0.466029933467533\\
-3.87347959648823e-05	0.41946774649463\\
-7.15580927349579e-05	0.377557820459789\\
-0.00010108012395128	0.339835033447347\\
-0.000127634860481865	0.305881434399356\\
-0.00015152216628937	0.275319979728245\\
};
\addplot [color=black,solid,mark=x,mark options={solid},forget plot]
  table[row sep=crcr]{%
0.000747845590505265	1.4142135623731\\
0.000635236693231446	1.27291613969195\\
0.00053412847183248	1.14573672105966\\
0.000443324151259217	1.03126298896983\\
0.000361755998848757	0.928227461592727\\
0.00028846984534608	0.835485828284941\\
0.000222613166730381	0.752010643579825\\
0.000163423136106581	0.676875430935378\\
0.00011021723879715	0.609247224834565\\
6.23841608993914e-05	0.548375965027024\\
1.93762899640786e-05	0.493586340889013\\
-1.92972387946452e-05	0.444271102238922\\
-5.40765924814501e-05	0.399882787028542\\
-8.53565641300495e-05	0.359929719098952\\
-0.000113491492187643	0.323968134020426\\
};
\addplot [color=black,solid,mark=x,mark options={solid},forget plot]
  table[row sep=crcr]{%
0.000578710120461909	1.20185042515466\\
0.000483365143974822	1.08177069696277\\
0.000397726369890652	0.973689007115459\\
0.000320789614573874	0.8764049853992\\
0.000251657903214204	0.78884168287072\\
0.000189528759140407	0.710026387123514\\
0.000133684511556436	0.639086182201942\\
8.34823784327201e-05	0.575233473034235\\
3.83468179064484e-05	0.5177606019741\\
-2.23809393216808e-06	0.466029933467533\\
-3.87347959648823e-05	0.41946774649463\\
-7.15580927349579e-05	0.377557820459789\\
-0.00010108012395128	0.339835033447347\\
-0.000127634860481865	0.305881434399356\\
-0.00015152216628937	0.275319979728245\\
};
\addplot [color=black,solid,mark=x,mark options={solid},forget plot]
  table[row sep=crcr]{%
0.000373290282107307	0.942809041582064\\
0.00029883390428256	0.848610759794632\\
0.000231927260781749	0.763824480706443\\
0.000171794889150729	0.68750865931322\\
0.000117743077376495	0.618818307728485\\
6.91503666452986e-05	0.55699055218996\\
2.54602636040149e-05	0.501340429053216\\
-1.38261878160508e-05	0.451250287290252\\
-4.91562325885775e-05	0.406164816556377\\
-8.09311282932692e-05	0.365583976684683\\
-0.000109510862735615	0.329057560592674\\
-0.000135218560970655	0.296180734825948\\
-0.000158344333834342	0.266588524685694\\
-0.000179148677587101	0.239953146065968\\
-0.00019786563237876	0.215978756013617\\
};
\addplot [color=black,solid,mark=x,mark options={solid},forget plot]
  table[row sep=crcr]{%
0.000217367604243283	0.74535599249993\\
0.000158708246551894	0.6708859236935\\
0.000105978705540932	0.603856120342304\\
5.85733894378144e-05	0.543523517288255\\
1.59497240125438e-05	0.489218696589278\\
-2.23786416034289e-05	0.440339851647399\\
-5.68478543737744e-05	0.396344307220589\\
-8.78491017527492e-05	0.356744783642232\\
-0.000115733490368577	0.321101408282976\\
-0.000140816085906049	0.289019574588623\\
-0.000163379871047643	0.260142796456865\\
-0.000183678881248775	0.234151459545296\\
-0.000201941413396214	0.210756708585459\\
-0.000218372481419512	0.189699623022409\\
-0.000233156409146162	0.170746190893769\\
};
\addplot [color=black,solid,mark=x,mark options={solid},forget plot]
  table[row sep=crcr]{%
0.000155387317498033	0.666666666666667\\
0.00010299328259783	0.600058666010219\\
5.58892362896015e-05	0.540105204092366\\
1.35361940534349e-05	0.486142300274515\\
-2.45490988702067e-05	0.437570422884226\\
-5.87998379640902e-05	0.393851971172591\\
-8.96048037918641e-05	0.354501079237506\\
-0.000117312701106049	0.319082232521777\\
-0.000142236699485773	0.287201813649878\\
-0.000164657832706213	0.258506941555869\\
-0.000184828629616802	0.232678802894661\\
-0.00020297581385485	0.20943139644561\\
-0.000219303186233468	0.188506564579455\\
-0.000233993811354686	0.169672462295914\\
-0.000247212272407134	0.152720079612198\\
};
\addplot [color=black,solid,mark=x,mark options={solid},forget plot]
  table[row sep=crcr]{%
0.000217367556605077	0.74535599249993\\
0.000158708208091026	0.670885706420798\\
0.000105978697526524	0.603856287849979\\
5.85733522620332e-05	0.543523298783245\\
1.59497235979164e-05	0.489218843838872\\
-2.2378669189956e-05	0.440339672390999\\
-5.684784842135e-05	0.396344424945242\\
-8.78491183676013e-05	0.356744655773149\\
-0.000115733480605257	0.321101494233031\\
-0.000140816093383975	0.28901949554136\\
-0.000163379860127798	0.260142853025973\\
-0.000183678883032673	0.234151420264049\\
-0.000201941404057287	0.210756740888782\\
-0.000218372478997791	0.189699612657452\\
-0.000233156404099719	0.170746203063496\\
};
\addplot [color=black,solid,mark=x,mark options={solid},forget plot]
  table[row sep=crcr]{%
0.000373290282107307	0.942809041582064\\
0.00029883390428256	0.848610759794632\\
0.000231927260781749	0.763824480706443\\
0.000171794889150729	0.68750865931322\\
0.000117743077376495	0.618818307728485\\
6.91503666452986e-05	0.55699055218996\\
2.54602636040149e-05	0.501340429053216\\
-1.38261878160508e-05	0.451250287290252\\
-4.91562325885775e-05	0.406164816556377\\
-8.09311282932692e-05	0.365583976684683\\
-0.000109510862735615	0.329057560592674\\
-0.000135218560970655	0.296180734825948\\
-0.000158344333834342	0.266588524685694\\
-0.000179148677587101	0.239953146065968\\
-0.00019786563237876	0.215978756013617\\
};
\addplot [color=black,solid,mark=x,mark options={solid},forget plot]
  table[row sep=crcr]{%
0.000578710165056088	1.20185042515466\\
0.000483365177444561	1.08177088021137\\
0.000397726383139568	0.973688844563822\\
0.000320789654837775	0.876405207048611\\
0.000251657909022225	0.788841529789667\\
0.000189528792515787	0.710026586390127\\
0.000133684509653378	0.639086052674395\\
8.34824009428427e-05	0.575233626037156\\
3.83468105925366e-05	0.517760501926382\\
-2.23808174190555e-06	0.466030036492987\\
-3.87348060077932e-05	0.419467676156606\\
-7.15580885453084e-05	0.377557879659535\\
-0.000101080134497117	0.339834989427531\\
-0.000127634861985652	0.305881459711535\\
-0.000151522173779514	0.275319956648929\\
};
\addplot [color=black,solid,mark=x,mark options={solid},forget plot]
  table[row sep=crcr]{%
0.000461418072760261	1.05409255338946\\
0.000378010869742768	0.948775672669272\\
0.000303075456256598	0.853981730869042\\
0.000235738943087545	0.768658048672052\\
0.000175220966645087	0.691859900901137\\
0.000120822854509576	0.622734368760948\\
7.19193044918188e-05	0.560515639615951\\
2.79499383438145e-05	0.504513158623656\\
-1.15873253717824e-05	0.454106069944396\\
-4.71427628423779e-05	0.408735311675825\\
-7.91201936750029e-05	0.367897536997738\\
-0.000107881968345993	0.33114012834833\\
-0.000133753319649708	0.298055031609297\\
-0.000157026202971249	0.268275773009221\\
-0.000177962859481749	0.241471590071145\\
};
\addplot [color=black,solid,mark=x,mark options={solid},forget plot]
  table[row sep=crcr]{%
0.000217367556605077	0.74535599249993\\
0.000158708208091026	0.670885706420798\\
0.000105978697526524	0.603856287849979\\
5.85733522620332e-05	0.543523298783245\\
1.59497235979164e-05	0.489218843838872\\
-2.2378669189956e-05	0.440339672390999\\
-5.684784842135e-05	0.396344424945242\\
-8.78491183676013e-05	0.356744655773149\\
-0.000115733480605257	0.321101494233031\\
-0.000140816093383975	0.28901949554136\\
-0.000163379860127798	0.260142853025973\\
-0.000183678883032673	0.234151420264049\\
-0.000201941404057287	0.210756740888782\\
-0.000218372478997791	0.189699612657452\\
-0.000233156404099719	0.170746203063496\\
};
\addplot [color=black,solid,mark=x,mark options={solid},forget plot]
  table[row sep=crcr]{%
1.97675465916394e-06	0.471404520791031\\
-3.49443372656666e-05	0.424305379897317\\
-6.81489895083244e-05	0.381912240353221\\
-9.80138073813373e-05	0.34375432965661\\
-0.000124876642761121	0.309409153864243\\
-0.000149040935872405	0.27849527609498\\
-0.000170779033511575	0.250670214526609\\
-0.000190335636967075	0.225625143645126\\
-0.000207930501751033	0.203082408278189\\
-0.000223761125542706	0.182791988342341\\
-0.000238004967074137	0.164528780296337\\
-0.000250821536000497	0.148090367412974\\
-0.000262354225348442	0.133294262342847\\
-0.000272731937051883	0.119976573032984\\
-0.000282070600906293	0.107989378006808\\
};
\addplot [color=black,solid,mark=x,mark options={solid},forget plot]
  table[row sep=crcr]{%
-0.000106166310200297	0.333333333333333\\
-0.000132210006524662	0.300029221471061\\
-0.000155637845864973	0.270052724167203\\
-0.00017671383187697	0.243071032493414\\
-0.000195675010653683	0.218785300123412\\
-0.000212734434828595	0.196925912720149\\
-0.000228083473058419	0.17725059908036\\
-0.000241894160470317	0.159541080751547\\
-0.000254321097051025	0.143600940458867\\
-0.000265503280331743	0.129253462757137\\
-0.000275565675469194	0.116339410704184\\
-0.000284620642584395	0.104715708101085\\
-0.000292769237257323	0.0942532731404986\\
-0.000300102324381448	0.0848362508878004\\
-0.000306701656471861	0.0763600194019831\\
};
\addplot [color=black,solid,mark=x,mark options={solid},forget plot]
  table[row sep=crcr]{%
1.97675465916501e-06	0.471404520791031\\
-3.49443372656666e-05	0.424305379897317\\
-6.81489895083244e-05	0.381912240353221\\
-9.80138073813366e-05	0.343754329656611\\
-0.000124876642761121	0.309409153864243\\
-0.000149040935872405	0.27849527609498\\
-0.000170779033511575	0.250670214526609\\
-0.000190335636967075	0.225625143645126\\
-0.000207930501751033	0.203082408278189\\
-0.000223761125542706	0.182791988342341\\
-0.000238004967074137	0.164528780296337\\
-0.000250821536000497	0.148090367412975\\
-0.000262354225348442	0.133294262342847\\
-0.000272731937051883	0.119976573032984\\
-0.000282070600906293	0.107989378006808\\
};
\addplot [color=black,solid,mark=x,mark options={solid},forget plot]
  table[row sep=crcr]{%
0.000217367604243283	0.74535599249993\\
0.000158708246551894	0.6708859236935\\
0.000105978705540932	0.603856120342304\\
5.85733894378144e-05	0.543523517288255\\
1.59497240125438e-05	0.489218696589278\\
-2.23786416034289e-05	0.440339851647399\\
-5.68478543737744e-05	0.396344307220589\\
-8.78491017527492e-05	0.356744783642232\\
-0.000115733490368577	0.321101408282976\\
-0.000140816085906049	0.289019574588623\\
-0.000163379871047643	0.260142796456865\\
-0.000183678881248775	0.234151459545296\\
-0.000201941413396214	0.210756708585459\\
-0.000218372481419512	0.189699623022409\\
-0.000233156409146162	0.170746190893769\\
};
\addplot [color=black,solid,mark=x,mark options={solid},forget plot]
  table[row sep=crcr]{%
0.000461418164064298	1.05409255338946\\
0.000378010947674185	0.948776138808791\\
0.000303075455677685	0.853981406060374\\
0.000235739002584941	0.768658437746888\\
0.000175220953082348	0.691859641109652\\
0.000120822890931313	0.622734649044088\\
7.19192828984425e-05	0.56051544979122\\
2.79499558427239e-05	0.504513333105839\\
-1.158734847378e-05	0.45410594467647\\
-4.71427568666948e-05	0.408735399005072\\
-7.91202122068705e-05	0.367897465464252\\
-0.000107881966431637	0.331140151879988\\
-0.000133753329072281	0.298055011638571\\
-0.000157026206412702	0.268275754993234\\
-0.000177962859037909	0.241471612237471\\
};
\addplot [color=black,solid,mark=x,mark options={solid},forget plot]
  table[row sep=crcr]{%
0.000418558451533631	1\\
0.000339506366742795	0.900087675089919\\
0.000268477023759253	0.810158131282656\\
0.0002046451395477	0.729213107024132\\
0.000147272478689644	0.656355973169123\\
9.56981007486431e-05	0.590777637306343\\
4.93301842041643e-05	0.531751930797434\\
7.63828267249773e-06	0.478623079771031\\
-2.98529511607001e-05	0.430802976202164\\
-6.3569973925978e-05	0.387760204026884\\
-9.38951929259365e-05	0.349018381769797\\
-0.000121171948310327	0.314146946063971\\
-0.000145708284282367	0.282759948454458\\
-0.00016778094324005	0.254508605332817\\
-0.00018763834371899	0.229080159632866\\
};
\addplot [color=black,solid,mark=x,mark options={solid},forget plot]
  table[row sep=crcr]{%
0.000155387317498033	0.666666666666667\\
0.00010299328259783	0.600058666010219\\
5.58892362896015e-05	0.540105204092366\\
1.35361940534349e-05	0.486142300274515\\
-2.45490988702067e-05	0.437570422884226\\
-5.87998379640902e-05	0.393851971172591\\
-8.96048037918641e-05	0.354501079237506\\
-0.000117312701106049	0.319082232521777\\
-0.000142236699485773	0.287201813649878\\
-0.000164657832706213	0.258506941555869\\
-0.000184828629616802	0.232678802894661\\
-0.00020297581385485	0.20943139644561\\
-0.000219303186233468	0.188506564579455\\
-0.000233993811354686	0.169672462295914\\
-0.000247212272407134	0.152720079612198\\
};
\addplot [color=black,solid,mark=x,mark options={solid},forget plot]
  table[row sep=crcr]{%
-0.000106166310200297	0.333333333333333\\
-0.000132210006524662	0.300029221471061\\
-0.000155637845864973	0.270052724167203\\
-0.00017671383187697	0.243071032493414\\
-0.000195675010653683	0.218785300123412\\
-0.000212734434828595	0.196925912720149\\
-0.000228083473058419	0.17725059908036\\
-0.000241894160470317	0.159541080751547\\
-0.000254321097051025	0.143600940458867\\
-0.000265503280331743	0.129253462757137\\
-0.000275565675469194	0.116339410704184\\
-0.000284620642584395	0.104715708101085\\
-0.000292769237257323	0.0942532731404986\\
-0.000300102324381448	0.0848362508878004\\
-0.000306701656471861	0.0763600194019831\\
};
\addplot [color=black,solid,mark=x,mark options={solid},forget plot]
  table[row sep=crcr]{%
-0.000366106547876789	0\\
-0.000366106547876789	0\\
};
\addplot [color=black,solid,mark=x,mark options={solid},forget plot]
  table[row sep=crcr]{%
-0.000106166310200297	0.333333333333333\\
-0.000132210006524662	0.300029221471061\\
-0.000155637845864973	0.270052724167202\\
-0.00017671383187697	0.243071032493414\\
-0.000195675010653683	0.218785300123412\\
-0.000212734434828594	0.196925912720149\\
-0.000228083473058419	0.17725059908036\\
-0.000241894160470317	0.159541080751547\\
-0.000254321097051025	0.143600940458867\\
-0.000265503280331743	0.129253462757137\\
-0.000275565675469194	0.116339410704184\\
-0.000284620642584395	0.104715708101086\\
-0.000292769237257323	0.0942532731404986\\
-0.000300102324381448	0.0848362508878004\\
-0.000306701656471861	0.0763600194019831\\
};
\addplot [color=black,solid,mark=x,mark options={solid},forget plot]
  table[row sep=crcr]{%
0.000155387317498033	0.666666666666667\\
0.00010299328259783	0.600058666010219\\
5.58892362896015e-05	0.540105204092366\\
1.35361940534349e-05	0.486142300274515\\
-2.45490988702067e-05	0.437570422884226\\
-5.87998379640902e-05	0.393851971172591\\
-8.96048037918641e-05	0.354501079237506\\
-0.000117312701106049	0.319082232521777\\
-0.000142236699485773	0.287201813649878\\
-0.000164657832706213	0.258506941555869\\
-0.000184828629616802	0.232678802894661\\
-0.00020297581385485	0.20943139644561\\
-0.000219303186233468	0.188506564579455\\
-0.000233993811354686	0.169672462295914\\
-0.000247212272407134	0.152720079612198\\
};
\addplot [color=black,solid,mark=x,mark options={solid},forget plot]
  table[row sep=crcr]{%
0.000418558451533631	1\\
0.000339506366742795	0.900087675089919\\
0.000268477023759253	0.810158131282656\\
0.0002046451395477	0.729213107024132\\
0.000147272478689644	0.656355973169123\\
9.56981007486431e-05	0.590777637306343\\
4.93301842041643e-05	0.531751930797434\\
7.63828267249773e-06	0.478623079771031\\
-2.98529511607001e-05	0.430802976202164\\
-6.3569973925978e-05	0.387760204026884\\
-9.38951929259365e-05	0.349018381769797\\
-0.000121171948310327	0.314146946063971\\
-0.000145708284282367	0.282759948454458\\
-0.00016778094324005	0.254508605332817\\
-0.00018763834371899	0.229080159632866\\
};
\addplot [color=black,solid,mark=x,mark options={solid},forget plot]
  table[row sep=crcr]{%
0.000461418164064298	1.05409255338946\\
0.000378010947674185	0.948776138808791\\
0.000303075455677685	0.853981406060374\\
0.000235739002584941	0.768658437746888\\
0.000175220953082348	0.691859641109651\\
0.000120822890931313	0.622734649044089\\
7.19192828984425e-05	0.56051544979122\\
2.79499558427239e-05	0.504513333105839\\
-1.158734847378e-05	0.45410594467647\\
-4.71427568666948e-05	0.408735399005072\\
-7.91202122068705e-05	0.367897465464252\\
-0.000107881966431637	0.331140151879988\\
-0.000133753329072281	0.298055011638571\\
-0.000157026206412702	0.268275754993234\\
-0.000177962859037909	0.241471612237471\\
};
\addplot [color=black,solid,mark=x,mark options={solid},forget plot]
  table[row sep=crcr]{%
0.000217367604243283	0.74535599249993\\
0.000158708246551894	0.6708859236935\\
0.000105978705540933	0.603856120342304\\
5.85733894378144e-05	0.543523517288254\\
1.59497240125438e-05	0.489218696589278\\
-2.23786416034289e-05	0.440339851647399\\
-5.68478543737744e-05	0.396344307220589\\
-8.78491017527492e-05	0.356744783642232\\
-0.000115733490368577	0.321101408282976\\
-0.000140816085906049	0.289019574588623\\
-0.000163379871047643	0.260142796456865\\
-0.000183678881248775	0.234151459545296\\
-0.000201941413396214	0.210756708585459\\
-0.000218372481419512	0.189699623022409\\
-0.000233156409146162	0.170746190893769\\
};
\addplot [color=black,solid,mark=x,mark options={solid},forget plot]
  table[row sep=crcr]{%
1.97675465916394e-06	0.471404520791031\\
-3.49443372656666e-05	0.424305379897316\\
-6.81489895083244e-05	0.381912240353221\\
-9.80138073813373e-05	0.34375432965661\\
-0.000124876642761121	0.309409153864243\\
-0.000149040935872405	0.27849527609498\\
-0.000170779033511575	0.250670214526609\\
-0.000190335636967075	0.225625143645126\\
-0.000207930501751033	0.203082408278189\\
-0.000223761125542706	0.182791988342341\\
-0.000238004967074137	0.164528780296337\\
-0.000250821536000497	0.148090367412974\\
-0.000262354225348442	0.133294262342847\\
-0.000272731937051883	0.119976573032984\\
-0.000282070600906293	0.107989378006809\\
};
\addplot [color=black,solid,mark=x,mark options={solid},forget plot]
  table[row sep=crcr]{%
-0.000106166310200297	0.333333333333333\\
-0.000132210006524662	0.300029221471061\\
-0.000155637845864973	0.270052724167202\\
-0.00017671383187697	0.243071032493414\\
-0.000195675010653683	0.218785300123412\\
-0.000212734434828594	0.196925912720149\\
-0.000228083473058419	0.17725059908036\\
-0.000241894160470317	0.159541080751547\\
-0.000254321097051025	0.143600940458867\\
-0.000265503280331743	0.129253462757137\\
-0.000275565675469194	0.116339410704184\\
-0.000284620642584395	0.104715708101086\\
-0.000292769237257323	0.0942532731404986\\
-0.000300102324381448	0.0848362508878004\\
-0.000306701656471861	0.0763600194019831\\
};
\addplot [color=black,solid,mark=x,mark options={solid},forget plot]
  table[row sep=crcr]{%
1.97675465916394e-06	0.471404520791031\\
-3.49443372656672e-05	0.424305379897316\\
-6.81489895083244e-05	0.381912240353221\\
-9.80138073813373e-05	0.34375432965661\\
-0.000124876642761121	0.309409153864243\\
-0.000149040935872405	0.27849527609498\\
-0.000170779033511575	0.250670214526608\\
-0.000190335636967075	0.225625143645126\\
-0.000207930501751033	0.203082408278189\\
-0.000223761125542706	0.182791988342341\\
-0.000238004967074137	0.164528780296337\\
-0.000250821536000497	0.148090367412974\\
-0.000262354225348442	0.133294262342847\\
-0.000272731937051883	0.119976573032984\\
-0.000282070600906293	0.107989378006808\\
};
\addplot [color=black,solid,mark=x,mark options={solid},forget plot]
  table[row sep=crcr]{%
0.000217367556605077	0.74535599249993\\
0.000158708208091026	0.670885706420798\\
0.000105978697526524	0.603856287849979\\
5.85733522620332e-05	0.543523298783245\\
1.59497235979164e-05	0.489218843838872\\
-2.2378669189956e-05	0.440339672390999\\
-5.684784842135e-05	0.396344424945242\\
-8.78491183676013e-05	0.356744655773149\\
-0.000115733480605257	0.321101494233031\\
-0.000140816093383975	0.28901949554136\\
-0.000163379860127798	0.260142853025973\\
-0.000183678883032673	0.234151420264049\\
-0.000201941404057287	0.210756740888782\\
-0.000218372478997791	0.189699612657452\\
-0.000233156404099719	0.170746203063496\\
};
\addplot [color=black,solid,mark=x,mark options={solid},forget plot]
  table[row sep=crcr]{%
0.000461418072760261	1.05409255338946\\
0.000378010869742768	0.948775672669272\\
0.0003030754562566	0.853981730869042\\
0.000235738943087547	0.768658048672053\\
0.000175220966645087	0.691859900901138\\
0.000120822854509577	0.622734368760948\\
7.19193044918188e-05	0.560515639615951\\
2.79499383438145e-05	0.504513158623656\\
-1.15873253717824e-05	0.454106069944395\\
-4.71427628423779e-05	0.408735311675825\\
-7.91201936750029e-05	0.367897536997738\\
-0.000107881968345993	0.33114012834833\\
-0.000133753319649708	0.298055031609297\\
-0.000157026202971249	0.268275773009221\\
-0.000177962859481749	0.241471590071145\\
};
\addplot [color=black,solid,mark=x,mark options={solid},forget plot]
  table[row sep=crcr]{%
0.000578710165056088	1.20185042515466\\
0.000483365177444561	1.08177088021137\\
0.000397726383139568	0.973688844563822\\
0.000320789654837775	0.876405207048611\\
0.000251657909022225	0.788841529789667\\
0.000189528792515787	0.710026586390127\\
0.000133684509653378	0.639086052674395\\
8.34824009428427e-05	0.575233626037156\\
3.83468105925366e-05	0.517760501926382\\
-2.23808174190555e-06	0.466030036492987\\
-3.87348060077932e-05	0.419467676156606\\
-7.15580885453084e-05	0.377557879659535\\
-0.000101080134497117	0.339834989427531\\
-0.000127634861985652	0.305881459711535\\
-0.000151522173779514	0.275319956648929\\
};
\addplot [color=black,solid,mark=x,mark options={solid},forget plot]
  table[row sep=crcr]{%
0.000373290282107307	0.942809041582064\\
0.00029883390428256	0.848610759794632\\
0.000231927260781749	0.763824480706443\\
0.000171794889150729	0.68750865931322\\
0.000117743077376495	0.618818307728485\\
6.91503666452986e-05	0.55699055218996\\
2.54602636040149e-05	0.501340429053216\\
-1.38261878160508e-05	0.451250287290252\\
-4.91562325885775e-05	0.406164816556377\\
-8.09311282932692e-05	0.365583976684683\\
-0.000109510862735615	0.329057560592674\\
-0.000135218560970655	0.296180734825948\\
-0.000158344333834342	0.266588524685694\\
-0.000179148677587101	0.239953146065968\\
-0.00019786563237876	0.215978756013617\\
};
\addplot [color=black,solid,mark=x,mark options={solid},forget plot]
  table[row sep=crcr]{%
0.000217367556605077	0.74535599249993\\
0.000158708208091026	0.670885706420798\\
0.000105978697526524	0.603856287849979\\
5.85733522620332e-05	0.543523298783245\\
1.59497235979164e-05	0.489218843838872\\
-2.2378669189956e-05	0.440339672390999\\
-5.684784842135e-05	0.396344424945242\\
-8.78491183676013e-05	0.356744655773149\\
-0.000115733480605257	0.321101494233031\\
-0.000140816093383975	0.28901949554136\\
-0.000163379860127798	0.260142853025973\\
-0.000183678883032673	0.234151420264049\\
-0.000201941404057287	0.210756740888782\\
-0.000218372478997791	0.189699612657452\\
-0.000233156404099719	0.170746203063496\\
};
\addplot [color=black,solid,mark=x,mark options={solid},forget plot]
  table[row sep=crcr]{%
0.000155387317498033	0.666666666666667\\
0.00010299328259783	0.600058666010219\\
5.58892362896015e-05	0.540105204092366\\
1.35361940534349e-05	0.486142300274515\\
-2.45490988702067e-05	0.437570422884226\\
-5.87998379640902e-05	0.393851971172591\\
-8.96048037918641e-05	0.354501079237506\\
-0.000117312701106049	0.319082232521777\\
-0.000142236699485773	0.287201813649878\\
-0.000164657832706213	0.258506941555869\\
-0.000184828629616802	0.232678802894661\\
-0.00020297581385485	0.20943139644561\\
-0.000219303186233468	0.188506564579455\\
-0.000233993811354686	0.169672462295914\\
-0.000247212272407134	0.152720079612198\\
};
\addplot [color=black,solid,mark=x,mark options={solid},forget plot]
  table[row sep=crcr]{%
0.000217367604243283	0.74535599249993\\
0.000158708246551894	0.6708859236935\\
0.000105978705540933	0.603856120342304\\
5.85733894378144e-05	0.543523517288254\\
1.59497240125438e-05	0.489218696589278\\
-2.23786416034289e-05	0.440339851647399\\
-5.68478543737744e-05	0.396344307220589\\
-8.78491017527492e-05	0.356744783642232\\
-0.000115733490368577	0.321101408282976\\
-0.000140816085906049	0.289019574588623\\
-0.000163379871047643	0.260142796456865\\
-0.000183678881248775	0.234151459545296\\
-0.000201941413396214	0.210756708585459\\
-0.000218372481419512	0.189699623022409\\
-0.000233156409146162	0.170746190893769\\
};
\addplot [color=black,solid,mark=x,mark options={solid},forget plot]
  table[row sep=crcr]{%
0.000373290282107307	0.942809041582064\\
0.00029883390428256	0.848610759794632\\
0.000231927260781749	0.763824480706443\\
0.000171794889150729	0.68750865931322\\
0.000117743077376495	0.618818307728485\\
6.91503666452986e-05	0.55699055218996\\
2.54602636040149e-05	0.501340429053216\\
-1.38261878160508e-05	0.451250287290252\\
-4.91562325885775e-05	0.406164816556377\\
-8.09311282932692e-05	0.365583976684683\\
-0.000109510862735615	0.329057560592674\\
-0.000135218560970655	0.296180734825948\\
-0.000158344333834342	0.266588524685694\\
-0.000179148677587101	0.239953146065968\\
-0.00019786563237876	0.215978756013617\\
};
\addplot [color=black,solid,mark=x,mark options={solid},forget plot]
  table[row sep=crcr]{%
0.000578710120461909	1.20185042515466\\
0.000483365143974822	1.08177069696277\\
0.000397726369890652	0.973689007115459\\
0.000320789614573874	0.8764049853992\\
0.000251657903214204	0.78884168287072\\
0.000189528759140407	0.710026387123514\\
0.000133684511556436	0.639086182201942\\
8.34823784327201e-05	0.575233473034235\\
3.83468179064484e-05	0.5177606019741\\
-2.23809393216808e-06	0.466029933467533\\
-3.87347959648823e-05	0.41946774649463\\
-7.15580927349579e-05	0.377557820459789\\
-0.00010108012395128	0.339835033447347\\
-0.000127634860481865	0.305881434399356\\
-0.00015152216628937	0.275319979728245\\
};
\addplot [color=black,solid,mark=x,mark options={solid},forget plot]
  table[row sep=crcr]{%
0.000747845590505265	1.4142135623731\\
0.000635236693231446	1.27291613969195\\
0.00053412847183248	1.14573672105966\\
0.000443324151259217	1.03126298896983\\
0.000361755998848757	0.928227461592727\\
0.00028846984534608	0.835485828284941\\
0.000222613166730381	0.752010643579825\\
0.000163423136106581	0.676875430935378\\
0.00011021723879715	0.609247224834565\\
6.23841608993914e-05	0.548375965027024\\
1.93762899640786e-05	0.493586340889013\\
-1.92972387946452e-05	0.444271102238922\\
-5.40765924814501e-05	0.399882787028542\\
-8.53565641300495e-05	0.359929719098952\\
-0.000113491492187643	0.323968134020426\\
};
\addplot [color=black,solid,mark=x,mark options={solid},forget plot]
  table[row sep=crcr]{%
0.000578710120461909	1.20185042515466\\
0.000483365143974822	1.08177069696277\\
0.000397726369890652	0.973689007115459\\
0.000320789614573874	0.8764049853992\\
0.000251657903214204	0.78884168287072\\
0.000189528759140407	0.710026387123514\\
0.000133684511556436	0.639086182201942\\
8.34823784327201e-05	0.575233473034235\\
3.83468179064484e-05	0.5177606019741\\
-2.23809393216808e-06	0.466029933467533\\
-3.87347959648823e-05	0.41946774649463\\
-7.15580927349579e-05	0.377557820459789\\
-0.00010108012395128	0.339835033447347\\
-0.000127634860481865	0.305881434399356\\
-0.00015152216628937	0.275319979728245\\
};
\addplot [color=black,solid,mark=x,mark options={solid},forget plot]
  table[row sep=crcr]{%
0.000461418072760261	1.05409255338946\\
0.000378010869742768	0.948775672669272\\
0.000303075456256598	0.853981730869042\\
0.000235738943087545	0.768658048672052\\
0.000175220966645087	0.691859900901137\\
0.000120822854509576	0.622734368760948\\
7.19193044918188e-05	0.560515639615951\\
2.79499383438145e-05	0.504513158623656\\
-1.15873253717824e-05	0.454106069944396\\
-4.71427628423779e-05	0.408735311675825\\
-7.91201936750029e-05	0.367897536997738\\
-0.000107881968345993	0.33114012834833\\
-0.000133753319649708	0.298055031609297\\
-0.000157026202971249	0.268275773009221\\
-0.000177962859481749	0.241471590071145\\
};
\addplot [color=black,solid,mark=x,mark options={solid},forget plot]
  table[row sep=crcr]{%
0.000418558451533631	1\\
0.000339506366742795	0.900087675089919\\
0.000268477023759253	0.810158131282656\\
0.0002046451395477	0.729213107024132\\
0.000147272478689644	0.656355973169123\\
9.56981007486431e-05	0.590777637306343\\
4.93301842041643e-05	0.531751930797434\\
7.63828267249773e-06	0.478623079771031\\
-2.98529511607001e-05	0.430802976202164\\
-6.3569973925978e-05	0.387760204026884\\
-9.38951929259365e-05	0.349018381769797\\
-0.000121171948310327	0.314146946063971\\
-0.000145708284282367	0.282759948454458\\
-0.00016778094324005	0.254508605332817\\
-0.00018763834371899	0.229080159632866\\
};
\addplot [color=black,solid,mark=x,mark options={solid},forget plot]
  table[row sep=crcr]{%
0.000461418164064298	1.05409255338946\\
0.000378010947674185	0.948776138808791\\
0.000303075455677685	0.853981406060374\\
0.000235739002584941	0.768658437746888\\
0.000175220953082348	0.691859641109651\\
0.000120822890931313	0.622734649044089\\
7.19192828984425e-05	0.56051544979122\\
2.79499558427239e-05	0.504513333105839\\
-1.158734847378e-05	0.45410594467647\\
-4.71427568666948e-05	0.408735399005072\\
-7.91202122068705e-05	0.367897465464252\\
-0.000107881966431637	0.331140151879988\\
-0.000133753329072281	0.298055011638571\\
-0.000157026206412702	0.268275754993234\\
-0.000177962859037909	0.241471612237471\\
};
\addplot [color=black,solid,mark=x,mark options={solid},forget plot]
  table[row sep=crcr]{%
0.000578710165056088	1.20185042515466\\
0.000483365177444561	1.08177088021137\\
0.000397726383139568	0.973688844563822\\
0.000320789654837775	0.876405207048611\\
0.000251657909022225	0.788841529789667\\
0.000189528792515787	0.710026586390127\\
0.000133684509653378	0.639086052674395\\
8.34824009428427e-05	0.575233626037156\\
3.83468105925366e-05	0.517760501926382\\
-2.23808174190555e-06	0.466030036492987\\
-3.87348060077932e-05	0.419467676156606\\
-7.15580885453084e-05	0.377557879659535\\
-0.000101080134497117	0.339834989427531\\
-0.000127634861985652	0.305881459711535\\
-0.000151522173779514	0.275319956648929\\
};
\addplot [color=black,solid,mark=x,mark options={solid},forget plot]
  table[row sep=crcr]{%
0.000747845590505265	1.4142135623731\\
0.000635236693231446	1.27291613969195\\
0.00053412847183248	1.14573672105966\\
0.000443324151259217	1.03126298896983\\
0.000361755998848757	0.928227461592727\\
0.00028846984534608	0.835485828284941\\
0.000222613166730381	0.752010643579825\\
0.000163423136106581	0.676875430935378\\
0.00011021723879715	0.609247224834565\\
6.23841608993914e-05	0.548375965027024\\
1.93762899640786e-05	0.493586340889013\\
-1.92972387946452e-05	0.444271102238922\\
-5.40765924814501e-05	0.399882787028542\\
-8.53565641300495e-05	0.359929719098952\\
};
\end{axis}
\end{tikzpicture}%

%% file: redsc.tex
%
%
\begin{tikzpicture}

\begin{axis}[%
width=0.571\fwidth,
height=\fheight,
at={(0\fwidth,0\fheight)},
scale only axis,
xmode=log,
xmin=822.645664814553,
xmax=115038677.764875,
xminorticks=true,
xlabel={\#points per dimension},
ymode=log,
ymin=203595819.13634,
ymax=6.6083907755405e+16,
yminorticks=true,
ylabel={total \#points stored},
axis background/.style={fill=white},
legend style={at={(0.5,1.03)},anchor=south,legend cell align=left,align=left,draw=white!15!black}
]
\addplot [color=black,dotted,line width=3.0pt]
  table[row sep=crcr]{%
1000	1000000000\\
1584.89319246111	3981071705.53498\\
2511.88643150958	15848931924.6111\\
3981.07170553497	63095734448.0194\\
6309.57344480193	251188643150.958\\
10000	1000000000000\\
15848.9319246111	3981071705534.98\\
25118.8643150958	15848931924611.2\\
39810.7170553497	63095734448019.2\\
63095.7344480193	251188643150958\\
100000	1e+15\\
158489.319246111	3.98107170553498e+15\\
251188.643150958	1.58489319246112e+16\\
398107.170553497	6.30957344480192e+16\\
409987.265483585	7.2692298510586e+16\\
};
\addlegendentry{naive};

\addplot [color=black,solid,line width=1.0pt,mark=x,mark options={solid}]
  table[row sep=crcr]{%
10000	100020000\\
15848.9319246111	251220341.014807\\
25118.8643150958	631007582.208825\\
39810.7170553497	1584972813.89522\\
63095.7344480193	3981197897.00387\\
100000	10000200000\\
158489.319246111	25119181293.7343\\
251188.643150958	63096236825.3057\\
398107.170553497	158490115460.452\\
630957.344480193	398108432468.186\\
1000000	1000002000000\\
1584893.19246111	2511889601295.97\\
2511886.43150958	6309578468574.81\\
3981071.70553497	15848939886754.5\\
6309573.44480193	39810729674496.6\\
10000000	100000020000000\\
15848931.9246111	251188674848822\\
25118864.3150958	630957394717923\\
39810717.0553497	1.58489327208255e+15\\
63095734.4480193	3.98107183172644e+15\\
100000000	1.00000002e+16\\
};
\addlegendentry{closed observables};

\addplot [color=gray,solid,line width=2.0pt,mark=o,mark options={solid}]
  table[row sep=crcr]{%
1000	8000000000\\
100000000	8000000000\\
};
\addlegendentry{maximum for a smartphone};

\addplot [color=gray,solid,line width=2.0pt,mark=x,mark options={solid}]
  table[row sep=crcr]{%
1000	1.408e+15\\
100000000	1.408e+15\\
};
\addlegendentry{maximum for a supercomputer};

\addplot [color=gray,dotted,line width=2.0pt]
  table[row sep=crcr]{%
100000	8000000000\\
100000	1.408e+15\\
};
\addlegendentry{distance between smartphone and supercomputer};

\end{axis}
\end{tikzpicture}%

%% file: errtime.tex
%
%
\definecolor{mycolor1}{rgb}{0.00000,0.44700,0.74100}%
\begin{tikzpicture}

\begin{axis}[%
width=0.951\fwidth,
height=\fheight,
at={(0\fwidth,0\fheight)},
scale only axis,
xmode=log,
xmin=6.30288191278177,
xmax=290.491641417783,
xminorticks=true,
xlabel={time steps from t=0 to t=2},
xmajorgrids,
xminorgrids,
ymode=log,
ymin=1.98922893374588e-05,
ymax=0.335807209098705,
yminorticks=true,
ylabel={maximum error at t=2},
ymajorgrids,
yminorgrids,
axis background/.style={fill=white},
legend style={at={(0.03,0.03)},anchor=south west,legend cell align=left,align=left,draw=white!15!black}
]
\addplot [color=red,solid,mark=x,mark options={solid}]
  table[row sep=crcr]{%
10	0.12\\
20	0.06\\
40	0.028\\
80	0.014\\
160	0.007\\
};
\addlegendentry{One-sided gradient};

\addplot [color=mycolor1,solid,mark=o,mark options={solid}]
  table[row sep=crcr]{%
10	0.0167\\
20	0.0037\\
40	0.00089299\\
80	0.00021896\\
160	5.9707e-05\\
};
\addlegendentry{Central differences gradient};

\addplot [color=black,dashed,forget plot]
  table[row sep=crcr]{%
10	0.1\\
160	0.00625\\
};
\addplot [color=black,dashed,forget plot]
  table[row sep=crcr]{%
10	0.0117\\
160	4.5703125e-05\\
};
\end{axis}
\end{tikzpicture}%

%% file: dwa4.tex
%
%
\begin{tikzpicture}

\begin{axis}[%
width=0.75\fwidth,
height=\fheight,
at={(0\fwidth,0\fheight)},
scale only axis,
point meta min=0.00193045413622771,
point meta max=0.982835989757635,
axis on top,
xmin=-0.00122448979591837,
xmax=0.0477551020408163,
xlabel={time},
ymin=-0.526315789473684,
ymax=0.526315789473684,
ylabel={space},
axis background/.style={fill=white},
colormap={mymap}{[1pt] rgb(0pt)=(1,1,1); rgb(99pt)=(0,0,0)},
colorbar
]
\addplot [forget plot] graphics [xmin=-0.00122448979591837,xmax=0.0477551020408163,ymin=-0.526315789473684,ymax=0.526315789473684] {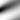};
\end{axis}
\end{tikzpicture}%

%% file: cmppde.tex
%
%
\begin{tikzpicture}

\begin{axis}[%
width=0.75\fwidth,
height=\fheight,
at={(0\fwidth,0\fheight)},
scale only axis,
point meta min=0.0173118201634044,
point meta max=0.997964136221985,
axis on top,
xmin=-0.526315789473684,
xmax=20.5263157894737,
xlabel={time step},
ymin=-0.526315789473684,
ymax=0.526315789473684,
ylabel={space},
axis background/.style={fill=white},
colormap/blackwhite,
colorbar
]
\addplot [forget plot] graphics [xmin=-0.526315789473684,xmax=20.5263157894737,ymin=-0.526315789473684,ymax=0.526315789473684] {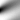};
\end{axis}
\end{tikzpicture}%

%% file: errpde.tex
%
%
\begin{tikzpicture}

\begin{axis}[%
width=0.75\fwidth,
height=\fheight,
at={(0\fwidth,0\fheight)},
scale only axis,
point meta min=2.15710964823721e-07,
point meta max=0.0114075025526915,
axis on top,
xmin=-0.555555555555556,
xmax=20.5555555555556,
xlabel={time step},
ymin=-0.526315789473684,
ymax=0.526315789473684,
ylabel={space},
axis background/.style={fill=white},
colormap={mymap}{[1pt] rgb(0pt)=(1,1,1); rgb(99pt)=(0,0,0)},
colorbar
]
\addplot [forget plot] graphics [xmin=-0.555555555555556,xmax=20.5555555555556,ymin=-0.526315789473684,ymax=0.526315789473684] {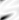};
\end{axis}
\end{tikzpicture}%

%% file: rwevol.tex
%
%
\begin{tikzpicture}

\begin{axis}[%
width=0.951\fwidth,
height=\fheight,
at={(0\fwidth,0\fheight)},
scale only axis,
xmin=0,
xmax=20,
xlabel={time [s]},
ymin=0,
ymax=90,
ylabel={\#persons},
axis background/.style={fill=white},
axis x line*=bottom,
axis y line*=left,
legend style={at={(0.03,0.97)},anchor=north west,legend cell align=left,align=left,draw=white!15!black}
]
\addplot [color=black,solid,mark=o,mark options={solid}]
  table[row sep=crcr]{%
0	43\\
1	48\\
2	53\\
3	56\\
4	62\\
5	66\\
6	70\\
7	74\\
8	76\\
9	78\\
10	80\\
11	83\\
12	85\\
13	87\\
14	89\\
15	90\\
16	90\\
17	90\\
18	90\\
19	90\\
20	90\\
21	90\\
};
\addlegendentry{platform};

\addplot [color=black,solid,mark=x,mark options={solid}]
  table[row sep=crcr]{%
0	47\\
1	42\\
2	37\\
3	34\\
4	28\\
5	24\\
6	20\\
7	16\\
8	14\\
9	12\\
10	10\\
11	7\\
12	5\\
13	3\\
14	1\\
15	0\\
16	0\\
17	0\\
18	0\\
19	0\\
20	0\\
21	0\\
};
\addlegendentry{train};

\end{axis}
\end{tikzpicture}%

%% file: rwt.tex
%
%
\begin{tikzpicture}

\begin{axis}[%
width=0.951\fwidth,
height=\fheight,
at={(0\fwidth,0\fheight)},
scale only axis,
xmin=0,
xmax=20,
xlabel={time [s]},
ymin=50,
ymax=100,
ylabel={\#persons on the platform},
axis background/.style={fill=white},
legend style={at={(0.97,0.03)},anchor=south east,legend cell align=left,align=left,draw=white!15!black}
]
\addplot [color=black,dotted,mark=x,mark options={solid}]
  table[row sep=crcr]{%
0	60\\
1	64\\
2	69\\
3	71\\
4	75\\
5	79\\
6	81\\
7	84\\
8	86\\
9	87\\
10	88\\
11	91\\
12	92\\
13	94\\
14	95\\
15	95\\
16	95\\
17	95\\
18	95\\
19	95\\
20	95\\
21	95\\
};
\addlegendentry{Simulations};

\addplot [color=black,dotted,mark=x,mark options={solid},forget plot]
  table[row sep=crcr]{%
0	60\\
1	63\\
2	68\\
3	71\\
4	75\\
5	77\\
6	79\\
7	84\\
8	86\\
9	88\\
10	90\\
11	92\\
12	94\\
13	94\\
14	95\\
15	95\\
16	95\\
17	95\\
18	95\\
19	95\\
20	95\\
21	95\\
};
\addplot [color=black,dotted,mark=x,mark options={solid},forget plot]
  table[row sep=crcr]{%
0	60\\
1	64\\
2	67\\
3	71\\
4	72\\
5	74\\
6	77\\
7	78\\
8	80\\
9	83\\
10	85\\
11	88\\
12	89\\
13	91\\
14	93\\
15	95\\
16	95\\
17	95\\
18	95\\
19	95\\
20	95\\
21	95\\
};
\addplot [color=black,solid,line width=2.0pt]
  table[row sep=crcr]{%
0	58.5921879807271\\
1	63.9989969460761\\
2	67.7559169965179\\
3	71.5533072848881\\
4	74.5932741888738\\
5	78.9347162090916\\
6	81.5000774292062\\
7	82.4999737184106\\
8	85.8803971598153\\
9	88.3043677950404\\
10	88.9262351192468\\
11	90.2838124839726\\
12	91.6413898486984\\
13	92.1412177176025\\
14	92.8255237010887\\
15	93.6407013243145\\
16	94.1547198174172\\
17	94.9982786890404\\
18	94.9982786890404\\
19	94.9982786890404\\
20	94.9982786890404\\
};
\addlegendentry{Closed Observables};

\end{axis}
\end{tikzpicture}%

%% file: rwsa25.tex
%
%
\begin{tikzpicture}

\begin{axis}[%
width=0.939\fwidth,
height=\fheight,
at={(0\fwidth,0\fheight)},
scale only axis,
xmin=0,
xmax=200,
tick align=outside,
xlabel={$\text{N}_\text{P}\text{(0)}$},
xmajorgrids,
ymin=10,
ymax=50,
ylabel={$\text{N}_\text{T}\text{(0)}$},
ymajorgrids,
zmin=0,
zmax=1,
zlabel={c(25)},
zmajorgrids,
view={45}{30},
axis background/.style={fill=white},
axis x line*=bottom,
axis y line*=left,
axis z line*=left
]

\addplot3[%
surf,
shader=flat corner,draw=black,z buffer=sort,colormap/blackwhite,mesh/rows=20]
table[row sep=crcr, point meta=\thisrow{c}] {%
x	y	z	c\\
0	10	1	1\\
0	12.1052631578947	1	1\\
0	14.2105263157895	1	1\\
0	16.3157894736842	1	1\\
0	18.4210526315789	1	1\\
0	20.5263157894737	0.999999999999999	0.999999999999999\\
0	22.6315789473684	1	1\\
0	24.7368421052632	1	1\\
0	26.8421052631579	1	1\\
0	28.9473684210526	1	1\\
0	31.0526315789474	0.982699118780527	0.982699118780527\\
0	33.1578947368421	0.990769526614628	0.990769526614628\\
0	35.2631578947368	1	1\\
0	37.3684210526316	1	1\\
0	39.4736842105263	1	1\\
0	41.5789473684211	0.976089629457966	0.976089629457966\\
0	43.6842105263158	0.973609311665627	0.973609311665627\\
0	45.7894736842105	0.973885349337629	0.973885349337629\\
0	47.8947368421053	0.978161456281452	0.978161456281452\\
0	50	0.97937272613078	0.97937272613078\\
10.5263157894737	10	1	1\\
10.5263157894737	12.1052631578947	1	1\\
10.5263157894737	14.2105263157895	1	1\\
10.5263157894737	16.3157894736842	1	1\\
10.5263157894737	18.4210526315789	1	1\\
10.5263157894737	20.5263157894737	1	1\\
10.5263157894737	22.6315789473684	1	1\\
10.5263157894737	24.7368421052632	1	1\\
10.5263157894737	26.8421052631579	1	1\\
10.5263157894737	28.9473684210526	1	1\\
10.5263157894737	31.0526315789474	0.990220594895875	0.990220594895875\\
10.5263157894737	33.1578947368421	1	1\\
10.5263157894737	35.2631578947368	1	1\\
10.5263157894737	37.3684210526316	1	1\\
10.5263157894737	39.4736842105263	0.999882856557917	0.999882856557917\\
10.5263157894737	41.5789473684211	0.974629471713163	0.974629471713163\\
10.5263157894737	43.6842105263158	0.980727268911025	0.980727268911025\\
10.5263157894737	45.7894736842105	0.986206354477941	0.986206354477941\\
10.5263157894737	47.8947368421053	0.982102903089912	0.982102903089912\\
10.5263157894737	50	0.984938690870259	0.984938690870259\\
21.0526315789474	10	1	1\\
21.0526315789474	12.1052631578947	1	1\\
21.0526315789474	14.2105263157895	1	1\\
21.0526315789474	16.3157894736842	1	1\\
21.0526315789474	18.4210526315789	1	1\\
21.0526315789474	20.5263157894737	1	1\\
21.0526315789474	22.6315789473684	1	1\\
21.0526315789474	24.7368421052632	1	1\\
21.0526315789474	26.8421052631579	1	1\\
21.0526315789474	28.9473684210526	1	1\\
21.0526315789474	31.0526315789474	0.993255097156952	0.993255097156952\\
21.0526315789474	33.1578947368421	1	1\\
21.0526315789474	35.2631578947368	1	1\\
21.0526315789474	37.3684210526316	1	1\\
21.0526315789474	39.4736842105263	0.989373120297212	0.989373120297212\\
21.0526315789474	41.5789473684211	0.973962862039543	0.973962862039543\\
21.0526315789474	43.6842105263158	0.96748138019032	0.96748138019032\\
21.0526315789474	45.7894736842105	0.955559870037416	0.955559870037416\\
21.0526315789474	47.8947368421053	0.955106640083202	0.955106640083202\\
21.0526315789474	50	0.963656871622838	0.963656871622838\\
31.5789473684211	10	1	1\\
31.5789473684211	12.1052631578947	1	1\\
31.5789473684211	14.2105263157895	1	1\\
31.5789473684211	16.3157894736842	1	1\\
31.5789473684211	18.4210526315789	1	1\\
31.5789473684211	20.5263157894737	1	1\\
31.5789473684211	22.6315789473684	1	1\\
31.5789473684211	24.7368421052632	1	1\\
31.5789473684211	26.8421052631579	1	1\\
31.5789473684211	28.9473684210526	0.993132556169739	0.993132556169739\\
31.5789473684211	31.0526315789474	0.991387269395542	0.991387269395542\\
31.5789473684211	33.1578947368421	0.991813743583775	0.991813743583775\\
31.5789473684211	35.2631578947368	1	1\\
31.5789473684211	37.3684210526316	1	1\\
31.5789473684211	39.4736842105263	0.97741673121384	0.97741673121384\\
31.5789473684211	41.5789473684211	0.95873777355943	0.95873777355943\\
31.5789473684211	43.6842105263158	0.951394293935193	0.951394293935193\\
31.5789473684211	45.7894736842105	0.964087354197877	0.964087354197877\\
31.5789473684211	47.8947368421053	0.947686259171414	0.947686259171414\\
31.5789473684211	50	0.947575219962294	0.947575219962294\\
42.1052631578947	10	1	1\\
42.1052631578947	12.1052631578947	1	1\\
42.1052631578947	14.2105263157895	1	1\\
42.1052631578947	16.3157894736842	1	1\\
42.1052631578947	18.4210526315789	0.997202792846965	0.997202792846965\\
42.1052631578947	20.5263157894737	1	1\\
42.1052631578947	22.6315789473684	1	1\\
42.1052631578947	24.7368421052632	1	1\\
42.1052631578947	26.8421052631579	1	1\\
42.1052631578947	28.9473684210526	0.989388160618932	0.989388160618932\\
42.1052631578947	31.0526315789474	0.9897772726409	0.9897772726409\\
42.1052631578947	33.1578947368421	0.984076903773971	0.984076903773971\\
42.1052631578947	35.2631578947368	0.999321715604159	0.999321715604159\\
42.1052631578947	37.3684210526316	0.997655452056967	0.997655452056967\\
42.1052631578947	39.4736842105263	0.969456944945518	0.969456944945518\\
42.1052631578947	41.5789473684211	0.955903307101721	0.955903307101721\\
42.1052631578947	43.6842105263158	0.960514650925542	0.960514650925542\\
42.1052631578947	45.7894736842105	0.982755936445045	0.982755936445045\\
42.1052631578947	47.8947368421053	0.936836009705487	0.936836009705487\\
42.1052631578947	50	0.913090335777605	0.913090335777605\\
52.6315789473684	10	1	1\\
52.6315789473684	12.1052631578947	1	1\\
52.6315789473684	14.2105263157895	1	1\\
52.6315789473684	16.3157894736842	1	1\\
52.6315789473684	18.4210526315789	0.999815227215871	0.999815227215871\\
52.6315789473684	20.5263157894737	1	1\\
52.6315789473684	22.6315789473684	1	1\\
52.6315789473684	24.7368421052632	1	1\\
52.6315789473684	26.8421052631579	1	1\\
52.6315789473684	28.9473684210526	0.983091731709875	0.983091731709875\\
52.6315789473684	31.0526315789474	0.988101585407072	0.988101585407072\\
52.6315789473684	33.1578947368421	1	1\\
52.6315789473684	35.2631578947368	1	1\\
52.6315789473684	37.3684210526316	1	1\\
52.6315789473684	39.4736842105263	0.964070005288717	0.964070005288717\\
52.6315789473684	41.5789473684211	0.94288953888915	0.94288953888915\\
52.6315789473684	43.6842105263158	0.927674189029072	0.927674189029072\\
52.6315789473684	45.7894736842105	0.914953823713006	0.914953823713006\\
52.6315789473684	47.8947368421053	0.872180778472164	0.872180778472164\\
52.6315789473684	50	0.851897351630289	0.851897351630289\\
63.1578947368421	10	1	1\\
63.1578947368421	12.1052631578947	1	1\\
63.1578947368421	14.2105263157895	1	1\\
63.1578947368421	16.3157894736842	1	1\\
63.1578947368421	18.4210526315789	1	1\\
63.1578947368421	20.5263157894737	1	1\\
63.1578947368421	22.6315789473684	1	1\\
63.1578947368421	24.7368421052632	1	1\\
63.1578947368421	26.8421052631579	1	1\\
63.1578947368421	28.9473684210526	0.968360396070267	0.968360396070267\\
63.1578947368421	31.0526315789474	0.967067740668496	0.967067740668496\\
63.1578947368421	33.1578947368421	0.972651159306227	0.972651159306227\\
63.1578947368421	35.2631578947368	0.999999999999999	0.999999999999999\\
63.1578947368421	37.3684210526316	0.980664067674749	0.980664067674749\\
63.1578947368421	39.4736842105263	0.930108862447283	0.930108862447283\\
63.1578947368421	41.5789473684211	0.896967022021663	0.896967022021663\\
63.1578947368421	43.6842105263158	0.856954324467959	0.856954324467959\\
63.1578947368421	45.7894736842105	0.823684331397479	0.823684331397479\\
63.1578947368421	47.8947368421053	0.768041083871161	0.768041083871161\\
63.1578947368421	50	0.726685366575496	0.726685366575496\\
73.6842105263158	10	1	1\\
73.6842105263158	12.1052631578947	1	1\\
73.6842105263158	14.2105263157895	1	1\\
73.6842105263158	16.3157894736842	1	1\\
73.6842105263158	18.4210526315789	1	1\\
73.6842105263158	20.5263157894737	1	1\\
73.6842105263158	22.6315789473684	1	1\\
73.6842105263158	24.7368421052632	1	1\\
73.6842105263158	26.8421052631579	0.983491456074796	0.983491456074796\\
73.6842105263158	28.9473684210526	0.948917006143404	0.948917006143404\\
73.6842105263158	31.0526315789474	0.93572397841946	0.93572397841946\\
73.6842105263158	33.1578947368421	0.986622842292673	0.986622842292673\\
73.6842105263158	35.2631578947368	0.983581809714111	0.983581809714111\\
73.6842105263158	37.3684210526316	0.878989386054995	0.878989386054995\\
73.6842105263158	39.4736842105263	0.8345519363086	0.8345519363086\\
73.6842105263158	41.5789473684211	0.76071074817758	0.76071074817758\\
73.6842105263158	43.6842105263158	0.693211263231226	0.693211263231226\\
73.6842105263158	45.7894736842105	0.596959254161551	0.596959254161551\\
73.6842105263158	47.8947368421053	0.578905804926702	0.578905804926702\\
73.6842105263158	50	0.563939086544905	0.563939086544905\\
84.2105263157895	10	1	1\\
84.2105263157895	12.1052631578947	1	1\\
84.2105263157895	14.2105263157895	1	1\\
84.2105263157895	16.3157894736842	0.997956816485909	0.997956816485909\\
84.2105263157895	18.4210526315789	0.989416727860281	0.989416727860281\\
84.2105263157895	20.5263157894737	0.981004480993603	0.981004480993603\\
84.2105263157895	22.6315789473684	0.975918653763611	0.975918653763611\\
84.2105263157895	24.7368421052632	0.969929196977271	0.969929196977271\\
84.2105263157895	26.8421052631579	0.953262074554633	0.953262074554633\\
84.2105263157895	28.9473684210526	0.886230887331906	0.886230887331906\\
84.2105263157895	31.0526315789474	0.846760592479636	0.846760592479636\\
84.2105263157895	33.1578947368421	0.772243975098128	0.772243975098128\\
84.2105263157895	35.2631578947368	0.700055610042279	0.700055610042279\\
84.2105263157895	37.3684210526316	0.67024465301331	0.67024465301331\\
84.2105263157895	39.4736842105263	0.627319620442309	0.627319620442309\\
84.2105263157895	41.5789473684211	0.551639253990286	0.551639253990286\\
84.2105263157895	43.6842105263158	0.534250639346621	0.534250639346621\\
84.2105263157895	45.7894736842105	0.48330052287771	0.48330052287771\\
84.2105263157895	47.8947368421053	0.513130185863866	0.513130185863866\\
84.2105263157895	50	0.490176829143715	0.490176829143715\\
94.7368421052632	10	0.965471565526666	0.965471565526666\\
94.7368421052632	12.1052631578947	0.953090232891231	0.953090232891231\\
94.7368421052632	14.2105263157895	0.957628143830984	0.957628143830984\\
94.7368421052632	16.3157894736842	0.973346125974589	0.973346125974589\\
94.7368421052632	18.4210526315789	0.965578165278907	0.965578165278907\\
94.7368421052632	20.5263157894737	0.961222318376529	0.961222318376529\\
94.7368421052632	22.6315789473684	0.93951643461648	0.93951643461648\\
94.7368421052632	24.7368421052632	0.921640980312527	0.921640980312527\\
94.7368421052632	26.8421052631579	0.856913870254548	0.856913870254548\\
94.7368421052632	28.9473684210526	0.796373096641393	0.796373096641393\\
94.7368421052632	31.0526315789474	0.7134173830065	0.7134173830065\\
94.7368421052632	33.1578947368421	0.691212082198436	0.691212082198436\\
94.7368421052632	35.2631578947368	0.6413610643773	0.6413610643773\\
94.7368421052632	37.3684210526316	0.627453311147455	0.627453311147455\\
94.7368421052632	39.4736842105263	0.570736642864848	0.570736642864848\\
94.7368421052632	41.5789473684211	0.521366068566528	0.521366068566528\\
94.7368421052632	43.6842105263158	0.497264108225511	0.497264108225511\\
94.7368421052632	45.7894736842105	0.46159571994826	0.46159571994826\\
94.7368421052632	47.8947368421053	0.479484997863816	0.479484997863816\\
94.7368421052632	50	0.488209070393541	0.488209070393541\\
105.263157894737	10	0.993649567705964	0.993649567705964\\
105.263157894737	12.1052631578947	0.993221655557545	0.993221655557545\\
105.263157894737	14.2105263157895	0.992958158333487	0.992958158333487\\
105.263157894737	16.3157894736842	0.955774335977882	0.955774335977882\\
105.263157894737	18.4210526315789	0.990926562183403	0.990926562183403\\
105.263157894737	20.5263157894737	0.987036577758436	0.987036577758436\\
105.263157894737	22.6315789473684	0.906364287928689	0.906364287928689\\
105.263157894737	24.7368421052632	0.811875594104598	0.811875594104598\\
105.263157894737	26.8421052631579	0.793717439856148	0.793717439856148\\
105.263157894737	28.9473684210526	0.701943384139175	0.701943384139175\\
105.263157894737	31.0526315789474	0.666477564160311	0.666477564160311\\
105.263157894737	33.1578947368421	0.610375253283778	0.610375253283778\\
105.263157894737	35.2631578947368	0.555899014100837	0.555899014100837\\
105.263157894737	37.3684210526316	0.542992997213338	0.542992997213338\\
105.263157894737	39.4736842105263	0.529625433700057	0.529625433700057\\
105.263157894737	41.5789473684211	0.492995534739471	0.492995534739471\\
105.263157894737	43.6842105263158	0.454750185187254	0.454750185187254\\
105.263157894737	45.7894736842105	0.456920322673447	0.456920322673447\\
105.263157894737	47.8947368421053	0.436932858124237	0.436932858124237\\
105.263157894737	50	0.428733410896714	0.428733410896714\\
115.789473684211	10	0.970664188615034	0.970664188615034\\
115.789473684211	12.1052631578947	0.9708696817841	0.9708696817841\\
115.789473684211	14.2105263157895	0.963214025861138	0.963214025861138\\
115.789473684211	16.3157894736842	0.921062576077832	0.921062576077832\\
115.789473684211	18.4210526315789	0.903011770455928	0.903011770455928\\
115.789473684211	20.5263157894737	0.855567466934371	0.855567466934371\\
115.789473684211	22.6315789473684	0.762149581847738	0.762149581847738\\
115.789473684211	24.7368421052632	0.723414794125235	0.723414794125235\\
115.789473684211	26.8421052631579	0.693285271568328	0.693285271568328\\
115.789473684211	28.9473684210526	0.631228914055653	0.631228914055653\\
115.789473684211	31.0526315789474	0.59642284340339	0.59642284340339\\
115.789473684211	33.1578947368421	0.585271760231243	0.585271760231243\\
115.789473684211	35.2631578947368	0.530638968830553	0.530638968830553\\
115.789473684211	37.3684210526316	0.523266726141477	0.523266726141477\\
115.789473684211	39.4736842105263	0.489747432736413	0.489747432736413\\
115.789473684211	41.5789473684211	0.47629974099849	0.47629974099849\\
115.789473684211	43.6842105263158	0.437759264391671	0.437759264391671\\
115.789473684211	45.7894736842105	0.394296723622083	0.394296723622083\\
115.789473684211	47.8947368421053	0.354789390840355	0.354789390840355\\
115.789473684211	50	0.335924303083423	0.335924303083423\\
126.315789473684	10	0.955947208567482	0.955947208567482\\
126.315789473684	12.1052631578947	0.956364440403624	0.956364440403624\\
126.315789473684	14.2105263157895	0.947586556703916	0.947586556703916\\
126.315789473684	16.3157894736842	0.866397201532446	0.866397201532446\\
126.315789473684	18.4210526315789	0.812951908081706	0.812951908081706\\
126.315789473684	20.5263157894737	0.746577505320111	0.746577505320111\\
126.315789473684	22.6315789473684	0.713754634539068	0.713754634539068\\
126.315789473684	24.7368421052632	0.6692835362432	0.6692835362432\\
126.315789473684	26.8421052631579	0.595015885885808	0.595015885885808\\
126.315789473684	28.9473684210526	0.579053097938417	0.579053097938417\\
126.315789473684	31.0526315789474	0.599226726115782	0.599226726115782\\
126.315789473684	33.1578947368421	0.557464966850613	0.557464966850613\\
126.315789473684	35.2631578947368	0.435582477548573	0.435582477548573\\
126.315789473684	37.3684210526316	0.415915427779791	0.415915427779791\\
126.315789473684	39.4736842105263	0.373368408894585	0.373368408894585\\
126.315789473684	41.5789473684211	0.308070537529686	0.308070537529686\\
126.315789473684	43.6842105263158	0.215385079934606	0.215385079934606\\
126.315789473684	45.7894736842105	0.263124791298107	0.263124791298107\\
126.315789473684	47.8947368421053	0.271336238028219	0.271336238028219\\
126.315789473684	50	0.287577586241035	0.287577586241035\\
136.842105263158	10	0.748125622203239	0.748125622203239\\
136.842105263158	12.1052631578947	0.720225820910402	0.720225820910402\\
136.842105263158	14.2105263157895	0.862305067282202	0.862305067282202\\
136.842105263158	16.3157894736842	0.778514414424461	0.778514414424461\\
136.842105263158	18.4210526315789	0.688723821393407	0.688723821393407\\
136.842105263158	20.5263157894737	0.566797429037478	0.566797429037478\\
136.842105263158	22.6315789473684	0.529240239112663	0.529240239112663\\
136.842105263158	24.7368421052632	0.560431268338304	0.560431268338304\\
136.842105263158	26.8421052631579	0.476006719191619	0.476006719191619\\
136.842105263158	28.9473684210526	0.410985921030966	0.410985921030966\\
136.842105263158	31.0526315789474	0.395418951198522	0.395418951198522\\
136.842105263158	33.1578947368421	0.327546877286884	0.327546877286884\\
136.842105263158	35.2631578947368	0.301929243734022	0.301929243734022\\
136.842105263158	37.3684210526316	0.289939724336188	0.289939724336188\\
136.842105263158	39.4736842105263	0.267341578983395	0.267341578983395\\
136.842105263158	41.5789473684211	0.251563373762369	0.251563373762369\\
136.842105263158	43.6842105263158	0.26006052877639	0.26006052877639\\
136.842105263158	45.7894736842105	0.251379147099996	0.251379147099996\\
136.842105263158	47.8947368421053	0.258437824108129	0.258437824108129\\
136.842105263158	50	0.202596501947713	0.202596501947713\\
147.368421052632	10	0.505874064601466	0.505874064601466\\
147.368421052632	12.1052631578947	0.418770775540436	0.418770775540436\\
147.368421052632	14.2105263157895	0.699408654426451	0.699408654426451\\
147.368421052632	16.3157894736842	0.563387417109017	0.563387417109017\\
147.368421052632	18.4210526315789	0.449508306899617	0.449508306899617\\
147.368421052632	20.5263157894737	0.378800506225562	0.378800506225562\\
147.368421052632	22.6315789473684	0.276041348057038	0.276041348057038\\
147.368421052632	24.7368421052632	0.273650937125591	0.273650937125591\\
147.368421052632	26.8421052631579	0.23166434344011	0.23166434344011\\
147.368421052632	28.9473684210526	0.214346384330305	0.214346384330305\\
147.368421052632	31.0526315789474	0.20035183943432	0.20035183943432\\
147.368421052632	33.1578947368421	0.202786587047409	0.202786587047409\\
147.368421052632	35.2631578947368	0.17404731547874	0.17404731547874\\
147.368421052632	37.3684210526316	0.165056793186768	0.165056793186768\\
147.368421052632	39.4736842105263	0.148436329356419	0.148436329356419\\
147.368421052632	41.5789473684211	0.142913037315711	0.142913037315711\\
147.368421052632	43.6842105263158	0.153831484422762	0.153831484422762\\
147.368421052632	45.7894736842105	0.186921941077087	0.186921941077087\\
147.368421052632	47.8947368421053	0.182653633500038	0.182653633500038\\
147.368421052632	50	0.186842683061489	0.186842683061489\\
157.894736842105	10	0.589792214434169	0.589792214434169\\
157.894736842105	12.1052631578947	0.427631543534895	0.427631543534895\\
157.894736842105	14.2105263157895	0.469011938180837	0.469011938180837\\
157.894736842105	16.3157894736842	0.47005693811153	0.47005693811153\\
157.894736842105	18.4210526315789	0.470904627666823	0.470904627666823\\
157.894736842105	20.5263157894737	0.459778101778348	0.459778101778348\\
157.894736842105	22.6315789473684	0.409697821873664	0.409697821873664\\
157.894736842105	24.7368421052632	0.351214462056213	0.351214462056213\\
157.894736842105	26.8421052631579	0.308245669288453	0.308245669288453\\
157.894736842105	28.9473684210526	0.300467306921778	0.300467306921778\\
157.894736842105	31.0526315789474	0.276839391719612	0.276839391719612\\
157.894736842105	33.1578947368421	0.259450673244519	0.259450673244519\\
157.894736842105	35.2631578947368	0.234994739411666	0.234994739411666\\
157.894736842105	37.3684210526316	0.196281762512378	0.196281762512378\\
157.894736842105	39.4736842105263	0.223878345870435	0.223878345870435\\
157.894736842105	41.5789473684211	0.25101531268277	0.25101531268277\\
157.894736842105	43.6842105263158	0.236982784799541	0.236982784799541\\
157.894736842105	45.7894736842105	0.225251128529834	0.225251128529834\\
157.894736842105	47.8947368421053	0.229826251243122	0.229826251243122\\
157.894736842105	50	0.219104506722942	0.219104506722942\\
168.421052631579	10	0.229580467966692	0.229580467966692\\
168.421052631579	12.1052631578947	0.244222339969151	0.244222339969151\\
168.421052631579	14.2105263157895	0.249451483472359	0.249451483472359\\
168.421052631579	16.3157894736842	0.247012543180038	0.247012543180038\\
168.421052631579	18.4210526315789	0.231542637129162	0.231542637129162\\
168.421052631579	20.5263157894737	0.25480193082707	0.25480193082707\\
168.421052631579	22.6315789473684	0.213180058128758	0.213180058128758\\
168.421052631579	24.7368421052632	0.144910457959207	0.144910457959207\\
168.421052631579	26.8421052631579	0.104223014907547	0.104223014907547\\
168.421052631579	28.9473684210526	0.0697152774550354	0.0697152774550354\\
168.421052631579	31.0526315789474	0.0592707376865219	0.0592707376865219\\
168.421052631579	33.1578947368421	0.0866823153658172	0.0866823153658172\\
168.421052631579	35.2631578947368	0.113481683134266	0.113481683134266\\
168.421052631579	37.3684210526316	0.114346282453818	0.114346282453818\\
168.421052631579	39.4736842105263	0.112159407113595	0.112159407113595\\
168.421052631579	41.5789473684211	0.122620336857794	0.122620336857794\\
168.421052631579	43.6842105263158	0.137678802075263	0.137678802075263\\
168.421052631579	45.7894736842105	0.169374991820386	0.169374991820386\\
168.421052631579	47.8947368421053	0.177314497435078	0.177314497435078\\
168.421052631579	50	0.155293417167355	0.155293417167355\\
178.947368421053	10	0.229266428522709	0.229266428522709\\
178.947368421053	12.1052631578947	0.263310835766098	0.263310835766098\\
178.947368421053	14.2105263157895	0.298950456193628	0.298950456193628\\
178.947368421053	16.3157894736842	0.32208674008685	0.32208674008685\\
178.947368421053	18.4210526315789	0.328084029491464	0.328084029491464\\
178.947368421053	20.5263157894737	0.311069951819398	0.311069951819398\\
178.947368421053	22.6315789473684	0.265668036745028	0.265668036745028\\
178.947368421053	24.7368421052632	0.234652125539105	0.234652125539105\\
178.947368421053	26.8421052631579	0.224784269157038	0.224784269157038\\
178.947368421053	28.9473684210526	0.189590367359748	0.189590367359748\\
178.947368421053	31.0526315789474	0.162315589726794	0.162315589726794\\
178.947368421053	33.1578947368421	0.153468956807306	0.153468956807306\\
178.947368421053	35.2631578947368	0.124168475918214	0.124168475918214\\
178.947368421053	37.3684210526316	0.0657208432540292	0.0657208432540292\\
178.947368421053	39.4736842105263	0.0923655633089084	0.0923655633089084\\
178.947368421053	41.5789473684211	0.0992497953333096	0.0992497953333096\\
178.947368421053	43.6842105263158	0.103348525651686	0.103348525651686\\
178.947368421053	45.7894736842105	0.118238629915711	0.118238629915711\\
178.947368421053	47.8947368421053	0.119286995011497	0.119286995011497\\
178.947368421053	50	0.133262967355136	0.133262967355136\\
189.473684210526	10	0.133587495478149	0.133587495478149\\
189.473684210526	12.1052631578947	0.161781358097798	0.161781358097798\\
189.473684210526	14.2105263157895	0.144229316448504	0.144229316448504\\
189.473684210526	16.3157894736842	0.129625258532356	0.129625258532356\\
189.473684210526	18.4210526315789	0.114808489300078	0.114808489300078\\
189.473684210526	20.5263157894737	0.10337740043515	0.10337740043515\\
189.473684210526	22.6315789473684	0.0899631353347369	0.0899631353347369\\
189.473684210526	24.7368421052632	0.0797790230804172	0.0797790230804172\\
189.473684210526	26.8421052631579	0.07815923830272	0.07815923830272\\
189.473684210526	28.9473684210526	0.0767720143226281	0.0767720143226281\\
189.473684210526	31.0526315789474	0.0756321659192112	0.0756321659192112\\
189.473684210526	33.1578947368421	0.0777666674844577	0.0777666674844577\\
189.473684210526	35.2631578947368	0.0723152217635659	0.0723152217635659\\
189.473684210526	37.3684210526316	0.0680802077671844	0.0680802077671844\\
189.473684210526	39.4736842105263	0.0646425285153203	0.0646425285153203\\
189.473684210526	41.5789473684211	0.0605216927024574	0.0605216927024574\\
189.473684210526	43.6842105263158	0.0574219111595005	0.0574219111595005\\
189.473684210526	45.7894736842105	0.0880296785102865	0.0880296785102865\\
189.473684210526	47.8947368421053	0.0830110221376451	0.0830110221376451\\
189.473684210526	50	0.0964154120234759	0.0964154120234759\\
200	10	0.4	0.4\\
200	12.1052631578947	0.296790876467918	0.296790876467918\\
200	14.2105263157895	0.220178618396216	0.220178618396216\\
200	16.3157894736842	0.158806835540799	0.158806835540799\\
200	18.4210526315789	0.0973099378179244	0.0973099378179244\\
200	20.5263157894737	0.0682821290655308	0.0682821290655308\\
200	22.6315789473684	0.0931346951081402	0.0931346951081402\\
200	24.7368421052632	0.112410698753408	0.112410698753408\\
200	26.8421052631579	0.124430124662115	0.124430124662115\\
200	28.9473684210526	0.134329946967128	0.134329946967128\\
200	31.0526315789474	0.152086459157109	0.152086459157109\\
200	33.1578947368421	0.14224043839679	0.14224043839679\\
200	35.2631578947368	0.131339599399507	0.131339599399507\\
200	37.3684210526316	0.104586611870903	0.104586611870903\\
200	39.4736842105263	0.10333272604241	0.10333272604241\\
200	41.5789473684211	0.0836081695038572	0.0836081695038572\\
200	43.6842105263158	0.0694629653375194	0.0694629653375194\\
200	45.7894736842105	0.0550257829482321	0.0550257829482321\\
200	47.8947368421053	0.0408585578233702	0.0408585578233702\\
200	50	0.0523169575787885	0.0523169575787885\\
};
\end{axis}
\end{tikzpicture}%

%% file: rwa25.tex
%
%
\begin{tikzpicture}

\begin{axis}[%
width=0.951\fwidth,
height=\fheight,
at={(0\fwidth,0\fheight)},
scale only axis,
xmin=0,
xmax=200,
tick align=outside,
xlabel={$\text{N}_\text{P}\text{(0)}$},
xmajorgrids,
ymin=10,
ymax=50,
ylabel={$\text{N}_\text{T}\text{(0)}$},
ymajorgrids,
zmin=0,
zmax=1,
zlabel={c(25)},
zmajorgrids,
view={45}{30},
axis background/.style={fill=white},
axis x line*=bottom,
axis y line*=left,
axis z line*=left
]

\addplot3[%
surf,
shader=flat corner,draw=black,z buffer=sort,colormap/blackwhite,mesh/rows=11]
table[row sep=crcr, point meta=\thisrow{c}] {%
x	y	z	c\\
0	10	1	1\\
0	20	1	1\\
0	30	1	1\\
0	40	1	1\\
0	50	1	1\\
20	10	1	1\\
20	20	1	1\\
20	30	1	1\\
20	40	1	1\\
20	50	0.978723404255319	0.978723404255319\\
40	10	1	1\\
40	20	1	1\\
40	30	1	1\\
40	40	1	1\\
40	50	1	1\\
60	10	1	1\\
60	20	1	1\\
60	30	1	1\\
60	40	1	1\\
60	50	0.84468085106383	0.84468085106383\\
80	10	1	1\\
80	20	1	1\\
80	30	0.960714285714286	0.960714285714286\\
80	40	0.721052631578947	0.721052631578947\\
80	50	0.51063829787234	0.51063829787234\\
100	10	1	1\\
100	20	1	1\\
100	30	0.689285714285714	0.689285714285714\\
100	40	0.589473684210526	0.589473684210526\\
100	50	0.497872340425532	0.497872340425532\\
120	10	1	1\\
120	20	0.910526315789474	0.910526315789474\\
120	30	0.639285714285714	0.639285714285714\\
120	40	0.5	0.5\\
120	50	0.272340425531915	0.272340425531915\\
140	10	0.7	0.7\\
140	20	0.49	0.49\\
140	30	0.367857142857143	0.367857142857143\\
140	40	0.253846153846154	0.253846153846154\\
140	50	0.218367346938775	0.218367346938775\\
160	10	0.466666666666667	0.466666666666667\\
160	20	0.489473684210526	0.489473684210526\\
160	30	0.314285714285714	0.314285714285714\\
160	40	0.223684210526316	0.223684210526316\\
160	50	0.204255319148936	0.204255319148936\\
180	10	0.22	0.22\\
180	20	0.36	0.36\\
180	30	0.16	0.16\\
180	40	0.0950000000000001	0.0950000000000001\\
180	50	0.110204081632653	0.110204081632653\\
200	10	0.4	0.4\\
200	20	0.147619047619048	0.147619047619048\\
200	30	0.156666666666667	0.156666666666667\\
200	40	0.1025	0.1025\\
200	50	0.0784313725490197	0.0784313725490197\\
};
\end{axis}
\end{tikzpicture}%

%% file: rwsa50.tex
%
%
\begin{tikzpicture}

\begin{axis}[%
width=0.939\fwidth,
height=\fheight,
at={(0\fwidth,0\fheight)},
scale only axis,
xmin=0,
xmax=200,
tick align=outside,
xlabel={$\text{N}_\text{P}\text{(0)}$},
xmajorgrids,
ymin=10,
ymax=50,
ylabel={$\text{N}_\text{T}\text{(0)}$},
ymajorgrids,
zmin=0,
zmax=1,
zlabel={c(50)},
zmajorgrids,
view={45}{30},
axis background/.style={fill=white},
axis x line*=bottom,
axis y line*=left,
axis z line*=left
]

\addplot3[%
surf,
shader=flat corner,draw=black,z buffer=sort,colormap/blackwhite,mesh/rows=20]
table[row sep=crcr, point meta=\thisrow{c}] {%
x	y	z	c\\
0	10	1	1\\
0	12.1052631578947	1	1\\
0	14.2105263157895	1	1\\
0	16.3157894736842	1	1\\
0	18.4210526315789	1	1\\
0	20.5263157894737	1	1\\
0	22.6315789473684	1	1\\
0	24.7368421052632	1	1\\
0	26.8421052631579	1	1\\
0	28.9473684210526	1	1\\
0	31.0526315789474	1	1\\
0	33.1578947368421	1	1\\
0	35.2631578947368	1	1\\
0	37.3684210526316	1	1\\
0	39.4736842105263	1	1\\
0	41.5789473684211	1	1\\
0	43.6842105263158	1	1\\
0	45.7894736842105	1	1\\
0	47.8947368421053	1	1\\
0	50	1	1\\
10.5263157894737	10	1	1\\
10.5263157894737	12.1052631578947	1	1\\
10.5263157894737	14.2105263157895	1	1\\
10.5263157894737	16.3157894736842	1	1\\
10.5263157894737	18.4210526315789	1	1\\
10.5263157894737	20.5263157894737	1	1\\
10.5263157894737	22.6315789473684	1	1\\
10.5263157894737	24.7368421052632	1	1\\
10.5263157894737	26.8421052631579	1	1\\
10.5263157894737	28.9473684210526	1	1\\
10.5263157894737	31.0526315789474	1	1\\
10.5263157894737	33.1578947368421	1	1\\
10.5263157894737	35.2631578947368	1	1\\
10.5263157894737	37.3684210526316	1	1\\
10.5263157894737	39.4736842105263	1	1\\
10.5263157894737	41.5789473684211	1	1\\
10.5263157894737	43.6842105263158	1	1\\
10.5263157894737	45.7894736842105	1	1\\
10.5263157894737	47.8947368421053	1	1\\
10.5263157894737	50	1	1\\
21.0526315789474	10	1	1\\
21.0526315789474	12.1052631578947	1	1\\
21.0526315789474	14.2105263157895	1	1\\
21.0526315789474	16.3157894736842	1	1\\
21.0526315789474	18.4210526315789	1	1\\
21.0526315789474	20.5263157894737	1	1\\
21.0526315789474	22.6315789473684	1	1\\
21.0526315789474	24.7368421052632	1	1\\
21.0526315789474	26.8421052631579	1	1\\
21.0526315789474	28.9473684210526	1	1\\
21.0526315789474	31.0526315789474	1	1\\
21.0526315789474	33.1578947368421	1	1\\
21.0526315789474	35.2631578947368	1	1\\
21.0526315789474	37.3684210526316	1	1\\
21.0526315789474	39.4736842105263	1	1\\
21.0526315789474	41.5789473684211	1	1\\
21.0526315789474	43.6842105263158	1	1\\
21.0526315789474	45.7894736842105	1	1\\
21.0526315789474	47.8947368421053	1	1\\
21.0526315789474	50	1	1\\
31.5789473684211	10	1	1\\
31.5789473684211	12.1052631578947	1	1\\
31.5789473684211	14.2105263157895	1	1\\
31.5789473684211	16.3157894736842	1	1\\
31.5789473684211	18.4210526315789	1	1\\
31.5789473684211	20.5263157894737	1	1\\
31.5789473684211	22.6315789473684	1	1\\
31.5789473684211	24.7368421052632	1	1\\
31.5789473684211	26.8421052631579	1	1\\
31.5789473684211	28.9473684210526	1	1\\
31.5789473684211	31.0526315789474	1	1\\
31.5789473684211	33.1578947368421	1	1\\
31.5789473684211	35.2631578947368	1	1\\
31.5789473684211	37.3684210526316	1	1\\
31.5789473684211	39.4736842105263	1	1\\
31.5789473684211	41.5789473684211	1	1\\
31.5789473684211	43.6842105263158	1	1\\
31.5789473684211	45.7894736842105	1	1\\
31.5789473684211	47.8947368421053	1	1\\
31.5789473684211	50	1	1\\
42.1052631578947	10	1	1\\
42.1052631578947	12.1052631578947	1	1\\
42.1052631578947	14.2105263157895	1	1\\
42.1052631578947	16.3157894736842	1	1\\
42.1052631578947	18.4210526315789	1	1\\
42.1052631578947	20.5263157894737	1	1\\
42.1052631578947	22.6315789473684	1	1\\
42.1052631578947	24.7368421052632	1	1\\
42.1052631578947	26.8421052631579	1	1\\
42.1052631578947	28.9473684210526	1	1\\
42.1052631578947	31.0526315789474	1	1\\
42.1052631578947	33.1578947368421	1	1\\
42.1052631578947	35.2631578947368	1	1\\
42.1052631578947	37.3684210526316	1	1\\
42.1052631578947	39.4736842105263	1	1\\
42.1052631578947	41.5789473684211	1	1\\
42.1052631578947	43.6842105263158	1	1\\
42.1052631578947	45.7894736842105	1	1\\
42.1052631578947	47.8947368421053	1	1\\
42.1052631578947	50	1	1\\
52.6315789473684	10	1	1\\
52.6315789473684	12.1052631578947	1	1\\
52.6315789473684	14.2105263157895	1	1\\
52.6315789473684	16.3157894736842	1	1\\
52.6315789473684	18.4210526315789	1	1\\
52.6315789473684	20.5263157894737	1	1\\
52.6315789473684	22.6315789473684	1	1\\
52.6315789473684	24.7368421052632	1	1\\
52.6315789473684	26.8421052631579	1	1\\
52.6315789473684	28.9473684210526	1	1\\
52.6315789473684	31.0526315789474	1	1\\
52.6315789473684	33.1578947368421	1	1\\
52.6315789473684	35.2631578947368	1	1\\
52.6315789473684	37.3684210526316	1	1\\
52.6315789473684	39.4736842105263	1	1\\
52.6315789473684	41.5789473684211	1	1\\
52.6315789473684	43.6842105263158	1	1\\
52.6315789473684	45.7894736842105	1	1\\
52.6315789473684	47.8947368421053	0.99934703457537	0.99934703457537\\
52.6315789473684	50	0.99564094854831	0.99564094854831\\
63.1578947368421	10	1	1\\
63.1578947368421	12.1052631578947	1	1\\
63.1578947368421	14.2105263157895	1	1\\
63.1578947368421	16.3157894736842	1	1\\
63.1578947368421	18.4210526315789	1	1\\
63.1578947368421	20.5263157894737	1	1\\
63.1578947368421	22.6315789473684	1	1\\
63.1578947368421	24.7368421052632	1	1\\
63.1578947368421	26.8421052631579	1	1\\
63.1578947368421	28.9473684210526	1	1\\
63.1578947368421	31.0526315789474	1	1\\
63.1578947368421	33.1578947368421	1	1\\
63.1578947368421	35.2631578947368	1	1\\
63.1578947368421	37.3684210526316	0.999089443578861	0.999089443578861\\
63.1578947368421	39.4736842105263	0.995632715897508	0.995632715897508\\
63.1578947368421	41.5789473684211	0.991978054135638	0.991978054135638\\
63.1578947368421	43.6842105263158	0.989099645392551	0.989099645392551\\
63.1578947368421	45.7894736842105	0.985975506843074	0.985975506843074\\
63.1578947368421	47.8947368421053	0.998824672057436	0.998824672057436\\
63.1578947368421	50	0.999326450134031	0.999326450134031\\
73.6842105263158	10	1	1\\
73.6842105263158	12.1052631578947	1	1\\
73.6842105263158	14.2105263157895	1	1\\
73.6842105263158	16.3157894736842	1	1\\
73.6842105263158	18.4210526315789	1	1\\
73.6842105263158	20.5263157894737	1	1\\
73.6842105263158	22.6315789473684	1	1\\
73.6842105263158	24.7368421052632	1	1\\
73.6842105263158	26.8421052631579	0.99881674195326	0.99881674195326\\
73.6842105263158	28.9473684210526	0.993853290294	0.993853290294\\
73.6842105263158	31.0526315789474	0.989510904349998	0.989510904349998\\
73.6842105263158	33.1578947368421	0.986622842292673	0.986622842292673\\
73.6842105263158	35.2631578947368	0.983581809714111	0.983581809714111\\
73.6842105263158	37.3684210526316	0.999721381432548	0.999721381432548\\
73.6842105263158	39.4736842105263	0.999182847933752	0.999182847933752\\
73.6842105263158	41.5789473684211	0.998727502322218	0.998727502322218\\
73.6842105263158	43.6842105263158	0.998311509960942	0.998311509960942\\
73.6842105263158	45.7894736842105	0.997924899168452	0.997924899168452\\
73.6842105263158	47.8947368421053	0.984595549141105	0.984595549141105\\
73.6842105263158	50	0.996605627861638	0.996605627861638\\
84.2105263157895	10	1	1\\
84.2105263157895	12.1052631578947	1	1\\
84.2105263157895	14.2105263157895	1	1\\
84.2105263157895	16.3157894736842	0.997956816485909	0.997956816485909\\
84.2105263157895	18.4210526315789	0.989416727860281	0.989416727860281\\
84.2105263157895	20.5263157894737	0.981004480993603	0.981004480993603\\
84.2105263157895	22.6315789473684	0.975918653763611	0.975918653763611\\
84.2105263157895	24.7368421052632	0.969929196977271	0.969929196977271\\
84.2105263157895	26.8421052631579	0.999640188137794	0.999640188137794\\
84.2105263157895	28.9473684210526	0.998891586810197	0.998891586810197\\
84.2105263157895	31.0526315789474	0.998342990122615	0.998342990122615\\
84.2105263157895	33.1578947368421	0.997732848243876	0.997732848243876\\
84.2105263157895	35.2631578947368	0.997258371724727	0.997258371724727\\
84.2105263157895	37.3684210526316	0.998250862466489	0.998250862466489\\
84.2105263157895	39.4736842105263	0.996142290379046	0.996142290379046\\
84.2105263157895	41.5789473684211	0.994073179371609	0.994073179371609\\
84.2105263157895	43.6842105263158	0.992302302398994	0.992302302398994\\
84.2105263157895	45.7894736842105	0.975546565024747	0.975546565024747\\
84.2105263157895	47.8947368421053	0.966729573655879	0.966729573655879\\
84.2105263157895	50	0.936453417451731	0.936453417451731\\
94.7368421052632	10	0.965471565526666	0.965471565526666\\
94.7368421052632	12.1052631578947	0.953090232891231	0.953090232891231\\
94.7368421052632	14.2105263157895	0.957628143830984	0.957628143830984\\
94.7368421052632	16.3157894736842	0.999329605798924	0.999329605798924\\
94.7368421052632	18.4210526315789	0.998121990315111	0.998121990315111\\
94.7368421052632	20.5263157894737	0.997161466459033	0.997161466459033\\
94.7368421052632	22.6315789473684	0.996455005153262	0.996455005153262\\
94.7368421052632	24.7368421052632	0.995795816278576	0.995795816278576\\
94.7368421052632	26.8421052631579	0.997484188891748	0.997484188891748\\
94.7368421052632	28.9473684210526	0.994816738824095	0.994816738824095\\
94.7368421052632	31.0526315789474	0.992487137399009	0.992487137399009\\
94.7368421052632	33.1578947368421	0.990458067715674	0.990458067715674\\
94.7368421052632	35.2631578947368	0.988429944709752	0.988429944709752\\
94.7368421052632	37.3684210526316	0.987741653008269	0.987741653008269\\
94.7368421052632	39.4736842105263	0.976550910006692	0.976550910006692\\
94.7368421052632	41.5789473684211	0.954522799744673	0.954522799744673\\
94.7368421052632	43.6842105263158	0.927838005033253	0.927838005033253\\
94.7368421052632	45.7894736842105	0.899596062647589	0.899596062647589\\
94.7368421052632	47.8947368421053	0.901419027000405	0.901419027000405\\
94.7368421052632	50	0.902098651646113	0.902098651646113\\
105.263157894737	10	0.993649567705964	0.993649567705964\\
105.263157894737	12.1052631578947	0.993221655557545	0.993221655557545\\
105.263157894737	14.2105263157895	0.992958158333487	0.992958158333487\\
105.263157894737	16.3157894736842	0.994407423990797	0.994407423990797\\
105.263157894737	18.4210526315789	0.990926562183403	0.990926562183403\\
105.263157894737	20.5263157894737	0.987036577758436	0.987036577758436\\
105.263157894737	22.6315789473684	0.98450190470002	0.98450190470002\\
105.263157894737	24.7368421052632	0.978454230924568	0.978454230924568\\
105.263157894737	26.8421052631579	0.997804896993747	0.997804896993747\\
105.263157894737	28.9473684210526	0.995663995070808	0.995663995070808\\
105.263157894737	31.0526315789474	0.994436321214353	0.994436321214353\\
105.263157894737	33.1578947368421	0.992023059994686	0.992023059994686\\
105.263157894737	35.2631578947368	0.98052909500081	0.98052909500081\\
105.263157894737	37.3684210526316	0.954312218241395	0.954312218241395\\
105.263157894737	39.4736842105263	0.93301279807349	0.93301279807349\\
105.263157894737	41.5789473684211	0.909507191133911	0.909507191133911\\
105.263157894737	43.6842105263158	0.884596258809995	0.884596258809995\\
105.263157894737	45.7894736842105	0.882425836628013	0.882425836628013\\
105.263157894737	47.8947368421053	0.807842876339792	0.807842876339792\\
105.263157894737	50	0.800200578273895	0.800200578273895\\
115.789473684211	10	0.970664188615034	0.970664188615034\\
115.789473684211	12.1052631578947	0.9708696817841	0.9708696817841\\
115.789473684211	14.2105263157895	0.963214025861138	0.963214025861138\\
115.789473684211	16.3157894736842	0.998636555570282	0.998636555570282\\
115.789473684211	18.4210526315789	0.997481008291705	0.997481008291705\\
115.789473684211	20.5263157894737	0.994939876663248	0.994939876663248\\
115.789473684211	22.6315789473684	0.995910191308278	0.995910191308278\\
115.789473684211	24.7368421052632	0.996113848301053	0.996113848301053\\
115.789473684211	26.8421052631579	0.958997073935194	0.958997073935194\\
115.789473684211	28.9473684210526	0.934212912928248	0.934212912928248\\
115.789473684211	31.0526315789474	0.916305672827727	0.916305672827727\\
115.789473684211	33.1578947368421	0.896346182425471	0.896346182425471\\
115.789473684211	35.2631578947368	0.945035143475921	0.945035143475921\\
115.789473684211	37.3684210526316	0.909128451131597	0.909128451131597\\
115.789473684211	39.4736842105263	0.866711326124476	0.866711326124476\\
115.789473684211	41.5789473684211	0.794350718156002	0.794350718156002\\
115.789473684211	43.6842105263158	0.755492598069656	0.755492598069656\\
115.789473684211	45.7894736842105	0.742749024626394	0.742749024626394\\
115.789473684211	47.8947368421053	0.673661579880811	0.673661579880811\\
115.789473684211	50	0.640368236431221	0.640368236431221\\
126.315789473684	10	0.992511985588702	0.992511985588702\\
126.315789473684	12.1052631578947	0.992189395074825	0.992189395074825\\
126.315789473684	14.2105263157895	1	1\\
126.315789473684	16.3157894736842	0.962295705881083	0.962295705881083\\
126.315789473684	18.4210526315789	0.953412571290271	0.953412571290271\\
126.315789473684	20.5263157894737	0.94524936407235	0.94524936407235\\
126.315789473684	22.6315789473684	0.936178368540074	0.936178368540074\\
126.315789473684	24.7368421052632	0.966353519239052	0.966353519239052\\
126.315789473684	26.8421052631579	0.895548529889898	0.895548529889898\\
126.315789473684	28.9473684210526	0.860511753409183	0.860511753409183\\
126.315789473684	31.0526315789474	0.823568525173756	0.823568525173756\\
126.315789473684	33.1578947368421	0.773453057093475	0.773453057093475\\
126.315789473684	35.2631578947368	0.835825471955096	0.835825471955096\\
126.315789473684	37.3684210526316	0.752978969930976	0.752978969930976\\
126.315789473684	39.4736842105263	0.678800300387749	0.678800300387749\\
126.315789473684	41.5789473684211	0.585490544412886	0.585490544412886\\
126.315789473684	43.6842105263158	0.504605242901444	0.504605242901444\\
126.315789473684	45.7894736842105	0.508303413164777	0.508303413164777\\
126.315789473684	47.8947368421053	0.490860679072935	0.490860679072935\\
126.315789473684	50	0.475406058688378	0.475406058688378\\
136.842105263158	10	0.893440016760008	0.893440016760008\\
136.842105263158	12.1052631578947	0.886271388067649	0.886271388067649\\
136.842105263158	14.2105263157895	0.979152133348938	0.979152133348938\\
136.842105263158	16.3157894736842	0.955386262732449	0.955386262732449\\
136.842105263158	18.4210526315789	0.926352092157106	0.926352092157106\\
136.842105263158	20.5263157894737	0.780152531177573	0.780152531177573\\
136.842105263158	22.6315789473684	0.737349333702629	0.737349333702629\\
136.842105263158	24.7368421052632	0.95338247699546	0.95338247699546\\
136.842105263158	26.8421052631579	0.848062861563483	0.848062861563483\\
136.842105263158	28.9473684210526	0.697733494099888	0.697733494099888\\
136.842105263158	31.0526315789474	0.638235095007176	0.638235095007176\\
136.842105263158	33.1578947368421	0.542366309808711	0.542366309808711\\
136.842105263158	35.2631578947368	0.509681191205863	0.509681191205863\\
136.842105263158	37.3684210526316	0.480358805436692	0.480358805436692\\
136.842105263158	39.4736842105263	0.445737616802287	0.445737616802287\\
136.842105263158	41.5789473684211	0.432554705359675	0.432554705359675\\
136.842105263158	43.6842105263158	0.420001625478074	0.420001625478074\\
136.842105263158	45.7894736842105	0.404037144012024	0.404037144012024\\
136.842105263158	47.8947368421053	0.386688749834979	0.386688749834979\\
136.842105263158	50	0.319880229395675	0.319880229395675\\
147.368421052632	10	0.918716845492625	0.918716845492625\\
147.368421052632	12.1052631578947	0.904487507247323	0.904487507247323\\
147.368421052632	14.2105263157895	0.939014101920763	0.939014101920763\\
147.368421052632	16.3157894736842	0.925095485730829	0.925095485730829\\
147.368421052632	18.4210526315789	0.765936841693782	0.765936841693782\\
147.368421052632	20.5263157894737	0.684489998383302	0.684489998383302\\
147.368421052632	22.6315789473684	0.479930335887213	0.479930335887213\\
147.368421052632	24.7368421052632	0.476538968860846	0.476538968860846\\
147.368421052632	26.8421052631579	0.432833420376311	0.432833420376311\\
147.368421052632	28.9473684210526	0.396889877827462	0.396889877827462\\
147.368421052632	31.0526315789474	0.346740902124843	0.346740902124843\\
147.368421052632	33.1578947368421	0.317981184796923	0.317981184796923\\
147.368421052632	35.2631578947368	0.274876784761679	0.274876784761679\\
147.368421052632	37.3684210526316	0.255835287260441	0.255835287260441\\
147.368421052632	39.4736842105263	0.228756295899439	0.228756295899439\\
147.368421052632	41.5789473684211	0.213820615186879	0.213820615186879\\
147.368421052632	43.6842105263158	0.236113333847552	0.236113333847552\\
147.368421052632	45.7894736842105	0.29662581918878	0.29662581918878\\
147.368421052632	47.8947368421053	0.364478118711849	0.364478118711849\\
147.368421052632	50	0.387887905740791	0.387887905740791\\
157.894736842105	10	0.749407228097896	0.749407228097896\\
157.894736842105	12.1052631578947	0.694696069146608	0.694696069146608\\
157.894736842105	14.2105263157895	0.619514333520373	0.619514333520373\\
157.894736842105	16.3157894736842	0.613319907648079	0.613319907648079\\
157.894736842105	18.4210526315789	0.620638914413933	0.620638914413933\\
157.894736842105	20.5263157894737	0.617890103750775	0.617890103750775\\
157.894736842105	22.6315789473684	0.605439760600801	0.605439760600801\\
157.894736842105	24.7368421052632	0.577640389598673	0.577640389598673\\
157.894736842105	26.8421052631579	0.551630574276749	0.551630574276749\\
157.894736842105	28.9473684210526	0.520943524361304	0.520943524361304\\
157.894736842105	31.0526315789474	0.481325015109592	0.481325015109592\\
157.894736842105	33.1578947368421	0.436314437183226	0.436314437183226\\
157.894736842105	35.2631578947368	0.39416486497102	0.39416486497102\\
157.894736842105	37.3684210526316	0.367241532643282	0.367241532643282\\
157.894736842105	39.4736842105263	0.35834334526299	0.35834334526299\\
157.894736842105	41.5789473684211	0.344881675160151	0.344881675160151\\
157.894736842105	43.6842105263158	0.338444196352539	0.338444196352539\\
157.894736842105	45.7894736842105	0.336953471340734	0.336953471340734\\
157.894736842105	47.8947368421053	0.337733604161292	0.337733604161292\\
157.894736842105	50	0.319005616821756	0.319005616821756\\
168.421052631579	10	0.42525572467151	0.42525572467151\\
168.421052631579	12.1052631578947	0.487641635493997	0.487641635493997\\
168.421052631579	14.2105263157895	0.500610376324583	0.500610376324583\\
168.421052631579	16.3157894736842	0.486363586514612	0.486363586514612\\
168.421052631579	18.4210526315789	0.47102959765399	0.47102959765399\\
168.421052631579	20.5263157894737	0.448863323131983	0.448863323131983\\
168.421052631579	22.6315789473684	0.424150132322297	0.424150132322297\\
168.421052631579	24.7368421052632	0.365308793673911	0.365308793673911\\
168.421052631579	26.8421052631579	0.273624557199163	0.273624557199163\\
168.421052631579	28.9473684210526	0.162354180111759	0.162354180111759\\
168.421052631579	31.0526315789474	0.105273590654696	0.105273590654696\\
168.421052631579	33.1578947368421	0.128404507676558	0.128404507676558\\
168.421052631579	35.2631578947368	0.16175758391834	0.16175758391834\\
168.421052631579	37.3684210526316	0.17533797955881	0.17533797955881\\
168.421052631579	39.4736842105263	0.191804943920714	0.191804943920714\\
168.421052631579	41.5789473684211	0.213420984378558	0.213420984378558\\
168.421052631579	43.6842105263158	0.227936175071996	0.227936175071996\\
168.421052631579	45.7894736842105	0.255147679674927	0.255147679674927\\
168.421052631579	47.8947368421053	0.266755482892259	0.266755482892259\\
168.421052631579	50	0.254344250083259	0.254344250083259\\
178.947368421053	10	0.338755932723381	0.338755932723381\\
178.947368421053	12.1052631578947	0.382617579606993	0.382617579606993\\
178.947368421053	14.2105263157895	0.436732888074318	0.436732888074318\\
178.947368421053	16.3157894736842	0.468236900184638	0.468236900184638\\
178.947368421053	18.4210526315789	0.477687496927441	0.477687496927441\\
178.947368421053	20.5263157894737	0.473018074332246	0.473018074332246\\
178.947368421053	22.6315789473684	0.42346485928431	0.42346485928431\\
178.947368421053	24.7368421052632	0.363636185017903	0.363636185017903\\
178.947368421053	26.8421052631579	0.312604334196033	0.312604334196033\\
178.947368421053	28.9473684210526	0.267106578766879	0.267106578766879\\
178.947368421053	31.0526315789474	0.217693089081413	0.217693089081413\\
178.947368421053	33.1578947368421	0.198774641635869	0.198774641635869\\
178.947368421053	35.2631578947368	0.17946374812167	0.17946374812167\\
178.947368421053	37.3684210526316	0.168166299983616	0.168166299983616\\
178.947368421053	39.4736842105263	0.171709762270792	0.171709762270792\\
178.947368421053	41.5789473684211	0.1717831036904	0.1717831036904\\
178.947368421053	43.6842105263158	0.170926667267272	0.170926667267272\\
178.947368421053	45.7894736842105	0.171128352936212	0.171128352936212\\
178.947368421053	47.8947368421053	0.168835892050073	0.168835892050073\\
178.947368421053	50	0.171227867758824	0.171227867758824\\
189.473684210526	10	0.256179842877952	0.256179842877952\\
189.473684210526	12.1052631578947	0.286211159426643	0.286211159426643\\
189.473684210526	14.2105263157895	0.346534374926258	0.346534374926258\\
189.473684210526	16.3157894736842	0.302901572302901	0.302901572302901\\
189.473684210526	18.4210526315789	0.273979159936457	0.273979159936457\\
189.473684210526	20.5263157894737	0.246711824448189	0.246711824448189\\
189.473684210526	22.6315789473684	0.214767282871924	0.214767282871924\\
189.473684210526	24.7368421052632	0.189263467998888	0.189263467998888\\
189.473684210526	26.8421052631579	0.179336303836467	0.179336303836467\\
189.473684210526	28.9473684210526	0.176758145377596	0.176758145377596\\
189.473684210526	31.0526315789474	0.174045213036042	0.174045213036042\\
189.473684210526	33.1578947368421	0.178908404048454	0.178908404048454\\
189.473684210526	35.2631578947368	0.175813138181391	0.175813138181391\\
189.473684210526	37.3684210526316	0.16549207549571	0.16549207549571\\
189.473684210526	39.4736842105263	0.155995917997683	0.155995917997683\\
189.473684210526	41.5789473684211	0.146120879451327	0.146120879451327\\
189.473684210526	43.6842105263158	0.134019953644109	0.134019953644109\\
189.473684210526	45.7894736842105	0.16209964251716	0.16209964251716\\
189.473684210526	47.8947368421053	0.173299168120123	0.173299168120123\\
189.473684210526	50	0.189925045458004	0.189925045458004\\
200	10	0.490000000000417	0.490000000000417\\
200	12.1052631578947	0.385469807248588	0.385469807248588\\
200	14.2105263157895	0.300134417388434	0.300134417388434\\
200	16.3157894736842	0.234679067213413	0.234679067213413\\
200	18.4210526315789	0.17836388053565	0.17836388053565\\
200	20.5263157894737	0.142290846079223	0.142290846079223\\
200	22.6315789473684	0.149319179799891	0.149319179799891\\
200	24.7368421052632	0.167859595294043	0.167859595294043\\
200	26.8421052631579	0.182089540938567	0.182089540938567\\
200	28.9473684210526	0.195135880819214	0.195135880819214\\
200	31.0526315789474	0.202946665594226	0.202946665594226\\
200	33.1578947368421	0.193327184385013	0.193327184385013\\
200	35.2631578947368	0.181961674913734	0.181961674913734\\
200	37.3684210526316	0.167383673089088	0.167383673089088\\
200	39.4736842105263	0.14423652483161	0.14423652483161\\
200	41.5789473684211	0.118815946054609	0.118815946054609\\
200	43.6842105263158	0.0990151831432687	0.0990151831432687\\
200	45.7894736842105	0.0805329713003218	0.0805329713003218\\
200	47.8947368421053	0.0623082472319836	0.0623082472319836\\
200	50	0.0614550340017803	0.0614550340017803\\
};
\end{axis}
\end{tikzpicture}%

%% file: rwa50.tex
%
%
\begin{tikzpicture}

\begin{axis}[%
width=0.951\fwidth,
height=\fheight,
at={(0\fwidth,0\fheight)},
scale only axis,
xmin=0,
xmax=200,
tick align=outside,
xlabel={$\text{N}_\text{P}\text{(0)}$},
xmajorgrids,
ymin=10,
ymax=50,
ylabel={$\text{N}_\text{T}\text{(0)}$},
ymajorgrids,
zmin=0,
zmax=1,
zlabel={c(50)},
zmajorgrids,
view={45}{30},
axis background/.style={fill=white},
axis x line*=bottom,
axis y line*=left,
axis z line*=left
]

\addplot3[%
surf,
shader=flat corner,draw=black,z buffer=sort,colormap/blackwhite,mesh/rows=11]
table[row sep=crcr, point meta=\thisrow{c}] {%
x	y	z	c\\
0	10	1	1\\
0	20	1	1\\
0	30	1	1\\
0	40	1	1\\
0	50	1	1\\
20	10	1	1\\
20	20	1	1\\
20	30	1	1\\
20	40	1	1\\
20	50	1	1\\
40	10	1	1\\
40	20	1	1\\
40	30	1	1\\
40	40	1	1\\
40	50	1	1\\
60	10	1	1\\
60	20	1	1\\
60	30	1	1\\
60	40	1	1\\
60	50	1	1\\
80	10	1	1\\
80	20	1	1\\
80	30	1	1\\
80	40	0.997368421052632	0.997368421052632\\
80	50	0.991489361702128	0.991489361702128\\
100	10	1	1\\
100	20	1	1\\
100	30	1	1\\
100	40	0.992105263157895	0.992105263157895\\
100	50	0.880851063829787	0.880851063829787\\
120	10	1	1\\
120	20	0.994736842105263	0.994736842105263\\
120	30	0.971428571428571	0.971428571428571\\
120	40	0.910526315789474	0.910526315789474\\
120	50	0.517021276595745	0.517021276595745\\
140	10	0.9	0.9\\
140	20	0.745	0.745\\
140	30	0.560714285714286	0.560714285714286\\
140	40	0.382051282051282	0.382051282051282\\
140	50	0.322448979591837	0.322448979591837\\
160	10	0.644444444444444	0.644444444444444\\
160	20	0.647368421052632	0.647368421052632\\
160	30	0.478571428571429	0.478571428571429\\
160	40	0.297368421052632	0.297368421052632\\
160	50	0.295744680851064	0.295744680851064\\
180	10	0.31	0.31\\
180	20	0.465	0.465\\
180	30	0.203333333333333	0.203333333333333\\
180	40	0.1575	0.1575\\
180	50	0.142857142857143	0.142857142857143\\
200	10	0.5	0.5\\
200	20	0.214285714285714	0.214285714285714\\
200	30	0.21	0.21\\
200	40	0.1425	0.1425\\
200	50	0.0901960784313726	0.0901960784313726\\
};
\end{axis}
\end{tikzpicture}%

%% file: dwa1.tex
%
%
\begin{tikzpicture}

\begin{axis}[%
width=0.75\fwidth,
height=\fheight,
at={(0\fwidth,0\fheight)},
scale only axis,
point meta min=0.00193045413622771,
point meta max=0.982835989757635,
axis on top,
xmin=-0.00122448979591837,
xmax=0.0477551020408163,
xlabel={time},
ymin=-0.526315789473684,
ymax=0.526315789473684,
ylabel={space},
axis background/.style={fill=white},
colormap={mymap}{[1pt] rgb(0pt)=(1,1,1); rgb(99pt)=(0,0,0)},
colorbar
]
\addplot [forget plot] graphics [xmin=-0.00122448979591837,xmax=0.0477551020408163,ymin=-0.526315789473684,ymax=0.526315789473684] {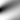};
\end{axis}
\end{tikzpicture}%

%% file: dwa2.tex
%
%
\begin{tikzpicture}

\begin{axis}[%
width=0.75\fwidth,
height=\fheight,
at={(0\fwidth,0\fheight)},
scale only axis,
point meta min=0.00193045413622771,
point meta max=0.982835989757635,
axis on top,
xmin=-0.00122448979591837,
xmax=0.0477551020408163,
xlabel={time},
ymin=-0.526315789473684,
ymax=0.526315789473684,
ylabel={space},
axis background/.style={fill=white},
colormap={mymap}{[1pt] rgb(0pt)=(1,1,1); rgb(99pt)=(0,0,0)},
colorbar
]
\addplot [forget plot] graphics [xmin=-0.00122448979591837,xmax=0.0477551020408163,ymin=-0.526315789473684,ymax=0.526315789473684] {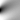};
\end{axis}
\end{tikzpicture}%

%% file: dwa3.tex
%
%
\begin{tikzpicture}

\begin{axis}[%
width=0.75\fwidth,
height=\fheight,
at={(0\fwidth,0\fheight)},
scale only axis,
point meta min=0.00193045413622771,
point meta max=0.982835989757635,
axis on top,
xmin=-0.00122448979591837,
xmax=0.0477551020408163,
xlabel={time},
ymin=-0.526315789473684,
ymax=0.526315789473684,
ylabel={space},
axis background/.style={fill=white},
colormap={mymap}{[1pt] rgb(0pt)=(1,1,1); rgb(99pt)=(0,0,0)},
colorbar
]
\addplot [forget plot] graphics [xmin=-0.00122448979591837,xmax=0.0477551020408163,ymin=-0.526315789473684,ymax=0.526315789473684] {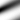};
\end{axis}
\end{tikzpicture}%

%% file: rwint.tex
%
%
\begin{tikzpicture}

\begin{axis}[%
width=0.951\fwidth,
height=\fheight,
at={(0\fwidth,0\fheight)},
scale only axis,
xmin=0,
xmax=35,
xlabel={time},
ymode=log,
ymin=0,
ymax=1,
yminorticks=true,
ylabel={absolute error},
axis background/.style={fill=white},
legend style={at={(0.03,0.97)},anchor=north west,legend cell align=left,align=left,draw=white!15!black}
]
\addplot [color=black,only marks,mark=x,mark options={solid},forget plot]
  table[row sep=crcr]{%
1	0\\
2	0\\
3	0\\
4	0\\
5	5.32907051820075e-15\\
6	0\\
7	0\\
8	0\\
9	0\\
10	0\\
11	0\\
12	0\\
13	0\\
14	0\\
15	0\\
16	0\\
17	0\\
18	0\\
19	0\\
20	0\\
21	0\\
22	0\\
23	0\\
24	0\\
25	0\\
26	0\\
27	0\\
28	0\\
29	0\\
30	0\\
31	0\\
32	0\\
33	0\\
34	0\\
};
\addplot [color=black,only marks,mark=o,mark options={solid},forget plot]
  table[row sep=crcr]{%
1	0\\
2	0\\
3	0\\
4	0\\
5	4.9960036108132e-15\\
6	0\\
7	0\\
8	0\\
9	0\\
10	0\\
11	0\\
12	0\\
13	0\\
14	0\\
15	0\\
16	0\\
17	0\\
18	0\\
19	0\\
20	0\\
21	0\\
22	0\\
23	0\\
24	0\\
25	0\\
26	0\\
27	0\\
28	0\\
29	0\\
30	0\\
31	0\\
32	0\\
33	0\\
34	0\\
};
\addplot [color=black,only marks,mark=x,mark options={solid},forget plot]
  table[row sep=crcr]{%
1	0\\
2	0\\
3	0\\
4	0\\
5	0\\
6	0\\
7	0\\
8	0\\
9	0\\
10	0\\
11	0\\
12	0\\
13	0\\
14	0\\
15	0\\
16	0\\
17	0\\
18	0\\
19	0\\
20	0\\
21	0\\
22	0\\
23	0\\
24	0\\
25	0\\
26	0\\
27	0\\
28	0\\
29	0\\
30	0\\
31	0\\
32	0\\
33	0\\
34	0\\
};
\addplot [color=black,only marks,mark=o,mark options={solid},forget plot]
  table[row sep=crcr]{%
1	0\\
2	0\\
3	0\\
4	0\\
5	0\\
6	0\\
7	0\\
8	0\\
9	0\\
10	0\\
11	0\\
12	0\\
13	0\\
14	0\\
15	0\\
16	0\\
17	0\\
18	0\\
19	0\\
20	0\\
21	0\\
22	0\\
23	0\\
24	0\\
25	0\\
26	0\\
27	0\\
28	0\\
29	0\\
30	0\\
31	0\\
32	0\\
33	0\\
34	0\\
};
\addplot [color=black,only marks,mark=x,mark options={solid},forget plot]
  table[row sep=crcr]{%
1	0\\
2	0\\
3	0\\
4	0\\
5	0\\
6	0\\
7	0\\
8	0\\
9	0\\
10	0\\
11	0\\
12	0\\
13	0\\
14	0\\
15	0\\
16	0\\
17	0\\
18	0\\
19	0\\
20	0\\
21	0\\
22	0\\
23	0\\
24	0\\
25	0\\
26	0\\
27	0\\
28	0\\
29	0\\
30	0\\
31	0\\
32	0\\
33	0\\
34	0\\
};
\addplot [color=black,only marks,mark=o,mark options={solid},forget plot]
  table[row sep=crcr]{%
1	0\\
2	0\\
3	0\\
4	0\\
5	0\\
6	0\\
7	0\\
8	0\\
9	0\\
10	0\\
11	0\\
12	0\\
13	0\\
14	0\\
15	0\\
16	0\\
17	0\\
18	0\\
19	0\\
20	0\\
21	0\\
22	0\\
23	0\\
24	0\\
25	0\\
26	0\\
27	0\\
28	0\\
29	0\\
30	0\\
31	0\\
32	0\\
33	0\\
34	0\\
};
\addplot [color=black,only marks,mark=x,mark options={solid},forget plot]
  table[row sep=crcr]{%
1	0\\
2	0\\
3	0\\
4	0\\
5	0\\
6	0\\
7	0\\
8	0\\
9	0\\
10	0\\
11	0\\
12	0\\
13	0\\
14	0\\
15	0\\
16	0\\
17	0\\
18	0\\
19	0\\
20	0\\
21	0\\
22	0\\
23	0\\
24	0\\
25	0\\
26	0\\
27	0\\
28	0\\
29	0\\
30	0\\
31	0\\
32	0\\
33	0\\
34	0\\
};
\addplot [color=black,only marks,mark=o,mark options={solid},forget plot]
  table[row sep=crcr]{%
1	0\\
2	0\\
3	0\\
4	0\\
5	0\\
6	0\\
7	0\\
8	0\\
9	0\\
10	0\\
11	0\\
12	0\\
13	0\\
14	0\\
15	0\\
16	0\\
17	0\\
18	0\\
19	0\\
20	0\\
21	0\\
22	0\\
23	0\\
24	0\\
25	0\\
26	0\\
27	0\\
28	0\\
29	0\\
30	0\\
31	0\\
32	0\\
33	0\\
34	0\\
};
\addplot [color=black,only marks,mark=x,mark options={solid},forget plot]
  table[row sep=crcr]{%
1	0\\
2	0\\
3	0\\
4	0\\
5	0\\
6	0\\
7	0\\
8	1.4210854715202e-14\\
9	0\\
10	0\\
11	0\\
12	0\\
13	0\\
14	0\\
15	0\\
16	0\\
17	0\\
18	0\\
19	0\\
20	0\\
21	0\\
22	0\\
23	0\\
24	0\\
25	0\\
26	0\\
27	0\\
28	0\\
29	0\\
30	0\\
31	0\\
32	0\\
33	0\\
34	0\\
};
\addplot [color=black,only marks,mark=o,mark options={solid},forget plot]
  table[row sep=crcr]{%
1	0\\
2	0\\
3	0\\
4	0\\
5	0\\
6	0\\
7	0\\
8	0\\
9	0\\
10	0\\
11	0\\
12	0\\
13	0\\
14	0\\
15	0\\
16	0\\
17	0\\
18	0\\
19	0\\
20	0\\
21	0\\
22	0\\
23	0\\
24	0\\
25	0\\
26	0\\
27	0\\
28	0\\
29	0\\
30	0\\
31	0\\
32	0\\
33	0\\
34	0\\
};
\addplot [color=black,only marks,mark=x,mark options={solid},forget plot]
  table[row sep=crcr]{%
1	0\\
2	0\\
3	0\\
4	0\\
5	0\\
6	0\\
7	0\\
8	0\\
9	0\\
10	0\\
11	0\\
12	0\\
13	0\\
14	0\\
15	2.8421709430404e-14\\
16	0\\
17	0\\
18	0\\
19	0\\
20	0\\
21	0\\
22	0\\
23	0\\
24	0\\
25	0\\
26	0\\
27	0\\
28	0\\
29	0\\
30	0\\
31	0\\
32	0\\
33	0\\
34	0\\
};
\addplot [color=black,only marks,mark=o,mark options={solid},forget plot]
  table[row sep=crcr]{%
1	0\\
2	0\\
3	0\\
4	0\\
5	0\\
6	0\\
7	0\\
8	0\\
9	0\\
10	0\\
11	0\\
12	0\\
13	0\\
14	0\\
15	2.88657986402541e-16\\
16	0\\
17	0\\
18	0\\
19	0\\
20	0\\
21	0\\
22	0\\
23	0\\
24	0\\
25	0\\
26	0\\
27	0\\
28	0\\
29	0\\
30	0\\
31	0\\
32	0\\
33	0\\
34	0\\
};
\addplot [color=black,only marks,mark=x,mark options={solid},forget plot]
  table[row sep=crcr]{%
1	0\\
2	0\\
3	0\\
4	0\\
5	0\\
6	0\\
7	0\\
8	0\\
9	0\\
10	0\\
11	0\\
12	0\\
13	0\\
14	0\\
15	0\\
16	0\\
17	0\\
18	0\\
19	0\\
20	0\\
21	0\\
22	0\\
23	0\\
24	0\\
25	0\\
26	0\\
27	0\\
28	0\\
29	0\\
30	0\\
31	0\\
32	0\\
33	0\\
34	0.182458376514916\\
};
\addplot [color=black,only marks,mark=o,mark options={solid},forget plot]
  table[row sep=crcr]{%
1	0\\
2	0\\
3	0\\
4	0\\
5	0\\
6	0\\
7	0\\
8	0\\
9	0\\
10	0\\
11	0\\
12	0\\
13	0\\
14	0\\
15	0\\
16	0\\
17	0\\
18	0\\
19	0\\
20	0\\
21	0\\
22	0\\
23	0\\
24	0\\
25	0\\
26	0\\
27	0\\
28	0\\
29	0\\
30	0\\
31	0\\
32	0\\
33	0\\
34	0.182313230047528\\
};
\addplot [color=black,only marks,mark=x,mark options={solid},forget plot]
  table[row sep=crcr]{%
1	0\\
2	0\\
3	0\\
4	0\\
5	0\\
6	0\\
7	0\\
8	0\\
9	0\\
10	0\\
11	0\\
12	0\\
13	0\\
14	0\\
15	0\\
16	0\\
17	0\\
18	0\\
19	0\\
20	0\\
21	0\\
22	0\\
23	0\\
24	0\\
25	0\\
26	0\\
27	0\\
28	0\\
29	0\\
30	0\\
31	0\\
32	0\\
33	0\\
34	0.0821725104042628\\
};
\addplot [color=black,only marks,mark=o,mark options={solid},forget plot]
  table[row sep=crcr]{%
1	0\\
2	0\\
3	0\\
4	0\\
5	0\\
6	0\\
7	0\\
8	0\\
9	0\\
10	0\\
11	0\\
12	0\\
13	0\\
14	0\\
15	0\\
16	0\\
17	0\\
18	0\\
19	0\\
20	0\\
21	0\\
22	0\\
23	0\\
24	0\\
25	0\\
26	0\\
27	0\\
28	0\\
29	2.22044604925031e-16\\
30	0\\
31	0\\
32	0\\
33	0\\
34	0.0813979073885998\\
};
\addplot [color=black,only marks,mark=x,mark options={solid},forget plot]
  table[row sep=crcr]{%
1	0\\
2	0\\
3	0\\
4	0\\
5	0\\
6	0\\
7	0\\
8	0\\
9	0\\
10	0\\
11	0\\
12	0\\
13	0\\
14	0\\
15	0\\
16	0\\
17	0\\
18	0\\
19	0\\
20	0\\
21	0\\
22	0\\
23	0\\
24	0\\
25	0\\
26	0\\
27	0\\
28	0\\
29	0\\
30	0\\
31	0\\
32	0\\
33	0\\
34	0.120203435629151\\
};
\addplot [color=black,only marks,mark=o,mark options={solid},forget plot]
  table[row sep=crcr]{%
1	0\\
2	0\\
3	0\\
4	0\\
5	0\\
6	0\\
7	0\\
8	0\\
9	0\\
10	0\\
11	0\\
12	0\\
13	0\\
14	0\\
15	0\\
16	0\\
17	0\\
18	0\\
19	0\\
20	0\\
21	0\\
22	0\\
23	0\\
24	0\\
25	0\\
26	0\\
27	0\\
28	0\\
29	0\\
30	0\\
31	0\\
32	0\\
33	0\\
34	0.119068163725178\\
};
\addplot [color=black,only marks,mark=x,mark options={solid},forget plot]
  table[row sep=crcr]{%
1	0\\
2	0\\
3	0\\
4	0\\
5	0\\
6	0\\
7	0\\
8	0\\
9	0\\
10	0\\
11	0\\
12	0\\
13	0\\
14	0\\
15	0\\
16	0\\
17	0\\
18	0\\
19	0\\
20	0\\
21	0\\
22	0\\
23	0\\
24	0\\
25	0\\
26	0\\
27	0\\
28	0\\
29	0\\
30	0\\
31	0\\
32	0\\
33	0\\
34	0.0289311009410085\\
};
\addplot [color=black,only marks,mark=o,mark options={solid},forget plot]
  table[row sep=crcr]{%
1	0\\
2	0\\
3	0\\
4	0\\
5	0\\
6	0\\
7	0\\
8	0\\
9	0\\
10	0\\
11	0\\
12	0\\
13	0\\
14	0\\
15	0\\
16	0\\
17	0\\
18	0\\
19	0\\
20	0\\
21	0\\
22	0\\
23	0\\
24	0\\
25	0\\
26	0\\
27	0\\
28	0\\
29	0\\
30	0\\
31	0\\
32	0\\
33	0\\
34	0.0285617590025327\\
};
\addplot [color=black,only marks,mark=x,mark options={solid},forget plot]
  table[row sep=crcr]{%
1	0\\
2	0\\
3	5.6843418860808e-14\\
4	0\\
5	0\\
6	0\\
7	5.6843418860808e-14\\
8	0\\
9	5.6843418860808e-14\\
10	0\\
11	0\\
12	2.8421709430404e-14\\
13	2.8421709430404e-14\\
14	0\\
15	2.8421709430404e-14\\
16	2.8421709430404e-14\\
17	0\\
18	2.8421709430404e-14\\
19	5.42854650120717e-12\\
20	1.13686837721616e-12\\
21	3.5242919693701e-12\\
22	1.13686837721616e-12\\
23	6.25277607468888e-13\\
24	1.89857018995099e-11\\
25	2.55511167779332e-11\\
26	7.56017470848747e-12\\
27	1.41824330057716e-11\\
28	9.18021214602049e-12\\
29	4.20641299569979e-12\\
30	1.39834810397588e-11\\
31	1.18518528324785e-11\\
32	3.75166564481333e-12\\
33	4.14956957683899e-12\\
34	0.0509159018376124\\
};
\addplot [color=black,only marks,mark=o,mark options={solid},forget plot]
  table[row sep=crcr]{%
1	1.77635683940025e-15\\
2	1.77635683940025e-15\\
3	3.19744231092045e-14\\
4	3.5527136788005e-15\\
5	3.5527136788005e-15\\
6	3.5527136788005e-15\\
7	1.4210854715202e-14\\
8	5.32907051820075e-15\\
9	2.66453525910038e-15\\
10	4.44089209850063e-15\\
11	0\\
12	1.77635683940025e-15\\
13	2.66453525910038e-15\\
14	5.32907051820075e-15\\
15	5.32907051820075e-15\\
16	1.06581410364015e-14\\
17	1.77635683940025e-15\\
18	0\\
19	5.43032285804657e-12\\
20	1.16440190822686e-12\\
21	3.49675843835939e-12\\
22	1.17683640610267e-12\\
23	6.16395823271887e-13\\
24	1.90372162478525e-11\\
25	2.54996024295906e-11\\
26	7.55662199480867e-12\\
27	1.4201972931005e-11\\
28	9.229061959104e-12\\
29	4.19841938992249e-12\\
30	1.40163436412877e-11\\
31	1.18731691145513e-11\\
32	3.74189568219663e-12\\
33	4.16999768049209e-12\\
34	0.051153890562869\\
};
\addplot [color=black,only marks,mark=x,mark options={solid},forget plot]
  table[row sep=crcr]{%
1	0\\
2	0\\
3	0\\
4	0\\
5	0\\
6	0\\
7	0\\
8	0\\
9	0\\
10	0\\
11	0\\
12	0\\
13	0\\
14	0\\
15	0\\
16	0\\
17	0\\
18	0\\
19	0\\
20	0\\
21	0\\
22	0\\
23	0\\
24	0\\
25	0\\
26	0\\
27	0\\
28	0\\
29	0\\
30	0\\
31	0\\
32	0\\
33	0\\
34	0\\
};
\addplot [color=black,only marks,mark=o,mark options={solid},forget plot]
  table[row sep=crcr]{%
1	0\\
2	0\\
3	0\\
4	0\\
5	0\\
6	0\\
7	0\\
8	0\\
9	0\\
10	0\\
11	0\\
12	0\\
13	0\\
14	0\\
15	0\\
16	0\\
17	0\\
18	0\\
19	0\\
20	0\\
21	0\\
22	0\\
23	0\\
24	0\\
25	0\\
26	0\\
27	0\\
28	0\\
29	0\\
30	0\\
31	0\\
32	0\\
33	0\\
34	0\\
};
\addplot [color=black,only marks,mark=x,mark options={solid},forget plot]
  table[row sep=crcr]{%
1	0\\
2	0\\
3	0\\
4	0\\
5	0\\
6	0\\
7	0\\
8	0\\
9	0\\
10	0\\
11	0\\
12	0\\
13	0\\
14	0\\
15	0\\
16	0\\
17	0\\
18	0\\
19	0\\
20	0\\
21	0\\
22	0\\
23	0\\
24	0\\
25	0\\
26	0\\
27	0\\
28	0\\
29	0\\
30	0\\
31	0\\
32	0\\
33	0\\
34	0\\
};
\addplot [color=black,only marks,mark=o,mark options={solid},forget plot]
  table[row sep=crcr]{%
1	0\\
2	0\\
3	0\\
4	0\\
5	0\\
6	0\\
7	0\\
8	0\\
9	0\\
10	0\\
11	0\\
12	0\\
13	0\\
14	0\\
15	0\\
16	0\\
17	0\\
18	0\\
19	0\\
20	0\\
21	0\\
22	0\\
23	0\\
24	0\\
25	0\\
26	0\\
27	0\\
28	0\\
29	0\\
30	0\\
31	0\\
32	0\\
33	0\\
34	0\\
};
\addplot [color=black,only marks,mark=x,mark options={solid},forget plot]
  table[row sep=crcr]{%
1	0\\
2	0\\
3	0\\
4	0\\
5	0\\
6	0\\
7	0\\
8	0\\
9	1.4210854715202e-14\\
10	0\\
11	0\\
12	0\\
13	0\\
14	0\\
15	0\\
16	0\\
17	0\\
18	0\\
19	0\\
20	0\\
21	0\\
22	0\\
23	0\\
24	0\\
25	0\\
26	0\\
27	0\\
28	0\\
29	0\\
30	0\\
31	0\\
32	0\\
33	0\\
34	0\\
};
\addplot [color=black,only marks,mark=o,mark options={solid},forget plot]
  table[row sep=crcr]{%
1	0\\
2	1.77635683940025e-15\\
3	0\\
4	0\\
5	2.66453525910038e-15\\
6	0\\
7	0\\
8	2.88657986402541e-15\\
9	3.02983429085235e-15\\
10	2.3265604394925e-16\\
11	0\\
12	0\\
13	0\\
14	0\\
15	0\\
16	0\\
17	0\\
18	0\\
19	0\\
20	0\\
21	0\\
22	0\\
23	0\\
24	0\\
25	0\\
26	0\\
27	0\\
28	0\\
29	0\\
30	0\\
31	0\\
32	0\\
33	0\\
34	0\\
};
\addplot [color=black,only marks,mark=x,mark options={solid},forget plot]
  table[row sep=crcr]{%
1	0\\
2	0\\
3	0\\
4	0\\
5	0\\
6	0\\
7	0\\
8	0\\
9	0\\
10	0\\
11	0\\
12	0\\
13	0\\
14	0\\
15	0\\
16	0\\
17	0\\
18	0\\
19	0\\
20	0\\
21	0\\
22	0\\
23	0\\
24	0\\
25	0\\
26	0\\
27	0\\
28	0\\
29	0\\
30	0\\
31	0\\
32	0\\
33	0\\
34	0\\
};
\addplot [color=black,only marks,mark=o,mark options={solid},forget plot]
  table[row sep=crcr]{%
1	0\\
2	3.5527136788005e-15\\
3	0\\
4	0\\
5	0\\
6	0\\
7	0\\
8	0\\
9	0\\
10	0\\
11	0\\
12	0\\
13	0\\
14	0\\
15	0\\
16	0\\
17	0\\
18	0\\
19	0\\
20	0\\
21	0\\
22	0\\
23	0\\
24	0\\
25	0\\
26	0\\
27	0\\
28	0\\
29	0\\
30	0\\
31	0\\
32	0\\
33	0\\
34	0\\
};
\addplot [color=black,only marks,mark=x,mark options={solid},forget plot]
  table[row sep=crcr]{%
1	0\\
2	0\\
3	0\\
4	0\\
5	0\\
6	0\\
7	0\\
8	0\\
9	0\\
10	0\\
11	0\\
12	0\\
13	0\\
14	0\\
15	0\\
16	0\\
17	0\\
18	0\\
19	0\\
20	0\\
21	0\\
22	0\\
23	0\\
24	0\\
25	0\\
26	0\\
27	0\\
28	0\\
29	0\\
30	0\\
31	0\\
32	0\\
33	0\\
34	0\\
};
\addplot [color=black,only marks,mark=o,mark options={solid},forget plot]
  table[row sep=crcr]{%
1	0\\
2	0\\
3	0\\
4	0\\
5	0\\
6	0\\
7	0\\
8	0\\
9	0\\
10	0\\
11	0\\
12	0\\
13	0\\
14	0\\
15	0\\
16	0\\
17	0\\
18	0\\
19	0\\
20	0\\
21	0\\
22	0\\
23	0\\
24	0\\
25	0\\
26	0\\
27	0\\
28	0\\
29	0\\
30	0\\
31	0\\
32	0\\
33	0\\
34	0\\
};
\addplot [color=black,only marks,mark=x,mark options={solid},forget plot]
  table[row sep=crcr]{%
1	0\\
2	0\\
3	0\\
4	0\\
5	0\\
6	0\\
7	0\\
8	0\\
9	0\\
10	0\\
11	0\\
12	0\\
13	0\\
14	0\\
15	0\\
16	0\\
17	0\\
18	0\\
19	0\\
20	0\\
21	0\\
22	0\\
23	0\\
24	0\\
25	0\\
26	0\\
27	0\\
28	0\\
29	0\\
30	0\\
31	0\\
32	0\\
33	0\\
34	0\\
};
\addplot [color=black,only marks,mark=o,mark options={solid},forget plot]
  table[row sep=crcr]{%
1	0\\
2	0\\
3	0\\
4	0\\
5	0\\
6	0\\
7	0\\
8	0\\
9	0\\
10	0\\
11	0\\
12	0\\
13	0\\
14	0\\
15	0\\
16	0\\
17	0\\
18	0\\
19	0\\
20	0\\
21	0\\
22	0\\
23	0\\
24	0\\
25	0\\
26	0\\
27	0\\
28	0\\
29	0\\
30	0\\
31	0\\
32	0\\
33	0\\
34	0\\
};
\addplot [color=black,only marks,mark=x,mark options={solid},forget plot]
  table[row sep=crcr]{%
1	0\\
2	0\\
3	0\\
4	0\\
5	0\\
6	0\\
7	0\\
8	0\\
9	0\\
10	0\\
11	0\\
12	5.6843418860808e-14\\
13	2.8421709430404e-14\\
14	2.8421709430404e-14\\
15	0\\
16	3.12638803734444e-13\\
17	8.5265128291212e-14\\
18	8.5265128291212e-14\\
19	1.4210854715202e-13\\
20	1.13686837721616e-13\\
21	1.53477230924182e-12\\
22	1.79056769411545e-12\\
23	2.27373675443232e-13\\
24	4.83169060316868e-13\\
25	1.98951966012828e-13\\
26	2.27373675443232e-13\\
27	4.83169060316868e-13\\
28	1.98951966012828e-13\\
29	2.27373675443232e-13\\
30	4.83169060316868e-13\\
31	1.98951966012828e-13\\
32	2.27373675443232e-13\\
33	4.83169060316868e-13\\
34	1.98951966012828e-13\\
};
\addplot [color=black,only marks,mark=o,mark options={solid},forget plot]
  table[row sep=crcr]{%
1	0\\
2	0\\
3	0\\
4	0\\
5	0\\
6	0\\
7	0\\
8	0\\
9	0\\
10	0\\
11	0\\
12	5.86197757002083e-14\\
13	8.88178419700125e-15\\
14	1.59872115546023e-14\\
15	6.21724893790088e-15\\
16	2.93987056920741e-13\\
17	8.61533067109121e-14\\
18	8.61533067109121e-14\\
19	1.4477308241112e-13\\
20	9.40358901857508e-14\\
21	1.50418566491339e-12\\
22	1.81837878088231e-12\\
23	2.00950367457153e-13\\
24	4.85306239639272e-13\\
25	2.001454557643e-13\\
26	2.00950367457153e-13\\
27	4.85306239639272e-13\\
28	2.001454557643e-13\\
29	2.00950367457153e-13\\
30	4.85306239639272e-13\\
31	2.001454557643e-13\\
32	2.00950367457153e-13\\
33	4.85306239639272e-13\\
34	2.001454557643e-13\\
};
\addplot [color=black,only marks,mark=x,mark options={solid},forget plot]
  table[row sep=crcr]{%
1	0\\
2	0\\
3	0\\
4	0\\
5	0\\
6	0\\
7	0\\
8	0\\
9	0\\
10	0\\
11	0\\
12	0\\
13	0\\
14	0\\
15	0\\
16	0\\
17	0\\
18	0\\
19	0\\
20	0\\
21	0\\
22	0\\
23	0\\
24	0\\
25	0\\
26	0\\
27	0\\
28	0\\
29	0\\
30	0\\
31	0\\
32	0\\
33	0\\
34	0.928521490968649\\
};
\addplot [color=black,only marks,mark=o,mark options={solid},forget plot]
  table[row sep=crcr]{%
1	0\\
2	0\\
3	0\\
4	0\\
5	0\\
6	0\\
7	0\\
8	0\\
9	0\\
10	0\\
11	0\\
12	0\\
13	0\\
14	0\\
15	0\\
16	0\\
17	0\\
18	0\\
19	0\\
20	0\\
21	0\\
22	0\\
23	0\\
24	0\\
25	0\\
26	0\\
27	0\\
28	0\\
29	0\\
30	8.88178419700125e-16\\
31	0\\
32	0\\
33	0\\
34	0.921785483248332\\
};
\addplot [color=black,only marks,mark=x,mark options={solid},forget plot]
  table[row sep=crcr]{%
1	0\\
2	0\\
3	0\\
4	0\\
5	0\\
6	0\\
7	0\\
8	0\\
9	0\\
10	0\\
11	0\\
12	0\\
13	0\\
14	0\\
15	0\\
16	0\\
17	0\\
18	0\\
19	0\\
20	0\\
21	0\\
22	0\\
23	0\\
24	0\\
25	0\\
26	0\\
27	0\\
28	0\\
29	0\\
30	0\\
31	0\\
32	0\\
33	0\\
34	0.0389508165238226\\
};
\addplot [color=black,only marks,mark=o,mark options={solid},forget plot]
  table[row sep=crcr]{%
1	0\\
2	0\\
3	0\\
4	0\\
5	0\\
6	0\\
7	0\\
8	0\\
9	0\\
10	0\\
11	0\\
12	0\\
13	0\\
14	0\\
15	0\\
16	0\\
17	0\\
18	0\\
19	0\\
20	1.77635683940025e-15\\
21	0\\
22	0\\
23	0\\
24	0\\
25	0\\
26	0\\
27	0\\
28	0\\
29	0\\
30	0\\
31	0\\
32	0\\
33	0\\
34	0.0403608206308874\\
};
\addplot [color=black,only marks,mark=x,mark options={solid},forget plot]
  table[row sep=crcr]{%
1	0\\
2	0\\
3	0\\
4	2.8421709430404e-14\\
5	0\\
6	0\\
7	0\\
8	0\\
9	0\\
10	0\\
11	0\\
12	0\\
13	0\\
14	0\\
15	0\\
16	0\\
17	0\\
18	0\\
19	0\\
20	0\\
21	0\\
22	0\\
23	0\\
24	0\\
25	0\\
26	0\\
27	0\\
28	0\\
29	0\\
30	0\\
31	0\\
32	0\\
33	0\\
34	0.109818970150286\\
};
\addplot [color=black,only marks,mark=o,mark options={solid},forget plot]
  table[row sep=crcr]{%
1	0\\
2	0\\
3	0\\
4	3.5527136788005e-15\\
5	0\\
6	0\\
7	3.5527136788005e-15\\
8	0\\
9	0\\
10	0\\
11	0\\
12	0\\
13	0\\
14	0\\
15	0\\
16	0\\
17	0\\
18	0\\
19	0\\
20	0\\
21	0\\
22	0\\
23	0\\
24	0\\
25	0\\
26	0\\
27	0\\
28	0\\
29	0\\
30	0\\
31	0\\
32	0\\
33	0\\
34	0.109844251027011\\
};
\addplot [color=black,only marks,mark=x,mark options={solid},forget plot]
  table[row sep=crcr]{%
1	0\\
2	0\\
3	0\\
4	0\\
5	0\\
6	0\\
7	0\\
8	0\\
9	0\\
10	0\\
11	0\\
12	0\\
13	0\\
14	0\\
15	0\\
16	0\\
17	0\\
18	0\\
19	0\\
20	0\\
21	0\\
22	0\\
23	0\\
24	0\\
25	0\\
26	0\\
27	0\\
28	0\\
29	0\\
30	0\\
31	0\\
32	0\\
33	0\\
34	0.120987527544685\\
};
\addplot [color=black,only marks,mark=o,mark options={solid},forget plot]
  table[row sep=crcr]{%
1	0\\
2	0\\
3	0\\
4	0\\
5	0\\
6	0\\
7	0\\
8	0\\
9	0\\
10	0\\
11	0\\
12	0\\
13	0\\
14	0\\
15	0\\
16	0\\
17	0\\
18	0\\
19	0\\
20	0\\
21	0\\
22	0\\
23	0\\
24	0\\
25	0\\
26	0\\
27	0\\
28	0\\
29	0\\
30	0\\
31	0\\
32	0\\
33	0\\
34	0.120277496761094\\
};
\addplot [color=black,only marks,mark=x,mark options={solid},forget plot]
  table[row sep=crcr]{%
1	0\\
2	0\\
3	0\\
4	0\\
5	0\\
6	3.5527136788005e-15\\
7	0\\
8	0\\
9	0\\
10	0\\
11	0\\
12	0\\
13	0\\
14	0\\
15	0\\
16	0\\
17	0\\
18	0\\
19	0\\
20	0\\
21	0\\
22	0\\
23	0\\
24	0\\
25	0\\
26	0\\
27	0\\
28	0\\
29	0\\
30	0\\
31	0\\
32	0\\
33	0\\
34	0\\
};
\addplot [color=black,only marks,mark=o,mark options={solid},forget plot]
  table[row sep=crcr]{%
1	0\\
2	0\\
3	0\\
4	0\\
5	0\\
6	1.77635683940025e-15\\
7	0\\
8	0\\
9	0\\
10	0\\
11	0\\
12	0\\
13	0\\
14	0\\
15	0\\
16	0\\
17	0\\
18	0\\
19	0\\
20	0\\
21	0\\
22	0\\
23	0\\
24	0\\
25	0\\
26	0\\
27	0\\
28	0\\
29	0\\
30	0\\
31	0\\
32	0\\
33	0\\
34	0\\
};
\addplot [color=black,only marks,mark=x,mark options={solid},forget plot]
  table[row sep=crcr]{%
1	0\\
2	0\\
3	0\\
4	0\\
5	0\\
6	0\\
7	0\\
8	0\\
9	0\\
10	0\\
11	0\\
12	0\\
13	7.105427357601e-15\\
14	7.105427357601e-15\\
15	0\\
16	0\\
17	0\\
18	0\\
19	0\\
20	0\\
21	0\\
22	0\\
23	0\\
24	0\\
25	0\\
26	0\\
27	0\\
28	0\\
29	0\\
30	0\\
31	0\\
32	0\\
33	0\\
34	0\\
};
\addplot [color=black,only marks,mark=o,mark options={solid},forget plot]
  table[row sep=crcr]{%
1	0\\
2	0\\
3	0\\
4	0\\
5	0\\
6	0\\
7	0\\
8	0\\
9	0\\
10	0\\
11	0\\
12	0\\
13	2.88657986402541e-15\\
14	3.70405807053396e-15\\
15	0\\
16	0\\
17	0\\
18	0\\
19	0\\
20	0\\
21	0\\
22	0\\
23	0\\
24	0\\
25	0\\
26	0\\
27	0\\
28	0\\
29	0\\
30	0\\
31	0\\
32	0\\
33	0\\
34	0\\
};
\addplot [color=black,only marks,mark=x,mark options={solid},forget plot]
  table[row sep=crcr]{%
1	0\\
2	0\\
3	0\\
4	0\\
5	0\\
6	0\\
7	0\\
8	0\\
9	0\\
10	0\\
11	0\\
12	0\\
13	0\\
14	0\\
15	0\\
16	0\\
17	0\\
18	0\\
19	0\\
20	0\\
21	0\\
22	0\\
23	0\\
24	0\\
25	0\\
26	0\\
27	0\\
28	0\\
29	0\\
30	0\\
31	0\\
32	0\\
33	0\\
34	0\\
};
\addplot [color=black,only marks,mark=o,mark options={solid},forget plot]
  table[row sep=crcr]{%
1	0\\
2	0\\
3	0\\
4	0\\
5	0\\
6	0\\
7	0\\
8	0\\
9	0\\
10	0\\
11	0\\
12	0\\
13	0\\
14	0\\
15	0\\
16	0\\
17	0\\
18	0\\
19	0\\
20	0\\
21	0\\
22	0\\
23	0\\
24	0\\
25	0\\
26	0\\
27	0\\
28	0\\
29	0\\
30	0\\
31	0\\
32	0\\
33	0\\
34	0\\
};
\addplot [color=black,only marks,mark=x,mark options={solid},forget plot]
  table[row sep=crcr]{%
1	0\\
2	0\\
3	0\\
4	0\\
5	0\\
6	1.4210854715202e-14\\
7	0\\
8	0\\
9	0\\
10	0\\
11	0\\
12	0\\
13	1.4210854715202e-14\\
14	0\\
15	0\\
16	0\\
17	0\\
18	0\\
19	0\\
20	0\\
21	0\\
22	0\\
23	0\\
24	0\\
25	0\\
26	0\\
27	0\\
28	0\\
29	0\\
30	0\\
31	0\\
32	0\\
33	0\\
34	0\\
};
\addplot [color=black,only marks,mark=o,mark options={solid},forget plot]
  table[row sep=crcr]{%
1	0\\
2	0\\
3	0\\
4	0\\
5	0\\
6	1.77635683940025e-15\\
7	1.77635683940025e-15\\
8	0\\
9	0\\
10	0\\
11	0\\
12	0\\
13	0\\
14	0\\
15	0\\
16	0\\
17	0\\
18	0\\
19	0\\
20	0\\
21	0\\
22	0\\
23	0\\
24	0\\
25	0\\
26	0\\
27	0\\
28	0\\
29	0\\
30	0\\
31	0\\
32	0\\
33	0\\
34	0\\
};
\addplot [color=black,only marks,mark=x,mark options={solid},forget plot]
  table[row sep=crcr]{%
1	0\\
2	0\\
3	0\\
4	0\\
5	1.4210854715202e-14\\
6	0\\
7	0\\
8	0\\
9	0\\
10	0\\
11	0\\
12	0\\
13	0\\
14	0\\
15	0\\
16	0\\
17	0\\
18	0\\
19	0\\
20	0\\
21	0\\
22	0\\
23	0\\
24	0\\
25	0\\
26	0\\
27	0\\
28	0\\
29	0\\
30	0\\
31	0\\
32	0\\
33	0\\
34	0\\
};
\addplot [color=black,only marks,mark=o,mark options={solid},forget plot]
  table[row sep=crcr]{%
1	0\\
2	0\\
3	0\\
4	0\\
5	0\\
6	0\\
7	0\\
8	0\\
9	0\\
10	0\\
11	0\\
12	0\\
13	0\\
14	0\\
15	0\\
16	0\\
17	0\\
18	5.55111512312578e-17\\
19	0\\
20	0\\
21	0\\
22	0\\
23	0\\
24	0\\
25	0\\
26	0\\
27	0\\
28	0\\
29	0\\
30	0\\
31	0\\
32	0\\
33	0\\
34	0\\
};
\addplot [color=black,only marks,mark=x,mark options={solid},forget plot]
  table[row sep=crcr]{%
1	0\\
2	0\\
3	0\\
4	0\\
5	0\\
6	0\\
7	0\\
8	0\\
9	2.8421709430404e-14\\
10	0\\
11	0\\
12	0\\
13	0\\
14	0\\
15	0\\
16	0\\
17	0\\
18	0\\
19	0\\
20	0\\
21	0\\
22	0\\
23	0\\
24	0\\
25	0\\
26	0\\
27	0\\
28	0\\
29	0\\
30	0\\
31	0\\
32	0\\
33	0\\
34	0\\
};
\addplot [color=black,only marks,mark=o,mark options={solid},forget plot]
  table[row sep=crcr]{%
1	0\\
2	0\\
3	0\\
4	0\\
5	0\\
6	0\\
7	0\\
8	0\\
9	3.5527136788005e-15\\
10	0\\
11	0\\
12	0\\
13	0\\
14	0\\
15	0\\
16	0\\
17	0\\
18	0\\
19	0\\
20	0\\
21	0\\
22	0\\
23	4.44089209850063e-16\\
24	0\\
25	0\\
26	0\\
27	0\\
28	0\\
29	0\\
30	0\\
31	0\\
32	0\\
33	0\\
34	0\\
};
\addplot [color=black,only marks,mark=x,mark options={solid},forget plot]
  table[row sep=crcr]{%
1	0\\
2	0\\
3	0\\
4	0\\
5	0\\
6	0\\
7	2.8421709430404e-14\\
8	0\\
9	0\\
10	0\\
11	0\\
12	0\\
13	0\\
14	0\\
15	0\\
16	0\\
17	0\\
18	0\\
19	0\\
20	0\\
21	0\\
22	0\\
23	0\\
24	0\\
25	0\\
26	0\\
27	0\\
28	0\\
29	0\\
30	0\\
31	0\\
32	0\\
33	0\\
34	0.32906411787107\\
};
\addplot [color=black,only marks,mark=o,mark options={solid},forget plot]
  table[row sep=crcr]{%
1	0\\
2	0\\
3	0\\
4	0\\
5	0\\
6	0\\
7	0\\
8	0\\
9	0\\
10	0\\
11	0\\
12	0\\
13	0\\
14	0\\
15	0\\
16	0\\
17	0\\
18	0\\
19	0\\
20	0\\
21	0\\
22	0\\
23	0\\
24	0\\
25	0\\
26	0\\
27	0\\
28	0\\
29	0\\
30	0\\
31	0\\
32	0\\
33	2.22044604925031e-16\\
34	0.329473332868722\\
};
\addplot [color=black,only marks,mark=x,mark options={solid},forget plot]
  table[row sep=crcr]{%
1	0\\
2	0\\
3	0\\
4	0\\
5	0\\
6	0\\
7	2.8421709430404e-14\\
8	0\\
9	0\\
10	0\\
11	2.8421709430404e-14\\
12	0\\
13	0\\
14	0\\
15	0\\
16	0\\
17	0\\
18	0\\
19	0\\
20	0\\
21	0\\
22	0\\
23	0\\
24	0\\
25	0\\
26	0\\
27	0\\
28	0\\
29	0\\
30	0\\
31	0\\
32	0\\
33	0\\
34	0.157459442348539\\
};
\addplot [color=black,only marks,mark=o,mark options={solid},forget plot]
  table[row sep=crcr]{%
1	0\\
2	0\\
3	0\\
4	0\\
5	0\\
6	0\\
7	0\\
8	3.5527136788005e-15\\
9	3.5527136788005e-15\\
10	0\\
11	0\\
12	0\\
13	3.5527136788005e-15\\
14	0\\
15	0\\
16	3.5527136788005e-15\\
17	3.5527136788005e-15\\
18	0\\
19	0\\
20	0\\
21	0\\
22	0\\
23	0\\
24	0\\
25	0\\
26	0\\
27	0\\
28	0\\
29	0\\
30	0\\
31	0\\
32	0\\
33	0\\
34	0.157886709726952\\
};
\addplot [color=black,only marks,mark=x,mark options={solid},forget plot]
  table[row sep=crcr]{%
1	0\\
2	0\\
3	0\\
4	0\\
5	0\\
6	0\\
7	0\\
8	0\\
9	0\\
10	0\\
11	0\\
12	0\\
13	0\\
14	0\\
15	0\\
16	0\\
17	0\\
18	0\\
19	0\\
20	0\\
21	0\\
22	0\\
23	0\\
24	0\\
25	0\\
26	0\\
27	0\\
28	0\\
29	0\\
30	0\\
31	0\\
32	0\\
33	0\\
34	0.19166223657038\\
};
\addplot [color=black,only marks,mark=o,mark options={solid},forget plot]
  table[row sep=crcr]{%
1	0\\
2	0\\
3	0\\
4	0\\
5	0\\
6	0\\
7	0\\
8	0\\
9	0\\
10	0\\
11	0\\
12	0\\
13	0\\
14	0\\
15	0\\
16	0\\
17	0\\
18	0\\
19	0\\
20	0\\
21	0\\
22	0\\
23	0\\
24	0\\
25	0\\
26	0\\
27	0\\
28	0\\
29	0\\
30	0\\
31	0\\
32	0\\
33	0\\
34	0.190883295230938\\
};
\addplot [color=black,only marks,mark=x,mark options={solid},forget plot]
  table[row sep=crcr]{%
1	0\\
2	0\\
3	0\\
4	0\\
5	0\\
6	0\\
7	0\\
8	0\\
9	0\\
10	0\\
11	0\\
12	0\\
13	0\\
14	0\\
15	0\\
16	0\\
17	0\\
18	0\\
19	0\\
20	0\\
21	0\\
22	0\\
23	0\\
24	0\\
25	0\\
26	0\\
27	0\\
28	0\\
29	0\\
30	0\\
31	0\\
32	0\\
33	0\\
34	0.0593677455804311\\
};
\addplot [color=black,only marks,mark=o,mark options={solid},forget plot]
  table[row sep=crcr]{%
1	0\\
2	0\\
3	0\\
4	0\\
5	0\\
6	0\\
7	0\\
8	0\\
9	0\\
10	0\\
11	0\\
12	0\\
13	0\\
14	0\\
15	0\\
16	0\\
17	0\\
18	0\\
19	0\\
20	0\\
21	0\\
22	0\\
23	0\\
24	0\\
25	0\\
26	0\\
27	0\\
28	0\\
29	0\\
30	0\\
31	0\\
32	0\\
33	0\\
34	0.0600106077966913\\
};
\addplot [color=black,only marks,mark=x,mark options={solid},forget plot]
  table[row sep=crcr]{%
1	0\\
2	0\\
3	0\\
4	0\\
5	0\\
6	0\\
7	0\\
8	0\\
9	2.8421709430404e-14\\
10	0\\
11	0\\
12	0\\
13	0\\
14	0\\
15	0\\
16	0\\
17	0\\
18	0\\
19	0\\
20	0\\
21	0\\
22	0\\
23	0\\
24	0\\
25	0\\
26	0\\
27	0\\
28	0\\
29	0\\
30	0\\
31	0\\
32	0\\
33	0\\
34	0.0289786283910303\\
};
\addplot [color=black,only marks,mark=o,mark options={solid},forget plot]
  table[row sep=crcr]{%
1	0\\
2	0\\
3	0\\
4	0\\
5	0\\
6	0\\
7	0\\
8	0\\
9	7.105427357601e-15\\
10	0\\
11	0\\
12	0\\
13	0\\
14	0\\
15	0\\
16	0\\
17	0\\
18	0\\
19	0\\
20	0\\
21	0\\
22	0\\
23	0\\
24	0\\
25	0\\
26	0\\
27	0\\
28	0\\
29	0\\
30	0\\
31	0\\
32	0\\
33	0\\
34	0.0282811900743241\\
};
\addplot [color=black,only marks,mark=x,mark options={solid},forget plot]
  table[row sep=crcr]{%
1	0\\
2	0\\
3	0\\
4	0\\
5	0\\
6	0\\
7	0\\
8	0\\
9	0\\
10	0\\
11	0\\
12	0\\
13	0\\
14	0\\
15	0\\
16	0\\
17	0\\
18	0\\
19	0\\
20	0\\
21	0\\
22	0\\
23	0\\
24	0\\
25	0\\
26	0\\
27	0\\
28	0\\
29	0\\
30	0\\
31	0\\
32	0\\
33	0\\
34	0\\
};
\addplot [color=black,only marks,mark=o,mark options={solid},forget plot]
  table[row sep=crcr]{%
1	0\\
2	0\\
3	0\\
4	0\\
5	0\\
6	0\\
7	0\\
8	0\\
9	0\\
10	0\\
11	0\\
12	0\\
13	0\\
14	0\\
15	0\\
16	0\\
17	0\\
18	0\\
19	0\\
20	0\\
21	0\\
22	0\\
23	0\\
24	0\\
25	0\\
26	0\\
27	0\\
28	0\\
29	0\\
30	0\\
31	0\\
32	0\\
33	0\\
34	0\\
};
\addplot [color=black,only marks,mark=x,mark options={solid},forget plot]
  table[row sep=crcr]{%
1	0\\
2	0\\
3	0\\
4	0\\
5	0\\
6	0\\
7	0\\
8	0\\
9	0\\
10	0\\
11	0\\
12	0\\
13	0\\
14	0\\
15	0\\
16	0\\
17	0\\
18	0\\
19	0\\
20	0\\
21	0\\
22	0\\
23	0\\
24	0\\
25	0\\
26	0\\
27	0\\
28	0\\
29	0\\
30	0\\
31	0\\
32	0\\
33	0\\
34	0\\
};
\addplot [color=black,only marks,mark=o,mark options={solid},forget plot]
  table[row sep=crcr]{%
1	0\\
2	0\\
3	0\\
4	0\\
5	0\\
6	0\\
7	0\\
8	0\\
9	0\\
10	0\\
11	0\\
12	0\\
13	0\\
14	0\\
15	0\\
16	0\\
17	0\\
18	0\\
19	0\\
20	0\\
21	0\\
22	0\\
23	0\\
24	0\\
25	0\\
26	0\\
27	0\\
28	0\\
29	0\\
30	0\\
31	0\\
32	0\\
33	0\\
34	0\\
};
\addplot [color=black,only marks,mark=x,mark options={solid},forget plot]
  table[row sep=crcr]{%
1	0\\
2	0\\
3	0\\
4	0\\
5	0\\
6	0\\
7	0\\
8	0\\
9	0\\
10	0\\
11	0\\
12	0\\
13	0\\
14	0\\
15	0\\
16	0\\
17	0\\
18	0\\
19	0\\
20	0\\
21	0\\
22	0\\
23	0\\
24	0\\
25	0\\
26	0\\
27	0\\
28	0\\
29	0\\
30	0\\
31	0\\
32	0\\
33	0\\
34	0\\
};
\addplot [color=black,only marks,mark=o,mark options={solid},forget plot]
  table[row sep=crcr]{%
1	0\\
2	0\\
3	0\\
4	0\\
5	0\\
6	0\\
7	0\\
8	0\\
9	0\\
10	0\\
11	0\\
12	0\\
13	0\\
14	0\\
15	0\\
16	0\\
17	0\\
18	0\\
19	0\\
20	0\\
21	0\\
22	0\\
23	0\\
24	0\\
25	0\\
26	0\\
27	0\\
28	0\\
29	0\\
30	0\\
31	0\\
32	0\\
33	0\\
34	0\\
};
\addplot [color=black,only marks,mark=x,mark options={solid},forget plot]
  table[row sep=crcr]{%
1	0\\
2	0\\
3	0\\
4	0\\
5	0\\
6	0\\
7	0\\
8	0\\
9	0\\
10	0\\
11	0\\
12	0\\
13	0\\
14	0\\
15	0\\
16	0\\
17	0\\
18	0\\
19	0\\
20	0\\
21	0\\
22	0\\
23	0\\
24	0\\
25	0\\
26	0\\
27	0\\
28	0\\
29	0\\
30	0\\
31	0\\
32	0\\
33	0\\
34	0\\
};
\addplot [color=black,only marks,mark=o,mark options={solid},forget plot]
  table[row sep=crcr]{%
1	0\\
2	0\\
3	0\\
4	0\\
5	0\\
6	0\\
7	0\\
8	0\\
9	0\\
10	0\\
11	0\\
12	0\\
13	0\\
14	0\\
15	0\\
16	0\\
17	0\\
18	0\\
19	0\\
20	0\\
21	0\\
22	0\\
23	0\\
24	0\\
25	0\\
26	0\\
27	0\\
28	0\\
29	0\\
30	0\\
31	0\\
32	0\\
33	0\\
34	0\\
};
\addplot [color=black,only marks,mark=x,mark options={solid},forget plot]
  table[row sep=crcr]{%
1	0\\
2	0\\
3	0\\
4	1.4210854715202e-14\\
5	0\\
6	0\\
7	0\\
8	0\\
9	0\\
10	0\\
11	0\\
12	0\\
13	0\\
14	0\\
15	0\\
16	0\\
17	0\\
18	0\\
19	0\\
20	0\\
21	0\\
22	0\\
23	0\\
24	0\\
25	0\\
26	0\\
27	0\\
28	0\\
29	0\\
30	0\\
31	0\\
32	0\\
33	0\\
34	0\\
};
\addplot [color=black,only marks,mark=o,mark options={solid},forget plot]
  table[row sep=crcr]{%
1	0\\
2	0\\
3	0\\
4	3.5527136788005e-15\\
5	0\\
6	0\\
7	0\\
8	0\\
9	0\\
10	0\\
11	0\\
12	0\\
13	0\\
14	0\\
15	0\\
16	0\\
17	0\\
18	0\\
19	0\\
20	0\\
21	0\\
22	0\\
23	0\\
24	0\\
25	0\\
26	0\\
27	0\\
28	0\\
29	0\\
30	0\\
31	0\\
32	0\\
33	0\\
34	0\\
};
\addplot [color=black,only marks,mark=x,mark options={solid},forget plot]
  table[row sep=crcr]{%
1	0\\
2	0\\
3	0\\
4	0\\
5	0\\
6	0\\
7	0\\
8	0\\
9	0\\
10	0\\
11	0\\
12	0\\
13	0\\
14	0\\
15	0\\
16	0\\
17	0\\
18	0\\
19	0\\
20	0\\
21	0\\
22	0\\
23	0\\
24	0\\
25	0\\
26	0\\
27	0\\
28	0\\
29	0\\
30	0\\
31	0\\
32	2.8421709430404e-14\\
33	0\\
34	0.0218869419043415\\
};
\addplot [color=black,only marks,mark=o,mark options={solid},forget plot]
  table[row sep=crcr]{%
1	0\\
2	0\\
3	0\\
4	0\\
5	0\\
6	0\\
7	0\\
8	0\\
9	0\\
10	0\\
11	0\\
12	0\\
13	0\\
14	0\\
15	0\\
16	0\\
17	0\\
18	0\\
19	0\\
20	0\\
21	0\\
22	0\\
23	0\\
24	0\\
25	0\\
26	0\\
27	0\\
28	0\\
29	0\\
30	0\\
31	0\\
32	1.33226762955019e-15\\
33	0\\
34	0.0218013893864816\\
};
\addplot [color=black,only marks,mark=x,mark options={solid},forget plot]
  table[row sep=crcr]{%
1	0\\
2	0\\
3	0\\
4	0\\
5	0\\
6	0\\
7	0\\
8	0\\
9	0\\
10	0\\
11	0\\
12	0\\
13	0\\
14	0\\
15	0\\
16	0\\
17	0\\
18	0\\
19	0\\
20	0\\
21	0\\
22	0\\
23	0\\
24	0\\
25	0\\
26	0\\
27	0\\
28	0\\
29	0\\
30	0\\
31	0\\
32	0\\
33	0\\
33.5311648252153	1.1\\
};
\addplot [color=black,only marks,mark=o,mark options={solid},forget plot]
  table[row sep=crcr]{%
1	0\\
2	0\\
3	0\\
4	0\\
5	0\\
6	0\\
7	0\\
8	0\\
9	0\\
10	0\\
11	0\\
12	0\\
13	0\\
14	0\\
15	0\\
16	0\\
17	0\\
18	0\\
19	0\\
20	0\\
21	0\\
22	0\\
23	0\\
24	0\\
25	0\\
26	0\\
27	0\\
28	0\\
29	0\\
30	0\\
31	0\\
32	0\\
33	0\\
33.5329259382852	1.1\\
};
\addplot [color=black,only marks,mark=x,mark options={solid},forget plot]
  table[row sep=crcr]{%
1	0\\
2	0\\
3	0\\
4	0\\
5	0\\
6	0\\
7	0\\
8	0\\
9	0\\
10	0\\
11	0\\
12	0\\
13	0\\
14	0\\
15	0\\
16	0\\
17	0\\
18	0\\
19	0\\
20	0\\
21	0\\
22	0\\
23	0\\
24	0\\
25	0\\
26	0\\
27	0\\
28	0\\
29	0\\
30	0\\
31	0\\
32	0\\
33	0\\
34	0.182532095416832\\
};
\addplot [color=black,only marks,mark=o,mark options={solid},forget plot]
  table[row sep=crcr]{%
1	0\\
2	0\\
3	0\\
4	0\\
5	0\\
6	0\\
7	0\\
8	0\\
9	0\\
10	0\\
11	0\\
12	0\\
13	0\\
14	0\\
15	0\\
16	0\\
17	0\\
18	0\\
19	0\\
20	0\\
21	0\\
22	0\\
23	0\\
24	0\\
25	0\\
26	0\\
27	0\\
28	0\\
29	0\\
30	0\\
31	0\\
32	0\\
33	0\\
34	0.185217818661354\\
};
\addplot [color=black,only marks,mark=x,mark options={solid},forget plot]
  table[row sep=crcr]{%
1	0\\
2	0\\
3	0\\
4	0\\
5	0\\
6	0\\
7	0\\
8	0\\
9	0\\
10	0\\
11	0\\
12	0\\
13	0\\
14	0\\
15	0\\
16	0\\
17	0\\
18	2.8421709430404e-14\\
19	0\\
20	0\\
21	0\\
22	0\\
23	0\\
24	0\\
25	0\\
26	0\\
27	0\\
28	0\\
29	0\\
30	0\\
31	0\\
32	0\\
33	0\\
34	0.0625710466990768\\
};
\addplot [color=black,only marks,mark=o,mark options={solid},forget plot]
  table[row sep=crcr]{%
1	0\\
2	0\\
3	0\\
4	0\\
5	0\\
6	0\\
7	0\\
8	0\\
9	0\\
10	0\\
11	0\\
12	0\\
13	0\\
14	0\\
15	0\\
16	0\\
17	0\\
18	7.105427357601e-15\\
19	0\\
20	0\\
21	0\\
22	0\\
23	0\\
24	0\\
25	0\\
26	0\\
27	0\\
28	0\\
29	0\\
30	0\\
31	0\\
32	0\\
33	0\\
34	0.0631260551914146\\
};
\addplot [color=black,only marks,mark=x,mark options={solid},forget plot]
  table[row sep=crcr]{%
1	0\\
2	0\\
3	0\\
4	0\\
5	0\\
6	0\\
7	0\\
8	0\\
9	0\\
10	2.8421709430404e-14\\
11	0\\
12	0\\
13	0\\
14	0\\
15	0\\
16	0\\
17	0\\
18	0\\
19	0\\
20	0\\
21	0\\
22	0\\
23	0\\
24	0\\
25	0\\
26	0\\
27	0\\
28	0\\
29	0\\
30	0\\
31	0\\
32	0\\
33	0\\
34	0.00978977153206984\\
};
\addplot [color=black,only marks,mark=o,mark options={solid},forget plot]
  table[row sep=crcr]{%
1	0\\
2	0\\
3	0\\
4	0\\
5	0\\
6	0\\
7	0\\
8	0\\
9	0\\
10	0\\
11	0\\
12	0\\
13	0\\
14	0\\
15	0\\
16	0\\
17	0\\
18	0\\
19	0\\
20	0\\
21	0\\
22	0\\
23	0\\
24	0\\
25	0\\
26	0\\
27	0\\
28	0\\
29	0\\
30	0\\
31	0\\
32	0\\
33	0\\
34	0.00844902170942419\\
};
\addplot [color=black,only marks,mark=x,mark options={solid},forget plot]
  table[row sep=crcr]{%
1	0\\
2	0\\
3	0\\
4	0\\
5	0\\
6	0\\
7	0\\
8	0\\
9	0\\
10	0\\
11	0\\
12	0\\
13	0\\
14	0\\
15	0\\
16	0\\
17	0\\
18	2.8421709430404e-14\\
19	0\\
20	0\\
21	0\\
22	0\\
23	0\\
24	0\\
25	0\\
26	0\\
27	0\\
28	0\\
29	0\\
30	0\\
31	0\\
32	0\\
33	0\\
34	0.213980981135478\\
};
\addplot [color=black,only marks,mark=o,mark options={solid},forget plot]
  table[row sep=crcr]{%
1	0\\
2	0\\
3	0\\
4	0\\
5	0\\
6	7.105427357601e-15\\
7	0\\
8	0\\
9	0\\
10	0\\
11	0\\
12	0\\
13	7.105427357601e-15\\
14	0\\
15	0\\
16	0\\
17	0\\
18	0\\
19	0\\
20	0\\
21	0\\
22	0\\
23	0\\
24	0\\
25	0\\
26	0\\
27	0\\
28	0\\
29	0\\
30	0\\
31	0\\
32	0\\
33	0\\
34	0.213867256569053\\
};
\addplot [color=black,only marks,mark=x,mark options={solid},forget plot]
  table[row sep=crcr]{%
1	0\\
2	0\\
3	0\\
4	0\\
5	0\\
6	0\\
7	0\\
8	0\\
9	0\\
10	0\\
11	0\\
12	0\\
13	0\\
14	0\\
15	7.105427357601e-15\\
16	7.105427357601e-15\\
17	0\\
18	0\\
19	0\\
20	0\\
21	0\\
22	0\\
23	0\\
24	0\\
25	0\\
26	0\\
27	0\\
28	0\\
29	0\\
30	0\\
31	0\\
32	0\\
33	0\\
34	0\\
};
\addplot [color=black,only marks,mark=o,mark options={solid},forget plot]
  table[row sep=crcr]{%
1	0\\
2	0\\
3	0\\
4	0\\
5	0\\
6	0\\
7	0\\
8	0\\
9	0\\
10	0\\
11	0\\
12	0\\
13	0\\
14	0\\
15	2.88657986402541e-15\\
16	3.70405807053396e-15\\
17	0\\
18	0\\
19	0\\
20	0\\
21	0\\
22	0\\
23	0\\
24	0\\
25	0\\
26	0\\
27	0\\
28	0\\
29	0\\
30	0\\
31	0\\
32	0\\
33	0\\
34	0\\
};
\addplot [color=black,only marks,mark=x,mark options={solid},forget plot]
  table[row sep=crcr]{%
1	0\\
2	0\\
3	0\\
4	0\\
5	0\\
6	0\\
7	0\\
8	0\\
9	0\\
10	0\\
11	0\\
12	0\\
13	0\\
14	0\\
15	0\\
16	0\\
17	0\\
18	0\\
19	0\\
20	0\\
21	0\\
22	0\\
23	0\\
24	0\\
25	0\\
26	0\\
27	0\\
28	0\\
29	0\\
30	0\\
31	0\\
32	0\\
33	0\\
34	0\\
};
\addplot [color=black,only marks,mark=o,mark options={solid},forget plot]
  table[row sep=crcr]{%
1	0\\
2	0\\
3	0\\
4	0\\
5	0\\
6	0\\
7	0\\
8	0\\
9	0\\
10	0\\
11	0\\
12	0\\
13	0\\
14	0\\
15	0\\
16	0\\
17	0\\
18	0\\
19	0\\
20	0\\
21	0\\
22	0\\
23	0\\
24	0\\
25	0\\
26	0\\
27	0\\
28	0\\
29	0\\
30	0\\
31	0\\
32	0\\
33	0\\
34	0\\
};
\addplot [color=black,only marks,mark=x,mark options={solid},forget plot]
  table[row sep=crcr]{%
1	0\\
2	0\\
3	0\\
4	0\\
5	0\\
6	0\\
7	0\\
8	0\\
9	0\\
10	0\\
11	0\\
12	0\\
13	0\\
14	1.4210854715202e-14\\
15	1.4210854715202e-14\\
16	0\\
17	0\\
18	0\\
19	0\\
20	0\\
21	0\\
22	0\\
23	0\\
24	0\\
25	0\\
26	0\\
27	0\\
28	0\\
29	0\\
30	0\\
31	0\\
32	0\\
33	0\\
34	0\\
};
\addplot [color=black,only marks,mark=o,mark options={solid},forget plot]
  table[row sep=crcr]{%
1	0\\
2	0\\
3	0\\
4	0\\
5	0\\
6	0\\
7	0\\
8	0\\
9	0\\
10	0\\
11	0\\
12	0\\
13	0\\
14	4.44089209850063e-16\\
15	0\\
16	0\\
17	0\\
18	0\\
19	0\\
20	0\\
21	0\\
22	0\\
23	0\\
24	0\\
25	0\\
26	0\\
27	0\\
28	0\\
29	0\\
30	0\\
31	0\\
32	0\\
33	0\\
34	0\\
};
\addplot [color=black,only marks,mark=x,mark options={solid},forget plot]
  table[row sep=crcr]{%
1	0\\
2	0\\
3	0\\
4	0\\
5	0\\
6	0\\
7	0\\
8	0\\
9	0\\
10	0\\
11	0\\
12	0\\
13	1.4210854715202e-14\\
14	0\\
15	0\\
16	0\\
17	0\\
18	1.98951966012828e-13\\
19	6.11066752753686e-13\\
20	2.68585154117318e-12\\
21	3.87956333725015e-12\\
22	8.05755462351954e-12\\
23	3.12638803734444e-13\\
24	1.58593138621654e-11\\
25	5.11590769747272e-13\\
26	1.32160948851379e-12\\
27	6.25277607468888e-13\\
28	4.54747350886464e-13\\
29	9.09494701772928e-13\\
30	1.32160948851379e-12\\
31	1.02318153949454e-12\\
32	1.80477854883065e-12\\
33	9.80548975348938e-13\\
34	1.64845914696343e-12\\
};
\addplot [color=black,only marks,mark=o,mark options={solid},forget plot]
  table[row sep=crcr]{%
1	0\\
2	0\\
3	0\\
4	0\\
5	0\\
6	0\\
7	0\\
8	0\\
9	0\\
10	0\\
11	0\\
12	0\\
13	1.77635683940025e-15\\
14	0\\
15	1.77635683940025e-15\\
16	0\\
17	2.66453525910038e-15\\
18	1.94511073914327e-13\\
19	6.14619466432487e-13\\
20	2.67741384618603e-12\\
21	3.88067356027477e-12\\
22	8.04784017205407e-12\\
23	3.093636458118e-13\\
24	1.58506541551977e-11\\
25	5.08565466560323e-13\\
26	1.31397601231465e-12\\
27	6.22065055663914e-13\\
28	4.25623459138468e-13\\
29	9.18910134857923e-13\\
30	1.30087990587962e-12\\
31	1.02804435514984e-12\\
32	1.79416658159908e-12\\
33	9.84390667033072e-13\\
34	1.63264793556705e-12\\
};
\addplot [color=black,only marks,mark=x,mark options={solid},forget plot]
  table[row sep=crcr]{%
1	0\\
2	0\\
3	0\\
4	0\\
5	0\\
6	0\\
7	0\\
8	0\\
9	0\\
10	0\\
11	0\\
12	0\\
13	0\\
14	0\\
15	0\\
16	0\\
17	0\\
18	0\\
19	0\\
20	0\\
21	0\\
22	0\\
23	0\\
24	2.8421709430404e-14\\
25	0\\
26	0\\
27	0\\
28	0\\
29	0\\
30	0\\
31	0\\
32	0\\
33	0\\
34	0.0644308263940445\\
};
\addplot [color=black,only marks,mark=o,mark options={solid},forget plot]
  table[row sep=crcr]{%
1	0\\
2	0\\
3	0\\
4	0\\
5	0\\
6	0\\
7	0\\
8	0\\
9	0\\
10	0\\
11	0\\
12	0\\
13	0\\
14	0\\
15	0\\
16	0\\
17	0\\
18	0\\
19	0\\
20	3.5527136788005e-15\\
21	0\\
22	0\\
23	5.32907051820075e-15\\
24	1.77635683940025e-15\\
25	0\\
26	0\\
27	2.66453525910038e-15\\
28	8.88178419700125e-16\\
29	0\\
30	0\\
31	0\\
32	0\\
33	3.33066907387547e-16\\
34	0.0644308263940657\\
};
\addplot [color=black,only marks,mark=x,mark options={solid},forget plot]
  table[row sep=crcr]{%
1	0\\
2	0\\
3	0\\
4	0\\
5	0\\
6	0\\
7	0\\
8	0\\
9	0\\
10	0\\
11	0\\
12	0\\
13	0\\
14	0\\
15	0\\
16	0\\
17	0\\
18	0\\
19	0\\
20	0\\
21	0\\
22	0\\
23	0\\
24	0\\
25	0\\
26	0\\
27	0\\
28	0\\
29	0\\
30	0\\
31	0\\
32	0\\
33	0\\
33.5671964130969	1.1\\
};
\addplot [color=black,only marks,mark=o,mark options={solid},forget plot]
  table[row sep=crcr]{%
1	0\\
2	0\\
3	0\\
4	0\\
5	0\\
6	0\\
7	0\\
8	0\\
9	0\\
10	0\\
11	0\\
12	0\\
13	0\\
14	0\\
15	0\\
16	0\\
17	0\\
18	0\\
19	0\\
20	0\\
21	0\\
22	0\\
23	0\\
24	0\\
25	0\\
26	0\\
27	0\\
28	0\\
29	0\\
30	0\\
31	0\\
32	0\\
33	0\\
33.5683171362862	1.1\\
};
\addplot [color=black,only marks,mark=x,mark options={solid},forget plot]
  table[row sep=crcr]{%
1	0\\
2	0\\
3	0\\
4	0\\
5	0\\
6	0\\
7	0\\
8	0\\
9	0\\
10	0\\
11	0\\
12	0\\
13	0\\
14	0\\
15	0\\
16	0\\
17	0\\
18	0\\
19	0\\
20	0\\
21	0\\
22	0\\
23	0\\
24	0\\
25	0\\
26	0\\
27	0\\
28	0\\
29	0\\
30	0\\
31	0\\
32	0\\
33	0\\
34	0.143541722904303\\
};
\addplot [color=black,only marks,mark=o,mark options={solid},forget plot]
  table[row sep=crcr]{%
1	0\\
2	0\\
3	0\\
4	0\\
5	0\\
6	0\\
7	0\\
8	0\\
9	0\\
10	0\\
11	0\\
12	0\\
13	0\\
14	0\\
15	0\\
16	0\\
17	0\\
18	0\\
19	0\\
20	0\\
21	0\\
22	0\\
23	0\\
24	0\\
25	0\\
26	0\\
27	0\\
28	0\\
29	0\\
30	0\\
31	0\\
32	0\\
33	0\\
34	0.150574597590762\\
};
\addplot [color=black,only marks,mark=x,mark options={solid},forget plot]
  table[row sep=crcr]{%
1	0\\
2	0\\
3	0\\
4	0\\
5	0\\
6	0\\
7	0\\
8	0\\
9	0\\
10	0\\
11	0\\
12	0\\
13	0\\
14	0\\
15	0\\
16	0\\
17	0\\
18	0\\
19	0\\
20	0\\
21	0\\
22	0\\
23	0\\
24	0\\
25	0\\
26	0\\
27	0\\
28	0\\
29	0\\
30	0\\
31	0\\
32	0\\
33	0\\
34	0.0903520382216243\\
};
\addplot [color=black,only marks,mark=o,mark options={solid},forget plot]
  table[row sep=crcr]{%
1	0\\
2	0\\
3	0\\
4	0\\
5	0\\
6	0\\
7	0\\
8	0\\
9	0\\
10	0\\
11	0\\
12	0\\
13	0\\
14	0\\
15	0\\
16	0\\
17	0\\
18	0\\
19	0\\
20	0\\
21	0\\
22	0\\
23	0\\
24	0\\
25	0\\
26	0\\
27	0\\
28	0\\
29	0\\
30	0\\
31	0\\
32	0\\
33	0\\
34	0.0902599467923864\\
};
\addplot [color=black,only marks,mark=x,mark options={solid},forget plot]
  table[row sep=crcr]{%
1	0\\
2	0\\
3	0\\
4	0\\
5	0\\
6	0\\
7	0\\
8	0\\
9	0\\
10	0\\
11	0\\
12	0\\
13	0\\
14	0\\
15	0\\
16	0\\
17	0\\
18	0\\
19	0\\
20	0\\
21	0\\
22	0\\
23	0\\
24	0\\
25	0\\
26	0\\
27	0\\
28	0\\
29	0\\
30	0\\
31	0\\
32	0\\
33	0\\
34	0.0678494832989713\\
};
\addplot [color=black,only marks,mark=o,mark options={solid},forget plot]
  table[row sep=crcr]{%
1	0\\
2	0\\
3	0\\
4	0\\
5	0\\
6	0\\
7	0\\
8	0\\
9	0\\
10	0\\
11	0\\
12	0\\
13	0\\
14	0\\
15	0\\
16	0\\
17	0\\
18	0\\
19	0\\
20	0\\
21	0\\
22	0\\
23	0\\
24	0\\
25	0\\
26	0\\
27	0\\
28	0\\
29	0\\
30	0\\
31	0\\
32	0\\
33	0\\
34	0.0676059157039859\\
};
\addplot [color=black,only marks,mark=x,mark options={solid},forget plot]
  table[row sep=crcr]{%
1	0\\
2	0\\
3	0\\
4	0\\
5	0\\
6	0\\
7	0\\
8	0\\
9	0\\
10	0\\
11	0\\
12	0\\
13	0\\
14	0\\
15	0\\
16	0\\
17	0\\
18	0\\
19	0\\
20	0\\
21	0\\
22	0\\
23	0\\
24	0\\
25	0\\
26	0\\
27	0\\
28	0\\
29	0\\
30	0\\
31	0\\
32	0\\
33	0\\
34	0.0185675896367741\\
};
\addplot [color=black,only marks,mark=o,mark options={solid},forget plot]
  table[row sep=crcr]{%
1	0\\
2	0\\
3	0\\
4	0\\
5	0\\
6	0\\
7	0\\
8	0\\
9	0\\
10	0\\
11	0\\
12	0\\
13	0\\
14	0\\
15	0\\
16	0\\
17	0\\
18	0\\
19	0\\
20	0\\
21	0\\
22	0\\
23	0\\
24	0\\
25	0\\
26	0\\
27	0\\
28	0\\
29	0\\
30	0\\
31	0\\
32	0\\
33	0\\
34	0.0183928387691736\\
};
\addplot [color=black,only marks,mark=x,mark options={solid}]
  table[row sep=crcr]{%
1	2.27373675443232e-13\\
2	5.38591393706156e-11\\
3	3.11786152451532e-11\\
4	0\\
5	0\\
6	0\\
7	0\\
8	3.09796632791404e-12\\
9	9.32232069317251e-12\\
10	6.22435436525848e-12\\
11	0\\
12	6.28119778411929e-12\\
13	3.09796632791404e-12\\
14	1.79056769411545e-12\\
15	5.62749846721999e-12\\
16	0\\
17	6.08224581810646e-12\\
18	0\\
19	0\\
20	0\\
21	0\\
22	6.08224581810646e-12\\
23	6.13908923696727e-12\\
24	6.08224581810646e-12\\
25	0\\
26	0\\
27	0\\
28	2.8421709430404e-14\\
29	5.6843418860808e-14\\
30	0\\
31	1.45234935189364e-11\\
32	3.86535248253495e-12\\
33	0.0529145433971507\\
};
\addlegendentry{Error to \#platform};

\addplot [color=black,only marks,mark=o,mark options={solid}]
  table[row sep=crcr]{%
1	2.20268248085631e-13\\
2	5.39088773621188e-11\\
3	3.11786152451532e-11\\
4	2.1316282072803e-14\\
5	1.4210854715202e-14\\
6	2.1316282072803e-14\\
7	2.1316282072803e-14\\
8	3.14059889205964e-12\\
9	9.35784782996052e-12\\
10	6.25277607468888e-12\\
11	1.4210854715202e-14\\
12	6.23856521997368e-12\\
13	3.12638803734444e-12\\
14	1.77635683940025e-12\\
15	5.6488147492928e-12\\
16	0\\
17	6.08935124546406e-12\\
18	1.4210854715202e-14\\
19	0\\
20	0\\
21	7.105427357601e-15\\
22	6.11066752753686e-12\\
23	6.11066752753686e-12\\
24	6.09645667282166e-12\\
25	0\\
26	0\\
27	7.105427357601e-15\\
28	0\\
29	0\\
30	0\\
31	1.45448098010093e-11\\
32	3.83693077310454e-12\\
33	0.0526546979171059\\
};
\addlegendentry{Error to \#train};

\end{axis}
\end{tikzpicture}%